\documentclass[1p]{elsarticle}
\usepackage[utf8]{inputenc}
\usepackage[T1]{fontenc}
\usepackage[american]{babel}
\usepackage{amsmath,amssymb,amsthm,mathtools,bm,array,booktabs,enumitem}
\IfFileExists{lmodern.sty}{\usepackage{lmodern}}{}
\IfFileExists{microtype.sty}{\usepackage{microtype}}{}
\usepackage[hidelinks]{hyperref}
\hypersetup{
 pdftitle={Forward--Backward McKean--Vlasov Systems with Jumps and a Monotone Source with a Control-Shifted Interface},
 pdfauthor={Kayembe Tshiswaka Tcheick; Mabela Matendo Rostin; Kabengele Mpunga Yannick; Bosonga Bofeki Jean-Pierre},
 pdfkeywords={McKean--Vlasov, BSDEs with jumps, maximal monotone operators, Yosida approximation, occupation time, control-shifted interface}}
\allowdisplaybreaks
\numberwithin{equation}{section}
\theoremstyle{plain}
\newtheorem{theorem}{Theorem}[section]
\newtheorem{lemma}[theorem]{Lemma}
\newtheorem{proposition}[theorem]{Proposition}
\newtheorem{corollary}[theorem]{Corollary}
\theoremstyle{definition}
\newtheorem{definition}[theorem]{Definition}
\newtheorem{hypothesis}[theorem]{Assumption}
\newtheorem{remark}[theorem]{Remark}

\newcommand{\R}{\mathbb R}
\newcommand{\E}{\mathbb E}

\newcommand{\Pcal}{\mathcal P}
\newcommand{\Lcal}{\mathcal L}
\newcommand{\F}{\mathbb F}
\newcommand{\Wtwo}{W_2}
\newcommand{\one}{\mathbf 1}

\keywordtitle{Keywords}
\makeatletter
\def\ps@pprintTitle{\let\@oddhead\@empty\let\@evenhead\@empty\let\@oddfoot\@empty\let\@evenfoot\@empty}
\makeatother

\begin{document}
\raggedbottom
\begin{frontmatter}

\title{Forward--Backward McKean--Vlasov Systems with Jumps and a Monotone Source with a Control-Shifted Interface: Well-Posedness, Stability, and Yosida Rates}

\author[aff1]{Kayembe Tshiswaka Tcheick\corref{cor1}}
\ead{tcheick.kayembe@unikin.ac.cd}
\author[aff1]{Mabela Matendo Rostin}
\author[aff1]{Kabengele Mpunga Yannick}
\author[aff2]{Bosonga Bofeki Jean-Pierre}

\cortext[cor1]{Corresponding author.}
\address[aff1]{University of Kinshasa, Faculty of Science and Technology, Department of Mathematics, Statistics and Computer Science, Kinshasa, Democratic Republic of the Congo}
\address[aff2]{University of Kinshasa, Faculty of Economics and Management, Department of Mathematical Economics, Kinshasa, Democratic Republic of the Congo}

\begin{abstract}
We study a forward--backward McKean--Vlasov system with jumps and one-way coupling: the forward equation is independent of the backward variables, whereas the BSDE depends on the forward state and its law. A source generated by a maximal monotone operator is evaluated along the forward state and then accumulated in the backward equation. The operator has the form $A(x,\mu,a)=\partial\phi(x+h(a))+c(\mu)$, so the control shifts the discontinuity interface of its minimal-norm selection. We establish modelwise well-posedness and strong stability, with estimates uniform over the model family. We then study the Yosida approximation: strong convergence, Minty-type strong--weak closedness when the state, its law, and the control---hence the moving interface---vary, and uniform convergence under an integrability condition. For an asymmetric loss-aversion prototype, a quantitative occupation estimate for a layer around the interface yields an explicit Yosida convergence rate. When the continuous martingale part is nondegenerate and the interface moves absolutely continuously, the occupation time is linear in the layer width. The same rate controls the backward components and the cumulative source.
\end{abstract}


\end{frontmatter}

\section*{Introduction}
\addcontentsline{toc}{section}{Introduction}

Fix a finite time horizon $[0,T]$ and let $\Pcal$ be a nonempty family of probabilistic models. Under each $P\in\Pcal$, $B$ denotes a continuous local martingale and $\widetilde N^P$ a compensated Poisson random measure on the mark space $E$; the precise canonical framework is given in Section~1. The initial state is denoted by $\xi$, the predictable control by $\alpha$, and $b$, $\sigma$, and $\beta$ denote the drift, diffusion, and jump coefficients, respectively. We consider a forward--backward McKean--Vlasov system with one-way coupling: the forward equation is independent of the backward variables, while the BSDE depends on the forward state and its law. For each model $P\in\Pcal$, the forward state $X^P$ satisfies
\begin{align}
X_t^P={}&\xi+\int_0^t b(s,X_{s-}^P,\mu_{s-}^P,\alpha_s)\,ds
+\int_0^t\sigma(s,X_{s-}^P,\mu_{s-}^P,\alpha_s)\,dB_s \nonumber\\
&+\int_0^t\int_E\beta(s,X_{s-}^P,\mu_{s-}^P,\alpha_s,e)\,
\widetilde N^P(ds,de),\qquad t\in[0,T],\quad P\text{-a.s.}\label{eq:intro-forward}
\end{align}
where $\mu_t^P=\mathcal L^P(X_t^P)$ is the law of the forward state under $P$. Once $X^P$ has been constructed, it enters the backward equation. Here $g$ is the terminal function, $f$ is the backward generator, $(Y^P,Z^P,U^P)$ are the backward unknowns, and $\eta_t^P:=\mathcal L^P(Y_t^P)$ denotes the law of $Y_t^P$. The BSDE with jumps is
\begin{align}
Y_t^P={}&g(X_T^P,\mu_T^P)
+\int_t^T f(s,X_{s-}^P,\mu_{s-}^P,Y_s^P,Z_s^P,U_s^P,
\eta_s^P,\alpha_s)\,ds \nonumber\\
&+\mathsf K_T^P-\mathsf K_t^P
-\int_t^T Z_s^P\,dB_s
-\int_t^T\int_EU_s^P(e)\,\widetilde N^P(ds,de),\qquad t\in[0,T],\quad P\text{-a.s.}\label{eq:intro-backward}
\end{align}
This one-way coupling gives the system a triangular structure: $(Y^P,Z^P,U^P)$ does not enter the forward equation.

The cumulative source is generated by a maximal monotone operator depending on the state, its law, and the control,
\[
A(x,\mu,a)=\partial\phi(x+h(a))+c(\mu),
\]
where $\phi:\R^d\to\R$ is convex, $D$ is the control set, $h:D\to\R^d$ shifts the interface, and $c:\mathcal P_2(\R^d)\to\R^d$ introduces the dependence on the law; here $\mathcal P_2(\R^d)$ is the Wasserstein space of probability measures with finite second moment. If $A^0$ denotes the minimal-norm selection and $\varrho\in\R^d$ is a fixed vector, we set
\[
\mathsf K_t^P=\int_0^t\langle\varrho,
A^0(X_{s-}^P,\mu_{s-}^P,\alpha_s)\rangle ds,
\qquad t\in[0,T],\quad P\text{-a.s.}
\]
In the asymmetric loss-aversion (ALA) prototype, $A^0$ changes branch across the moving interface
\[
\Sigma(a)=\bigcup_{k=1}^d\{x\in\R^d:x_k=-h_k(a)\}.
\]
The singular difficulty comes from the discontinuity of the selection when $X_t^P+h(\alpha_t)$ crosses $\Sigma(\alpha_t)$.

This architecture differs from standard stochastic variational inequalities. The key issue is not merely the presence of a maximal monotone operator, but where it acts: here the operator is evaluated on the forward state and its law, and its minimal-norm selection then drives the BSDE through $\mathsf K$. The singular behavior is therefore governed by the amount of time $X^P+h(\alpha)$ spends near an interface shifted by the control.

McKean--Vlasov equations go back to McKean \cite{McKean1966} and Sznitman \cite{Sznitman1991}; a modern account is given by Carmona and Delarue \cite{CarmonaDelarue2018}. Mean-field BSDEs were introduced in \cite{BuckdahnDjehicheLiPeng2009,BuckdahnLiPeng2009}. For BSDEs with jumps, see Tang--Li \cite{TangLi1994}, Barles--Buckdahn--Pardoux \cite{BarlesBuckdahnPardoux1997}, Royer \cite{Royer2006}, and, in the mean-field setting, Li \cite{Li2018}. McKean--Vlasov equations with jumps are studied in \cite{Erny2022,LiMin2021}.

The closest multivalued works differ from our setting in three specific ways. First, Qiao and Gong \cite{QiaoGong2023} place the maximal monotone operator in a multivalued McKean--Vlasov backward equation. Ning, Wu, and Zheng \cite{NingWuZheng2024} consider SDEs, McKean--Vlasov SVIs, and a forward--backward variational system in which the subdifferentials act directly in the inclusions. Duan and Wu \cite{DuanWu2026} likewise study mean-field forward--backward systems with oblique subgradients, using Moreau--Yosida approximation and compactness. These frameworks address constraints or inclusions on the system variables; they do not cover the forward-driven source considered here, where $A^0(X^P,\mathcal L^P(X^P),\alpha)$ is first evaluated along the forward state and then integrated into the BSDE.

Second, combining jumps with a moving boundary does not by itself produce our mechanism. Elmansouri, Ouknine, and Ouknine \cite{ElmansouriOuknineOuknine2026} study a one-dimensional multivalued BSDE with jumps and a time-dependent family of maximal monotone operators whose domains have moving boundaries. Djehiche, Dumitrescu, and Zeng \cite{DjehicheDumitrescuZeng2025} consider mean-field reflected BSDEs with jumps and a law-dependent obstacle. In both directions, the finite-variation process is tied to the backward variable, a constraint, or a reflection mechanism. Here it instead comes from a monotone selection evaluated on $X^P$, while the interface is shifted endogenously by $h(\alpha_t)$.

Third, although Yosida approximation is a classical tool for constructing and identifying multivalued solutions, obtaining an explicit rate is delicate when the minimal-norm selection is discontinuous across a moving interface. The works above do not provide, in their respective settings, the quantitative chain developed here:
\[
\begin{gathered}
\text{occupation of a control-shifted interface}\\[-1mm]
\Downarrow\\[-1mm]
\text{Yosida error}\\[-1mm]
\Downarrow\\[-1mm]
\text{error in }(Y,Z,U,\mathsf K).
\end{gathered}
\]
Occupation-time estimates are well known in other stochastic problems \cite{Avikainen2009,FournierPrintems2010,TissotDaguette2026,Friesen2026}. Here they are used to control a discontinuous, law-dependent monotone source evaluated along a McKean--Vlasov state with jumps.

Sections~2--3 provide the well-posedness and stability framework; the main contributions concern the singular limit:
\begin{enumerate}[label=(\roman*),leftmargin=2.2em]
\item under each $P\in\Pcal$, we construct the one-way coupled system and obtain uniform estimates whenever the structural constants are uniform;
\item writing $A^\lambda$ for the Yosida approximation of $A$, we prove strong convergence of $A^\lambda(X^P,\mu^P,\alpha)$ to the minimal-norm selection $A^0(X^P,\mu^P,\alpha)$, together with Minty-type strong--weak closedness when the state, its law, and the control---hence the moving interface---vary;
\item a quantitative occupation estimate for a layer around the interface yields an explicit Yosida convergence rate, which is then transferred to $(Y,Z,U,\mathsf K)$;
\item for the ALA prototype, we prove linear occupation when the continuous martingale part is nondegenerate in the normal directions and $h(\alpha)$ is absolutely continuous.
\end{enumerate}
If the occupation of a layer of width $r$ is of order $r^\theta$ and the minimal-norm selection has a moment of order $2+\delta$, then
\begin{align*}
&\sup_{P\in\Pcal}\E^P\int_0^T
\big|A^\lambda(X_{t-}^P,\mu_{t-}^P,\alpha_t)
-A^0(X_{t-}^P,\mu_{t-}^P,\alpha_t)\big|^2dt\\
&\hspace{35mm}\le C\lambda^\gamma,
\qquad
\gamma=\min\left\{2,\frac{\theta\delta}{2+\delta+\theta}\right\}.
\end{align*}
The same exponent controls the error of the regularized backward system.

The scope is deliberately limited. The coupling is one-way rather than fully coupled; all constructions are modelwise, with no quasi-sure aggregator; and the jump compensator is common to all models. The explicit ALA occupation estimate also uses a continuous martingale component that is nondegenerate in the normal directions. Pure-jump dynamics, rougher moving interfaces, genuinely fully coupled systems, and aggregation under nondominated uncertainty require additional arguments and are not covered here. Section~1 introduces the framework and the operator, Section~2 proves well-posedness, Section~3 establishes stability, Section~4 studies the Yosida approximation, Minty closedness, and occupation estimates, and Section~5 concludes.

\section{Brownian--Poisson Framework}

\subsection{Conventions, model family, and function spaces}

Fix $T>0$, $d,m,r_E\ge1$, $E=\R^{r_E}\setminus\{0\}$, and a L\'evy measure $\Pi$ such that
\begin{equation}
\int_E|e|^2\Pi(de)<\infty.
\label{eq:Pi2}
\end{equation}
This second-moment assumption is maintained throughout the paper; our conventions for L\'evy measures and compensated integrals follow \cite{Applebaum2009}.

For $p\ge1$, let $\mathcal P_p(\R^n)$ denote the set of probability measures $\mu$ on $\R^n$ such that
\[
M_p(\mu)^p:=\int_{\R^n}|x|^p\mu(dx)<\infty.
\]
We write $\delta_x$ for the Dirac mass at $x$. The space $L_\Pi^p(E;\R^n)$ is equipped with the norm
\[
\|u\|_{L_\Pi^p}^p=\int_E|u(e)|^p\Pi(de),
\]
and we use the shorthand $\|u\|_\Pi=\|u\|_{L_\Pi^2}$. Matrix norms are Hilbert--Schmidt norms. The letter $C$ denotes a deterministic constant that may change from line to line. Unless stated otherwise, $C$ may depend on $T$, the dimensions, and the common structural constants introduced below, but not on the particular model $P$ or the Yosida parameter $\lambda$. Moment estimates and convergence rates may also depend on uniform moment bounds and on the parameters appearing in the relevant occupation or regularity assumptions. For estimates at fixed $\lambda$, $C_\lambda$ denotes a constant that may depend on $\lambda$. Thus, notation such as $C_\lambda=C(1+\lambda^{-2})$ is allowed; when $\lambda\downarrow0$ is studied, the factor $1+\lambda^{-2}$ remains explicit and the multiplicative constant $C$ stays independent of $\lambda$.

More precisely, we may take
\[
\Omega=\R^d\times C_0([0,T];\R^m)\times\mathcal M,
\]
where $C_0$ is the space of continuous paths starting from zero and $\mathcal M$ is the space of integer-valued measures on $(0,T]\times E$ that are finite on $(0,T]\times\{|e|>1/n\}$ for every $n\ge1$. The $\sigma$-field on $\mathcal M$ is generated by the evaluation maps. The canonical coordinates are $\xi$, $B$, and $N$. The raw canonical filtration is
\[
\mathcal F_t^0=\sigma\!\left(\xi,B_s,N((0,s]\times G):
0\le s\le t,\ G\in\mathcal B(E),\ \operatorname{dist}(G,0)>0\right),
\]
where $\mathcal B(E)$ is the Borel $\sigma$-field of $E$. For each probability measure $P$ in the nonempty family $\Pcal$, let $\mathcal N^P$ be the ideal of all subsets of $P$-null sets in $\mathcal F_T^0$. We complete under $P$ and define
\[
\mathcal F_T^P=\mathcal F_T^0\vee\mathcal N^P,\qquad
\mathcal F_t^P=\bigcap_{t<u\le T}(\mathcal F_u^0\vee\mathcal N^P)
\quad(t<T),\qquad \F^P=(\mathcal F_t^P)_{0\le t\le T}.
\]
The filtration is complete and right-continuous, and $\xi$ is $\mathcal F_0^P$-measurable. Stochastic integrands are predictable with respect to $\F^P$; for the jump integrand $U$, we use $\mathcal P(\F^P)\otimes\mathcal B(E)$, where $\mathcal P(\F^P)$ is the predictable $\sigma$-field. All martingale statements below refer to this augmented filtration, separately under each $P$. We assume that $B$ is a continuous local martingale and that
\begin{align}
\langle B\rangle_t&=\int_0^t a_s^Pds,\qquad
\underline a\le a_s^P\le\overline a,\quad dt\otimes P\text{-a.e.},
\label{eq:qv}\\
\widetilde N^P(dt,de)&:=N(dt,de)-\Pi(de)dt,
\label{eq:compensated}
\end{align}
where $\mathbb S_m^{>0}$ denotes the cone of symmetric positive-definite $m\times m$ matrices, $\underline a,\overline a\in\mathbb S_m^{>0}$ are independent of $P$, and matrix inequalities are understood in the Loewner order. Throughout the paper, $a_t^P$ (and likewise $\underline a,\overline a$) denotes a covariance matrix, whereas the letter $a\in D$, when used as an argument of $b,\sigma,\beta,f,A$, denotes a control value. The control set remains denoted by $D$, while the data sets compared in Section~3 are denoted by $\mathfrak D$ and $\bar{\mathfrak D}$. Under each $P$, $N$ is a Poisson random measure relative to $\F^P$, with compensator exactly $\Pi(de)dt$. Finally, every square-integrable local martingale under $P$ admits the representation
\begin{equation}
M_t=M_0+\int_0^t Z_s\,dB_s+
\int_0^t\int_EU_s(e)\widetilde N^P(ds,de).
\label{eq:PRP}
\end{equation}
This representation property is an additional assumption on the admissible models; it does not follow from the usual augmentation alone. Without it, a BSDE in a more general filtration may require an additional orthogonal martingale. Our results concern the representation \eqref{eq:PRP} without such a component.

The family $\Pcal$ is called dominated if there exists a probability measure $Q$ such that $P\ll Q$ for every $P\in\Pcal$, and nondominated otherwise. Families of martingale laws with different quadratic-variation densities provide standard examples of mutually singular models; see \cite{DenisMartini2006,SonerTouziZhang2012}. Our framework allows this situation but remains strictly modelwise: \eqref{eq:PRP} is imposed under each $P$, and no common quasi-sure aggregator is constructed. Second-order formulations under nondominated families, with or without jumps, follow a different architecture; see \cite{KaziTaniPossamaiZhou2015a,KaziTaniPossamaiZhou2015b}.

The canonical process $B$ is therefore not assumed to be a standard Brownian motion. For each $P$, the matrix $(a_t^P)^{-1/2}$ is well defined $dt\otimes P$-a.e., and
\[
W_t^P:=\int_0^t(a_s^P)^{-1/2}\,dB_s
\]
is a continuous local martingale satisfying $\langle W^P\rangle_t=tI_m$, where $I_m$ is the $m\times m$ identity matrix. L\'evy's characterization theorem shows that $W^P$ is a Brownian motion under $P$; conversely, $B_t=\int_0^t(a_s^P)^{1/2}dW_s^P$. This identification is strictly modelwise.

For $z\in\R^{1\times m}$ and $u:E\to\R$, set
\[
|z|_{a_t^P}^2:=za_t^Pz^\top,\qquad
\|u\|_\Pi^2:=\int_E|u(e)|^2\Pi(de).
\]
Under $P$, for every $n\ge1$, define
\begin{align*}
\mathbb S^2(P;\R^n)
&=\Bigl\{V\text{ adapted c\`adl\`ag with values in }\R^n:
\E^P\sup_{t\le T}|V_t|^2<\infty\Bigr\},\\
\mathbb H_B^2(P)&=\{Z\text{ predictable}:
\E^P\int_0^T|Z_t|_{a_t^P}^2dt<\infty\},\\
\mathbb H_\Pi^2(P)&=\{U\text{ predictable}:
\E^P\int_0^T\|U_t\|_\Pi^2dt<\infty\}.
\end{align*}
We write $\mathbb S^2(P)=\mathbb S^2(P;\R)$ in the scalar case. Thus the forward process $X^P$ naturally belongs to $\mathbb S^2(P;\R^d)$, while the scalar backward process $Y^P$ belongs to $\mathbb S^2(P)$. For $p\ge2$, $\mathbb S^p(P;\R^n)$ is defined analogously, with norm $(\E^P\sup_{t\le T}|V_t|^p)^{1/p}$. The phrase ``uniform over the model family'' means that the supremum of the corresponding norms over $P\in\Pcal$ is finite; it does not mean that the processes admit a common version across models.

A strong solution under $P$ means a solution adapted to the completed filtration generated by $(\xi,B,N)$ and constructed from these given coordinates. All stochastic identities are understood up to indistinguishability under the fixed model.

Families of solutions may be collected in the spaces
\[
\mathfrak S^2(\Pcal)=\{(V^P)_P:\sup_P\|V^P\|_{\mathbb S^2(P)}<\infty\},
\]
with analogous definitions for $\mathfrak H_B^2$ and $\mathfrak H_\Pi^2$. The norm is the supremum of the modelwise norms. These spaces are complete: a Cauchy sequence converges, for each fixed $P$, in the corresponding complete space; the uniform Cauchy bound then passes to the limit modelwise and finally to the supremum. This construction requires neither measurability with respect to $P$ nor common versions of the processes.

\begin{lemma}[Uniform martingale estimates]\label{lem:mart}
Fix $P\in\Pcal$. Let $Z\in\mathbb H_B^2(P)$ and $U\in\mathbb H_\Pi^2(P)$. There exists a constant $C<\infty$, independent of $P$, such that
\begin{align}
\E^P\sup_{t\le T}\left|\int_0^tZ_s\,dB_s\right|^2
&\le C\E^P\int_0^T|Z_s|_{a_s^P}^2ds,\label{eq:BDGB}\\
\E^P\sup_{t\le T}\left|\int_0^t\int_EU_s(e)\widetilde N^P(ds,de)\right|^2
&\le C\E^P\int_0^T\|U_s\|_\Pi^2ds.\label{eq:BDGN}
\end{align}
Moreover, the two stochastic integrals are orthogonal in $L^2(P)$.
\end{lemma}

\begin{proof}
Set
\[
M_t^c:=\int_0^t Z_s\,dB_s,
\qquad
M_t^d:=\int_0^t\int_EU_s(e)\widetilde N^P(ds,de).
\]
For the continuous part, the predictable quadratic variation is
\[
\langle M^c\rangle_t=\int_0^t Z_sa_s^PZ_s^\top ds
=\int_0^t|Z_s|_{a_s^P}^2ds.
\]
Doob's $L^2$ inequality and the isometry for stochastic integrals with respect to a continuous martingale yield
\begin{align*}
\E^P\sup_{t\le T}|M_t^c|^2
&\le4\E^P|M_T^c|^2
=4\E^P\langle M^c\rangle_T\\
&=4\E^P\int_0^T|Z_s|_{a_s^P}^2ds.
\end{align*}
For the purely discontinuous part, the compensated Poisson isometry gives
\[
\E^P|M_T^d|^2
=\E^P\int_0^T\int_E|U_s(e)|^2\Pi(de)ds
=\E^P\int_0^T\|U_s\|_\Pi^2ds.
\]
Another application of Doob's inequality yields
\[
\E^P\sup_{t\le T}|M_t^d|^2
\le4\E^P\int_0^T\|U_s\|_\Pi^2ds.
\]
The constants are independent of $P$: the continuous norm is written directly in terms of the bracket density $a^P$, while the compensator of $N$ is the same $\Pi(de)dt$ for every model.

Finally, $M^c$ is continuous and $M^d$ is purely discontinuous, so their quadratic covariation vanishes identically,
\[
[M^c,M^d]_t=\sum_{s\le t}\Delta M_s^c\Delta M_s^d=0,
\]
since $\Delta M^c=0$. Integration by parts shows that $M^cM^d$ is a local martingale; square integrability allows us to take expectations and obtain
\[
\E^P[M_T^cM_T^d]=0.
\]
Hence the two stochastic integrals are orthogonal in $L^2(P)$.
\end{proof}

\begin{lemma}[$p$th-order martingale inequality]\label{lem:martp}
Fix $P\in\Pcal$ and $p\ge2$. Let $Z$ and $U$ be predictable and satisfy
\[
\E^P\left[\left(\int_0^T|Z_s|_{a_s^P}^2ds\right)^{p/2}
+\left(\int_0^T\|U_s\|_\Pi^2ds\right)^{p/2}
+\int_0^T\|U_s\|_{L_\Pi^p}^pds\right]<\infty,
\]
and set
\[
M_t=\int_0^t Z_s\,dB_s+\int_0^t\int_E U_s(e)\widetilde N^P(ds,de).
\]
Then there exist constants $0<c_p\le C_p<\infty$, independent of $P$, such that
\begin{align}
c_p\E^P\Bigg[&\left(\int_0^T|Z_s|_{a_s^P}^2ds\right)^{p/2}
+\left(\int_0^T\|U_s\|_\Pi^2ds\right)^{p/2}\nonumber\\
&+\int_0^T\|U_s\|_{L_\Pi^p}^pds\Bigg]
\le \E^P\sup_{t\le T}|M_t|^p,\label{eq:martplower}\\
\E^P\sup_{t\le T}|M_t|^p
&\le C_p\E^P\Bigg[
\left(\int_0^T|Z_s|_{a_s^P}^2ds\right)^{p/2}\nonumber\\
&\quad+\left(\int_0^T\|U_s\|_\Pi^2ds\right)^{p/2}
+\int_0^T\|U_s\|_{L_\Pi^p}^pds\Bigg].\label{eq:martpupper}
\end{align}
\end{lemma}

\begin{proof}
Write again $M=M^c+M^d$, with $M^c=Z\!\cdot\!B$ and $M^d=U\!\star\!\widetilde N^P$. For $p\ge2$, the two-sided BDG inequality applied to the c\`adl\`ag martingale $M$ gives constants $c_p^{(0)},C_p^{(0)}>0$ such that
\[
c_p^{(0)}\E^P[M]_T^{p/2}
\le \E^P\sup_{t\le T}|M_t|^p
\le C_p^{(0)}\E^P[M]_T^{p/2}.
\]
Since the continuous and purely discontinuous parts are orthogonal,
\[
[M]_T=[M^c]_T+[M^d]_T,
\qquad [M^c]_T=\int_0^T|Z_s|_{a_s^P}^2ds.
\]
For the compensated Poisson integral, the scalar $L^p$ Bichteler--Jacod isomorphism gives
\begin{align*}
\E^P\sup_{t\le T}|M_t^d|^p
\asymp_p
\E^P\left(\int_0^T\|U_s\|_\Pi^2ds\right)^{p/2}
+\E^P\int_0^T\|U_s\|_{L_\Pi^p}^pds,
\end{align*}
see \cite[Theorem~1.1]{Dirksen2014}. For the continuous part,
\[
\E^P\sup_{t\le T}|M_t^c|^p
\asymp_p
\E^P\left(\int_0^T|Z_s|_{a_s^P}^2ds\right)^{p/2}.
\]
For the upper bound \eqref{eq:martpupper}, it is enough to use
\[
\sup_{t\le T}|M_t|^p
\le 2^{p-1}\left(\sup_{t\le T}|M_t^c|^p+
\sup_{t\le T}|M_t^d|^p\right)
\]
together with the preceding estimates.

For the lower bound, we avoid a circular argument. Since
\[
[M]_T=[M^c]_T+[M^d]_T,
\]
we have separately $[M^c]_T\le[M]_T$ and $[M^d]_T\le[M]_T$. The two-sided BDG inequality for $M$ therefore yields
\[
\E^P[M^c]_T^{p/2}+\E^P[M^d]_T^{p/2}
\le C_p\,\E^P[M]_T^{p/2}
\le C_p\,\E^P\sup_{t\le T}|M_t|^p.
\]
For the continuous part,
\[
[M^c]_T=\int_0^T|Z_s|_{a_s^P}^2\,ds,
\]
which gives the first term in \eqref{eq:martplower}. For the discontinuous part, the Bichteler--Jacod isomorphism applied to $M^d$ gives
\[
\E^P\left(\int_0^T\|U_s\|_\Pi^2ds\right)^{p/2}
+\E^P\int_0^T\|U_s\|_{L_\Pi^p}^pds
\le C_p\,\E^P\sup_{t\le T}|M_t^d|^p.
\]
Finally, BDG and $[M^d]_T\le[M]_T$ imply
\[
\E^P\sup_{t\le T}|M_t^d|^p
\le C_p\E^P[M^d]_T^{p/2}
\le C_p\E^P[M]_T^{p/2}
\le C_p\E^P\sup_{t\le T}|M_t|^p.
\]
Combining these bounds yields \eqref{eq:martplower}.

The constants do not depend on $P$. Indeed, the continuous part is measured in the $|Z|_{a^P}$ norm; the bounds $\underline a\le a^P\le\overline a$ provide common equivalence constants whenever the Euclidean norm is used. The Poisson part uses the same compensator $\Pi(de)dt$ for every model. The inequalities are first established for simple integrands. If $(Z^n,U^n)$ approximates $(Z,U)$ in the norms on the right-hand side, the corresponding stochastic integrals converge in $L^p(\Omega;\mathbb S^p)$ by the same estimates, which yields the desired extension by completion.
\end{proof}

\subsection{Wasserstein distance and the law-dependent operator}

For $\mu,\nu\in\mathcal P_2(\R^n)$, let $\operatorname{Cpl}(\mu,\nu)$ denote the set of probability measures on $\R^n\times\R^n$ with marginals $\mu$ and $\nu$. We use the distance
\[
W_2(\mu,\nu)^2=\inf_{\pi\in\operatorname{Cpl}(\mu,\nu)}
\int|x-y|^2\pi(dx,dy).
\]

For $\mu\in\mathcal P_2(\R^q)$, write $M_2(\mu)^2=\int|x|^2\mu(dx)$. If $V,V'\in L^2(P)$, then
\begin{equation}
\Wtwo^2(\Lcal^P(V),\Lcal^P(V'))\le\E^P|V-V'|^2.
\label{eq:coupling}
\end{equation}

\begin{lemma}[Law flows and measurable representatives]\label{lem:lawmeas}
Fix $P\in\Pcal$. Let $V,V'\in L^2(dt\otimes P;\R^n)$ be jointly measurable. Then there exist Borel flows $m^V,m^{V'}:[0,T]\to\mathcal P_2(\R^n)$ such that $m_t^V=\Lcal^P(V_t)$ and $m_t^{V'}=\Lcal^P(V_t')$ for a.e. $t$. Their equivalence classes do not depend on the representatives of $V,V'$, and
\[
\int_0^T W_2^2(m_t^V,m_t^{V'})dt
\le\E^P\int_0^T|V_t-V_t'|^2dt.
\]
\end{lemma}
\begin{proof}
Since $V\in L^2(dt\otimes P;\R^n)$, Tonelli's theorem gives
\[
q_V(t):=\E^P|V_t|^2\quad\text{measurable},
\qquad
\int_0^Tq_V(t)dt=
\E^P\int_0^T|V_t|^2dt<\infty.
\]
Hence $q_V(t)<\infty$ for a.e. $t$. On the null set where this moment is infinite, set $m_t^V=\delta_0$; otherwise, set $m_t^V=\Lcal^P(V_t)$. Use the same convention for $V'$.

We now verify measurability of the flow. Let $\psi\in C_b(\R^n)$. The map $(t,\omega)\mapsto\psi(V_t(\omega))$ is jointly measurable and bounded, so Fubini's theorem implies that
\[
t\longmapsto \int_{\R^n}\psi(x)m_t^V(dx)
=\E^P\psi(V_t)
\]
is measurable. Choosing a countable family in $C_b(\R^n)$ that determines weak convergence gives measurability of $t\mapsto m_t^V$ for the Borel $\sigma$-field of the weak topology. Moreover,
\[
t\longmapsto M_2(m_t^V)^2
=\E^P|V_t|^2=q_V(t)
\]
is measurable. Since the $W_2$ topology on $\mathcal P_2(\R^n)$ is characterized by weak convergence together with convergence of second moments, $m^V$ is Borel as a map into $(\mathcal P_2(\R^n),W_2)$. The same argument applies to $m^{V'}$.

For a.e. $t$, the joint law
\[
\pi_t:=\Lcal^P(V_t,V_t')
\]
is a coupling of $m_t^V$ and $m_t^{V'}$. By the definition of $W_2$,
\[
W_2^2(m_t^V,m_t^{V'})
\le \int_{\R^n\times\R^n}|x-y|^2\pi_t(dx,dy)
=\E^P|V_t-V_t'|^2.
\]
Integrating in time and applying Tonelli's theorem yields
\[
\int_0^T W_2^2(m_t^V,m_t^{V'})dt
\le \E^P\int_0^T|V_t-V_t'|^2dt.
\]
If $V=V'$ in $L^2(dt\otimes P)$, the right-hand side is zero, so the two flows agree for a.e. $t$. In particular, any expression involving $\Lcal^P(V_t)$ inside a time integral depends only on the equivalence class of $V$ in $L^2(dt\otimes P)$.
\end{proof}

Let $D\subset\R^k$ be a closed set containing $0$. Let $h:D\to\R^d$ be Borel, Lipschitz, and of linear growth: there exist $L_h,C_h<\infty$ such that
\begin{equation}
 |h(a)-h(a')|\le L_h|a-a'|,
 \qquad |h(a)|\le C_h(1+|a|).
\label{eq:hlip}
\end{equation}
Let $c:\mathcal P_2(\R^d)\to\R^d$ be bounded and Lipschitz with respect to the Wasserstein distance:
\begin{equation}
 |c(\mu)-c(\nu)|\le L_cW_2(\mu,\nu),
 \qquad \sup_{\mu\in\mathcal P_2(\R^d)}|c(\mu)|\le C_c.
\label{eq:clip}
\end{equation}
Finally, let $\phi:\R^d\to\R$ be convex, finite, lower semicontinuous, and $\alpha_A$-strongly convex. The law-dependent, control-shifted operator is
\begin{equation}
 A(x,\mu,a):=\partial\phi(x+h(a))+c(\mu),
 \qquad (x,\mu,a)\in\R^d\times\mathcal P_2(\R^d)\times D.
\label{eq:A}
\end{equation}
Here $\partial\phi$ is the convex subdifferential. For fixed $(\mu,a)$, $A(\cdot,\mu,a)$ is the state-space translate of the maximal monotone operator $\partial\phi+c(\mu)$. Equivalently, if
\[
 \Phi(x,\mu,a):=\phi(x+h(a))+\langle c(\mu),x+h(a)\rangle,
\]
then
\[
 A(x,\mu,a)=\partial_x\Phi(x,\mu,a).
\]
Thus every value $A(x,\mu,a)$ is nonempty, closed, and convex, and has a unique element of minimal Euclidean norm. We denote this minimal-norm selection by $A^0(x,\mu,a)$.

For $\lambda>0$, define the resolvent and the Yosida approximation in the state variable by
\begin{equation}
 J_\lambda(x,\mu,a)
 :=(I+\lambda A(\cdot,\mu,a))^{-1}(x),
 \qquad
 A^\lambda(x,\mu,a):=\lambda^{-1}(x-J_\lambda(x,\mu,a)).
\label{eq:YosidaDef}
\end{equation}
By the definition of the resolvent, for every $(x,\mu,a)$ and every $\lambda>0$,
\begin{align}
A^\lambda(x,\mu,a)&\in
A(J_\lambda(x,\mu,a),\mu,a),\label{eq:Yosida-graph}\\
J_\lambda(x,\mu,a)-x&=-\lambda A^\lambda(x,\mu,a).
\label{eq:Yosida-displacement}
\end{align}
The next proposition collects the resolvent properties used below.

\begin{proposition}[Resolvent and Yosida approximation for the moving source]\label{prop:Yosida}
Let $\lambda>0$, $x,x'\in\R^d$, $\mu,\mu'\in\mathcal P_2(\R^d)$, and $a,a'\in D$. Then $A(\cdot,\mu,a)$ is maximal and $\alpha_A$-strongly monotone, and
\begin{align}
|J_\lambda(x,\mu,a)-J_\lambda(x',\mu',a')|
&\le |x-x'|+2L_h|a-a'|+\lambda L_cW_2(\mu,\mu'),\label{eq:Jlip}\\
|A^\lambda(x,\mu,a)-A^\lambda(x',\mu',a')|
&\le\lambda^{-1}\{|x-x'|+L_h|a-a'|\}+L_cW_2(\mu,\mu'),\label{eq:AYlip}\\
|A^\lambda(x,\mu,a)|&\le|A^0(x,\mu,a)|,\qquad
A^\lambda(x,\mu,a)\longrightarrow A^0(x,\mu,a)\quad(\lambda\downarrow0).\label{eq:Aconv}
\end{align}
If $a=a'$, then
\begin{equation}
|J_\lambda(x,\mu,a)-J_\lambda(x',\mu',a)|
\le |x-x'|+\lambda L_cW_2(\mu,\mu').
\label{eq:Jlip-samea}
\end{equation}
\end{proposition}

\begin{proof}
First fix $(\mu,a)$. The function
\[
\Phi(x,\mu,a)=\phi(x+h(a))+\langle c(\mu),x+h(a)\rangle
\]
is proper, finite, and lower semicontinuous. Since $\phi$ is $\alpha_A$-strongly convex and the second term is affine in $x$, $\Phi(\cdot,\mu,a)$ remains $\alpha_A$-strongly convex. The convex subdifferential theorem therefore implies that
\[
A(\cdot,\mu,a)=\partial_x\Phi(\cdot,\mu,a)
\]
is maximal monotone. More precisely, if $v\in A(x,\mu,a)$ and $v'\in A(x',\mu,a)$, strong convexity gives
\begin{equation}
\langle v-v',x-x'\rangle\ge\alpha_A|x-x'|^2.
\label{eq:Yosida-strongmon-proof}
\end{equation}

Now fix $(x,\mu,a)$ and set
\[
y=J_\lambda(x,\mu,a),\qquad v=A^\lambda(x,\mu,a).
\]
By the resolvent identity,
\[
x=y+\lambda v,\qquad v\in A(y,\mu,a).
\]
Let $v_0=A^0(x,\mu,a)\in A(x,\mu,a)$. Monotonicity of $A$ yields
\[
0\le\langle v-v_0,y-x\rangle
=-\lambda\langle v-v_0,v\rangle.
\]
Hence $|v|^2\le\langle v_0,v\rangle\le|v_0||v|$, and therefore
\begin{equation}
|A^\lambda(x,\mu,a)|\le|A^0(x,\mu,a)|.
\label{eq:Yosida-dom-proof}
\end{equation}
In particular,
\[
|J_\lambda(x,\mu,a)-x|
=\lambda|A^\lambda(x,\mu,a)|
\le\lambda|A^0(x,\mu,a)|\longrightarrow0.
\]
The family $A^\lambda(x,\mu,a)$ is bounded. In finite dimension, every subsequence therefore has a further subsequence converging to some vector $v_*$. Since
\[
J_{\lambda_n}(x,\mu,a)\to x,
\qquad
A^{\lambda_n}(x,\mu,a)\to v_*,
\qquad
A^{\lambda_n}(x,\mu,a)\in A(J_{\lambda_n}(x,\mu,a),\mu,a),
\]
closedness of the graph of a maximal monotone operator gives $v_*\in A(x,\mu,a)$. The domination \eqref{eq:Yosida-dom-proof} yields $|v_*|\le|A^0(x,\mu,a)|$. By minimality of $A^0$, necessarily $v_*=A^0(x,\mu,a)$. Since every cluster point is the same, the whole family converges:
\[
A^\lambda(x,\mu,a)\to A^0(x,\mu,a).
\]
This proves \eqref{eq:Aconv}.

It remains to control the dependence on $(x,\mu,a)$. Set
\[
z=x+h(a),\qquad z'=x'+h(a'),
\qquad B(s,\mu)=\partial\phi(s)+c(\mu),
\]
and write $\widehat J_\lambda(\cdot,\mu)=(I+\lambda B(\cdot,\mu))^{-1}$. We make the translation identity explicit. If $y=J_\lambda(x,\mu,a)$, there exists $p\in\partial\phi(y+h(a))$ such that
\[
x=y+\lambda\bigl(p+c(\mu)\bigr).
\]
Setting $w=y+h(a)$ and $z=x+h(a)$ gives
\[
z=w+\lambda\bigl(p+c(\mu)\bigr),\qquad p\in\partial\phi(w),
\]
so $w=\widehat J_\lambda(z,\mu)$. Consequently,
\begin{equation}
J_\lambda(x,\mu,a)+h(a)=\widehat J_\lambda(z,\mu),
\qquad
A^\lambda(x,\mu,a)=B^\lambda(z,\mu).
\label{eq:translation-yosida}
\end{equation}
For fixed law, firm nonexpansiveness of the resolvent gives
\[
|\widehat J_\lambda(z,\mu)-\widehat J_\lambda(z',\mu)|\le|z-z'|.
\]
For fixed state $z'$, set
\[
y=\widehat J_\lambda(z',\mu),\qquad
\bar y=\widehat J_\lambda(z',\mu').
\]
There exist $p\in\partial\phi(y)$ and $\bar p\in\partial\phi(\bar y)$ such that
\[
z'=y+\lambda(p+c(\mu)),
\qquad
z'=\bar y+\lambda(\bar p+c(\mu')).
\]
Subtracting and taking the inner product with $y-\bar y$ yields
\[
|y-\bar y|^2
+\lambda\langle p-\bar p,y-\bar y\rangle
=-\lambda\langle c(\mu)-c(\mu'),y-\bar y\rangle.
\]
The monotone term is nonnegative, and Cauchy--Schwarz gives
\[
|y-\bar y|\le\lambda|c(\mu)-c(\mu')|
\le\lambda L_cW_2(\mu,\mu').
\]
The triangle inequality therefore yields
\[
|\widehat J_\lambda(z,\mu)-\widehat J_\lambda(z',\mu')|
\le|z-z'|+\lambda L_cW_2(\mu,\mu').
\]
Returning to $J_\lambda$ and using $|z-z'|\le|x-x'|+L_h|a-a'|$, we obtain
\[
|J_\lambda(x,\mu,a)-J_\lambda(x',\mu',a')|
\le |x-x'|+2L_h|a-a'|+\lambda L_cW_2(\mu,\mu'),
\]
which is \eqref{eq:Jlip}. If $a=a'$, the term generated by subtracting the two translations $h(a)$ and $h(a')$ vanishes, giving \eqref{eq:Jlip-samea}.

Finally, for the Yosida approximation, at fixed law,
\[
|B^\lambda(z,\mu)-B^\lambda(z',\mu)|\le\lambda^{-1}|z-z'|.
\]
At fixed state, the identity $B^\lambda(z,\mu)=\lambda^{-1}(z-\widehat J_\lambda(z,\mu))$ and the preceding law estimate give
\[
|B^\lambda(z,\mu)-B^\lambda(z,\mu')|
\le L_cW_2(\mu,\mu').
\]
One last application of the triangle inequality, together with $|z-z'|\le|x-x'|+L_h|a-a'|$, proves \eqref{eq:AYlip} and completes the proof.
\end{proof}

The estimates \eqref{eq:Jlip} and \eqref{eq:AYlip} show that, for fixed $\lambda>0$, the maps
\[
(x,\mu,a)\longmapsto J_\lambda(x,\mu,a),
\qquad
(x,\mu,a)\longmapsto A^\lambda(x,\mu,a)
\]
are continuous on $\R^d\times\mathcal P_2(\R^d)\times D$ when $\mathcal P_2(\R^d)$ is endowed with $W_2$, and hence are Borel measurable. By \eqref{eq:Aconv}, $A^0$ is Borel measurable as the pointwise limit of $A^{1/n}$. Thus $A^0(X_{t-},\Lcal(X_{t-}),\alpha_t)$ and its Yosida approximations are predictable whenever $X$ is adapted c\`adl\`ag and $\alpha$ is predictable.

\paragraph{Control-shifted ALA prototype.}
For $r\in\R$, write $r_+=\max\{r,0\}$ and $r_-=\max\{-r,0\}$. The main prototype considered in this paper is
\begin{equation}
\phi_{\rm ALA}(s)
=\sum_{k=1}^d\{\kappa_+(s_k)_++\kappa_-(s_k)_-\}
+\frac{\alpha_A}{2}|s|^2,
\qquad \kappa_->\kappa_+>0,
\label{eq:ALAphi}
\end{equation}
combined with the shift $s=x+h(a)$ and the law translation $c(\mu)$:
\begin{equation}
A_{\rm ALA}(x,\mu,a)
:=\partial\phi_{\rm ALA}(x+h(a))+c(\mu).
\label{eq:ALAmov}
\end{equation}
We invoke the ALA prototype only when its piecewise structure is needed. The well-posedness and stability results of Sections~2--3, as well as the general convergence and closedness results at the beginning of Section~4, remain stated for the general operator \eqref{eq:A} under the corresponding assumptions. Results that use orthant geometry or the interface \eqref{eq:Sigma} are explicitly identified as belonging to the \emph{ALA prototype}. For $(x,\mu,a)\in\R^d\times\mathcal P_2(\R^d)\times D$, set $s=x+h(a)$. Componentwise, the minimal-norm selection is
\begin{equation}
[A^0(x,\mu,a)]_k
=\alpha_As_k+c_k(\mu)+
\begin{cases}
\kappa_+,&s_k>0,\\
-\kappa_-,&s_k<0,\\
\operatorname{proj}_{[-\kappa_-,\kappa_+]}
\bigl(-c_k(\mu)\bigr),&s_k=0.
\end{cases}
\label{eq:A0explicit}
\end{equation}
The switching set depends on the control:
\begin{equation}
\Sigma(a):=\bigcup_{k=1}^d\{x\in\R^d:x_k=-h_k(a)\},
\qquad
\rho(x,a):=\operatorname{dist}(x,\Sigma(a))
=\min_{1\le k\le d}|x_k+h_k(a)|.
\label{eq:Sigma}
\end{equation}
For every $(x,\mu,a)\in\R^d\times\mathcal P_2(\R^d)\times D$, the source satisfies the uniform growth bound
\begin{equation}
|A^0(x,\mu,a)|
\le C_A(1+|x|+|a|),
\label{eq:A0growth}
\end{equation}
where $C_A$ depends only on $d,\alpha_A,\kappa_\pm,C_h,C_c$. In particular, no moment of $\mu$ is needed in this estimate because $c$ is bounded.

In the scalar case $d=1$, one choice used in the companion control problem is
\[
c(\mu)=\eta_c\tanh\!\left(\int_\R y\,\mu(dy)\right),
\qquad 0\le\eta_c<\min\{\kappa_+,\kappa_-\}.
\]
The interface then reduces to the single moving threshold $\Sigma(a)=\{-h(a)\}$.

\subsection{Coefficients, controls, and the system}

The control $\alpha$ takes values in the closed set $D\subset\R^k$ fixed above; that is, for every $P\in\Pcal$, $\alpha_t\in D$ for all $t\in[0,T]$, $P$-a.s. To avoid ambiguity caused by different completed filtrations, we require $\alpha$ to admit a version predictable with respect to the raw canonical filtration generated by $(\xi,B,N)$. The same functional then defines, after completion, an $\F^P$-predictable control under each model. For $p\ge2$, its integrability is specified by
\begin{equation}
\sup_{P\in\Pcal}\E^P\int_0^T|\alpha_t|^pdt<\infty.
\label{eq:controlq}
\end{equation}
The set of controls satisfying these conditions is denoted by $\mathcal A^p$. Well-posedness only requires $\alpha\in\mathcal A^2$. The coefficients have dimensions
\begin{align*}
b&:[0,T]\times\R^d\times\mathcal P_2(\R^d)\times D\to\R^d,\\
\sigma&:[0,T]\times\R^d\times\mathcal P_2(\R^d)\times D
\to\R^{d\times m},\\
\beta&:[0,T]\times\R^d\times\mathcal P_2(\R^d)\times D\times E
\to\R^d,\\
f&:[0,T]\times\R^d\times\mathcal P_2(\R^d)\times\R
\times\R^{1\times m}\times L_\Pi^2(E;\R)
\times\mathcal P_2(\R)\times D\to\R,\\
g&:\R^d\times\mathcal P_2(\R^d)\to\R.
\end{align*}

\begin{hypothesis}[Forward coefficients]\label{HF}
The functions $b,\sigma,\beta$ are measurable, and there exist $L_F,C_F<\infty$ such that, for every $t\in[0,T]$, all $x,x'\in\R^d$, $\mu,\mu'\in\mathcal P_2(\R^d)$, and $a,a'\in D$,
\begin{align}
&|b(t,x,\mu,a)-b(t,x',\mu',a')|
+\|\sigma(t,x,\mu,a)-\sigma(t,x',\mu',a')\|\nonumber\\
&\quad+\left(\int_E|\beta(t,x,\mu,a,e)
-\beta(t,x',\mu',a',e)|^2\Pi(de)\right)^{1/2}\nonumber\\
&\le L_F\bigl(|x-x'|+\Wtwo(\mu,\mu')+|a-a'|\bigr),
\label{eq:FLip}\\
&|b(t,0,\delta_0,0)|+\|\sigma(t,0,\delta_0,0)\|
+\|\beta(t,0,\delta_0,0,\cdot)\|_{L_\Pi^2}\le C_F.
\label{eq:Fzero}
\end{align}
\end{hypothesis}

\begin{hypothesis}[Backward coefficients]\label{HB}
The functions $f$ and $g$ are Borel measurable. There exist $L_B,C_B<\infty$ such that, for every $t\in[0,T]$, all $x,x'\in\R^d$, $\mu,\mu'\in\mathcal P_2(\R^d)$, $y,y'\in\R$, $z,z'\in\R^{1\times m}$, $u,u'\in L_\Pi^2(E;\R)$, $\eta,\eta'\in\mathcal P_2(\R)$, and $a,a'\in D$,
\begin{align}
&|f(t,x,\mu,y,z,u,\eta,a)
-f(t,x',\mu',y',z',u',\eta',a')|\nonumber\\
&\le L_B\bigl(
|x-x'|+\Wtwo(\mu,\mu')+|y-y'|+|z-z'|
+\|u-u'\|_\Pi+\Wtwo(\eta,\eta')+|a-a'|
\bigr),\label{eq:fLip}\\
&|f(t,0,\delta_0,0,0,0,\delta_0,0)|\le C_B.\label{eq:fzero}
\end{align}
Moreover, for all $x,x'\in\R^d$ and $\mu,\mu'\in\mathcal P_2(\R^d)$,
\begin{align}
&|g(x,\mu)-g(x',\mu')|
\le L_B\bigl(|x-x'|+\Wtwo(\mu,\mu')\bigr),
\qquad |g(0,\delta_0)|\le C_B.
\label{eq:gLip}
\end{align}
\end{hypothesis}

\begin{hypothesis}[Data and control-dependent selected source]\label{HD}
The initial variable $\xi$ is $\mathcal F_0^P$-measurable under each $P$, its law does not depend on the model, and
\[
\sup_{P\in\Pcal}\E^P|\xi|^2<\infty.
\]
The source is the translated operator \eqref{eq:A}; its minimal-norm selection $(x,\mu,a)\mapsto A^0(x,\mu,a)$ is Borel measurable, and we assume that, for every $(x,\mu,a)\in\R^d\times\mathcal P_2(\R^d)\times D$,
\[
|A^0(x,\mu,a)|\le C_A(1+|x|+|a|).
\]
This condition is automatic for the ALA prototype by \eqref{eq:A0growth}. Finally, $\varrho\in\R^d$ is fixed.
\end{hypothesis}

\begin{hypothesis}[Higher moments]\label{HqJ}
For a fixed $q>4$, assume $\alpha\in\mathcal A^q$ and $\sup_{P\in\Pcal}\E^P|\xi|^q<\infty$. For every $t\in[0,T]$, $x\in\R^d$, $\mu\in\mathcal P_q(\R^d)$, and $a\in D$,
\begin{equation}
\int_E|\beta(t,x,\mu,a,e)|^q\Pi(de)
\le C_F^q(1+|x|^q+M_q(\mu)^q+|a|^q).
\label{eq:betaq}
\end{equation}
This assumption supplements the quadratic framework only when higher-order moments and quantitative regularization estimates are needed. It is not required for well-posedness or quadratic stability.
\end{hypothesis}

For $P\in\Pcal$ and $t\in[0,T]$, set
\[
\mu_t^P=\Lcal^P(X_t^P),\quad
\mu_{t-}^P=\Lcal^P(X_{t-}^P),\quad
\eta_t^P=\Lcal^P(Y_t^P).
\]
Under the preceding assumptions, the system is rigorously formulated as
\begin{align}
X_t^P
={}&\xi+\int_0^t b(s,X_{s-}^P,\mu_{s-}^P,\alpha_s)ds
+\int_0^t\sigma(s,X_{s-}^P,\mu_{s-}^P,\alpha_s)dB_s\nonumber\\
&+\int_0^t\int_E\beta(s,X_{s-}^P,\mu_{s-}^P,\alpha_s,e)
\widetilde N^P(ds,de),\qquad t\in[0,T],\quad P\text{-a.s.}\label{eq:FWD}\\
Y_t^P
={}&g(X_T^P,\mu_T^P)
+\int_t^T f(s,X_{s-}^P,\mu_{s-}^P,Y_s^P,Z_s^P,U_s^P,
\eta_s^P,\alpha_s)ds\nonumber\\
&+\mathsf K_T^P-\mathsf K_t^P
-\int_t^TZ_s^P\,dB_s
-\int_t^T\int_EU_s^P(e)\widetilde N^P(ds,de),\qquad t\in[0,T],\quad P\text{-a.s.}\label{eq:BWD}\\
\mathsf K_t^P
={}&\int_0^t\langle\varrho,
A^0(X_{s-}^P,\mu_{s-}^P,\alpha_s)\rangle ds,\qquad t\in[0,T],\quad P\text{-a.s.}
\label{eq:Kdef}
\end{align}
Left limits are used to make the integrands predictable. Since a c\`adl\`ag process has at most countably many jump times, $X_{s-}^P=X_s^P$ for $dt\otimes P$-a.e. $(s,\omega)$.

\begin{definition}[Modelwise solution]
\label{def:modelsolution}
Fix $P\in\Pcal$. A solution of the selected system is a sextuple
\[
(X^P,Y^P,Z^P,U^P,\Gamma^P,\mathsf K^P)
\]
such that
\begin{align*}
&X^P\in\mathbb S^2(P;\R^d),\qquad
Y^P,\mathsf K^P\in\mathbb S^2(P),\\
&Z^P\in\mathbb H_B^2(P),\qquad
U^P\in\mathbb H_\Pi^2(P),\qquad
\Gamma^P\in L^2(dt\otimes P;\R^d).
\end{align*}
The processes $X^P,Y^P$ are adapted and c\`adl\`ag, $Z^P,U^P$ are predictable, $\Gamma^P$ is predictable, and $\mathsf K^P$ is adapted, continuous, and absolutely continuous in time. They satisfy
\begin{align*}
\mu_t^P&=\Lcal^P(X_t^P),&
\eta_t^P&=\Lcal^P(Y_t^P),\\
\Gamma_t^P&=A^0(X_{t-}^P,\mu_{t-}^P,\alpha_t)
&&dt\otimes P\text{-a.e.},
\end{align*}
and the identities \eqref{eq:FWD}--\eqref{eq:Kdef} for every $t\in[0,T]$, up to indistinguishability under $P$. All stochastic integrals are taken in their respective square-integrable spaces. A regularized solution with parameter $\lambda>0$ is defined analogously by replacing $A^0(X,\mu,\alpha)$ with $A^\lambda(X,\mu,\alpha)$ in $\Gamma^P$ and in \eqref{eq:Kdef}.
\end{definition}

\paragraph{Dimension of the backward variable.}
The variable $Y$ is scalar and the vector-valued source is projected through $\langle\varrho,\cdot\rangle$. Monotonicity of $A$ remains a property in the state variable $x$ and does not induce any monotonicity in $Y$. A vector-valued extension of $Y$ is possible under the corresponding Lipschitz assumptions but is not pursued here.

\begin{remark}[Nature of $\mathsf K$]
The operator acts on the forward state. The process $\mathsf K^P$ is absolutely continuous, and its total variation satisfies
\[
\operatorname{Var}(\mathsf K^P)_T
\le |\varrho|\int_0^T|A^0(X_{s-}^P,\mu_{s-}^P,\alpha_s)|ds.
\]
It need not be increasing because the inner product with $\varrho$ may change sign. It does not carry the jumps of $Y^P$; those come from the Poisson integral. Thus \eqref{eq:BWD} is neither a reflected equation nor a backward stochastic variational inequality (BSVI) in the variable $Y$.
\end{remark}

Occupation assumptions are used only in Section~4 to obtain the quantitative Yosida rate.

\section{Modelwise Well-Posedness}

\subsection{Forward equation with jumps}

For the forward equation with a prescribed law flow, we use standard Brownian--Poisson SDE theory; see \cite[Chapter~6]{Applebaum2009}. The McKean--Vlasov dependence is then handled by a fixed point on the flow of laws.

\begin{lemma}[Integral stability of the forward equation]\label{lem:Fstab}
Fix $P\in\Pcal$ and assume \ref{HF}. Let $m,\bar m\in C([0,T];\mathcal P_2(\R^d))$, let $\xi,\bar\xi\in L^2(\mathcal F_0^P,P;\R^d)$, and let $\alpha,\bar\alpha\in L^2(dt\otimes P)$ be two predictable $D$-valued controls. Let $X,\bar X\in\mathbb S^2(P;\R^d)$ be the solutions corresponding to $(m,\xi,\alpha)$ and $(\bar m,\bar\xi,\bar\alpha)$, respectively, under the same model and driven by the same integrators. Then, for every $t\in[0,T]$,
\begin{align}
\E^P\sup_{r\le t}|X_r-\bar X_r|^2
\le C\Bigl(&\E^P|\xi-\bar\xi|^2
+\E^P\int_0^t|\alpha_s-\bar\alpha_s|^2ds\nonumber\\
&+\int_0^t\Wtwo^2(m_s,\bar m_s)ds\Bigr),
\label{eq:Fstab}
\end{align}
where $C$ depends only on $T$, the structural constants in \ref{HF}, $\overline a$, and the dimensions.
\end{lemma}

\begin{proof}
Subtracting the two forward equations first, before introducing shorthand notation, gives for every $r\in[0,T]$, $P$-a.s.,
\begin{align*}
X_r-\bar X_r
={}&\xi-\bar\xi
+\int_0^r\Bigl[b(s,X_{s-},m_s,\alpha_s)
-b(s,\bar X_{s-},\bar m_s,\bar\alpha_s)\Bigr]ds\\
&+\int_0^r\Bigl[\sigma(s,X_{s-},m_s,\alpha_s)
-\sigma(s,\bar X_{s-},\bar m_s,\bar\alpha_s)\Bigr]dB_s\\
&+\int_0^r\int_E\Bigl[\beta(s,X_{s-},m_s,\alpha_s,e)
-\beta(s,\bar X_{s-},\bar m_s,\bar\alpha_s,e)\Bigr]\widetilde N^P(ds,de).
\end{align*}
We now introduce the compact notation used in the estimates:
\[
\Delta X_t:=X_t-\bar X_t,\qquad
\Delta\alpha_t:=\alpha_t-\bar\alpha_t,
\qquad
\Delta\xi:=\xi-\bar\xi,
\]
and, for $s\in[0,T]$ and $e\in E$,
\begin{align*}
\Delta b_s&:=b(s,X_{s-},m_s,\alpha_s)-b(s,\bar X_{s-},\bar m_s,\bar\alpha_s),\\
\Delta\sigma_s&:=\sigma(s,X_{s-},m_s,\alpha_s)-\sigma(s,\bar X_{s-},\bar m_s,\bar\alpha_s),\\
\Delta\beta_s(e)&:=\beta(s,X_{s-},m_s,\alpha_s,e)-\beta(s,\bar X_{s-},\bar m_s,\bar\alpha_s,e).
\end{align*}
The process identity defining $\Delta X$ is understood up to indistinguishability under $P$; the identities for the predictable integrands are understood $dt\otimes P$-a.e., and for $\Delta\beta$ also $\Pi(de)$-a.e. in $e$. With this notation, the preceding identity becomes
\[
\Delta X_r=\Delta\xi+\int_0^r\Delta b_sds
+\int_0^r\Delta\sigma_s dB_s
+\int_0^r\int_E\Delta\beta_s(e)\widetilde N^P(ds,de),
\qquad r\in[0,T],\quad P\text{-a.s.}
\]
For $r\le t$, the inequality $|x_1+x_2+x_3+x_4|^2\le4\sum_{j=1}^4|x_j|^2$ gives, after taking the supremum and expectation,
\begin{align*}
\E^P\sup_{r\le t}|\Delta X_r|^2
\le{}&4\E^P|\Delta\xi|^2
+4\E^P\sup_{r\le t}\left|\int_0^r\Delta b_sds\right|^2\\
&+4\E^P\sup_{r\le t}\left|\int_0^r\Delta\sigma_s dB_s\right|^2
+4\E^P\sup_{r\le t}\left|\int_0^r\int_E\Delta\beta_s(e)\widetilde N^P(ds,de)\right|^2.
\end{align*}
Cauchy--Schwarz in time yields
\[
\sup_{r\le t}\left|\int_0^r\Delta b_sds\right|^2
\le t\int_0^t|\Delta b_s|^2ds.
\]
For the continuous martingale, Lemma~\ref{lem:mart} and $a^P\le\overline a$ give
\[
\E^P\sup_{r\le t}\left|\int_0^r\Delta\sigma_s dB_s\right|^2
\le C\E^P\int_0^t\|\Delta\sigma_s\|^2ds,
\]
and, for the Poisson part,
\[
\E^P\sup_{r\le t}\left|\int_0^r\int_E\Delta\beta_s(e)\widetilde N^P(ds,de)\right|^2
\le C\E^P\int_0^t\|\Delta\beta_s\|_{L_\Pi^2}^2ds.
\]
Assumption~\ref{HF} implies, with a constant depending only on $L_F$,
\begin{align*}
|\Delta b_s|^2+\|\Delta\sigma_s\|^2+\|\Delta\beta_s\|_{L_\Pi^2}^2
\le C\bigl(&|\Delta X_{s-}|^2+W_2^2(m_s,\bar m_s)\\
&+|\alpha_s-\bar\alpha_s|^2\bigr).
\end{align*}
Since $|\Delta X_{s-}|\le\sup_{u\le s}|\Delta X_u|$ outside a $ds\otimes P$-null set, defining
\[
u_P(t):=\E^P\sup_{r\le t}|\Delta X_r|^2,
\]
we obtain
\begin{align*}
u_P(t)
\le{}&C\E^P|\xi-\bar\xi|^2
+C\E^P\int_0^t|\alpha_s-\bar\alpha_s|^2ds
+C\int_0^tW_2^2(m_s,\bar m_s)ds\\
&+C\int_0^tu_P(s)ds.
\end{align*}
The function
\[
h_P(t)=C\left(\E^P|\xi-\bar\xi|^2+
\E^P\int_0^t|\alpha_s-\bar\alpha_s|^2ds+
\int_0^tW_2^2(m_s,\bar m_s)ds\right)
\]
is nondecreasing. The Gronwall form
\begin{equation}
u(t)\le h(t)+k\int_0^tu(s)ds,
\quad h\text{ nondecreasing}
\quad\Longrightarrow\quad
u(t)\le e^{kt}h(t)
\label{eq:GronwallExplicit}
\end{equation}
yields \eqref{eq:Fstab}. The constants in the three stochastic estimates are common to the model family; no dominating measure is used.
\end{proof}

\begin{corollary}[Stability with random law flows]\label{cor:Fstabrandom}
Fix $P\in\Pcal$ and assume \ref{HF}. Let $m,\bar m$ be progressively measurable $\mathcal P_2(\R^d)$-valued flows such that
\[
\E^P\int_0^T\{M_2(m_s)^2+M_2(\bar m_s)^2\}ds<\infty.
\]
Under the remaining assumptions of Lemma~\ref{lem:Fstab}, the solutions $X,\bar X$ satisfy, for every $t\in[0,T]$,
\begin{align*}
\E^P\sup_{r\le t}|X_r-\bar X_r|^2
\le C\Bigl(&\E^P|\xi-\bar\xi|^2
+\E^P\int_0^t|\alpha_s-\bar\alpha_s|^2ds\\
&+\E^P\int_0^t\Wtwo^2(m_s,\bar m_s)ds\Bigr).
\end{align*}
\end{corollary}

\begin{proof}
For progressively measurable random flows, the difference equation is the same as in Lemma~\ref{lem:Fstab}. Assumption~\ref{HF} now gives, pointwise in $(s,\omega)$,
\[
|\Delta b_s|^2+\|\Delta\sigma_s\|^2+\|\Delta\beta_s\|_{L_\Pi^2}^2
\le C\bigl(|\Delta X_{s-}|^2+W_2^2(m_s,\bar m_s)+|\Delta\alpha_s|^2\bigr).
\]
The moment assumptions on the flows ensure integrability of all terms. Repeating Cauchy--Schwarz for the drift and the two estimates from Lemma~\ref{lem:mart}, and then taking expectations, gives
\[
\E^P\sup_{r\le t}|\Delta X_r|^2
\le C\E^P|\Delta\xi|^2
+C\E^P\int_0^t\bigl(|\Delta\alpha_s|^2+W_2^2(m_s,\bar m_s)\bigr)ds
+C\int_0^t\E^P\sup_{u\le s}|\Delta X_u|^2ds.
\]
Applying \eqref{eq:GronwallExplicit} proves the claim. Modifications of the flows on a $dt\otimes P$-null set do not affect any of the integrals.
\end{proof}

\begin{lemma}[Construction of the frozen-law forward equation]\label{lem:frozenfull}
Fix $P\in\Pcal$ and assume \ref{HF}. Let $m\in C([0,T];\mathcal P_2(\R^d))$, $\xi\in L^2(\mathcal F_0^P,P;\R^d)$, and let $\alpha$ be a predictable $D$-valued control such that $\E^P\int_0^T|\alpha_t|^2dt<\infty$. Then \eqref{eq:FWD}, with the law flow replaced by $m$, admits a unique solution $X^m\in\mathbb S^2(P;\R^d)$ and
\[
\E^P\sup_{t\le T}|X_t^m|^2\le C\left[1+\E^P|\xi|^2+
\E^P\int_0^T|\alpha_t|^2dt+\int_0^TM_2(m_t)^2dt\right].
\]
\end{lemma}
\begin{proof}
Fix $P$ and the deterministic flow $m$. Set $X_t^0\equiv\xi$ and, for $n\ge0$, define
\begin{align*}
X_t^{n+1}={}&\xi+
\int_0^t b(s,X_{s-}^n,m_s,\alpha_s)ds
+\int_0^t\sigma(s,X_{s-}^n,m_s,\alpha_s)dB_s\\
&+\int_0^t\int_E
\beta(s,X_{s-}^n,m_s,\alpha_s,e)\widetilde N^P(ds,de).
\end{align*}
We first verify that every Picard iterate belongs to $\mathbb S^2(P;\R^d)$. Assume $X^n\in\mathbb S^2(P;\R^d)$ and write
\[
b_s^n=b(s,X_{s-}^n,m_s,\alpha_s),\qquad
\sigma_s^n=\sigma(s,X_{s-}^n,m_s,\alpha_s),
\]
\[
\beta_s^n(e)=\beta(s,X_{s-}^n,m_s,\alpha_s,e).
\]
Using $(a+b+c+d)^2\le4(a^2+b^2+c^2+d^2)$, Cauchy--Schwarz for the drift, and Lemma~\ref{lem:mart} for the two martingales,
\begin{align*}
\E^P\sup_{r\le t}|X_r^{n+1}|^2
\le{}&4\E^P|\xi|^2
+4t\E^P\int_0^t|b_s^n|^2ds\\
&+16\E^P\int_0^t|\sigma_s^n|_{a_s^P}^2ds
+16\E^P\int_0^t\|\beta_s^n\|_\Pi^2ds.
\end{align*}
The bound $a_s^P\le\overline a$ implies $|\sigma_s^n|_{a_s^P}^2\le C\|\sigma_s^n\|^2$. By \eqref{eq:FLip}--\eqref{eq:Fzero},
\[
|b_s^n|^2+\|\sigma_s^n\|^2+\|\beta_s^n\|_\Pi^2
\le C\bigl(1+|X_{s-}^n|^2+M_2(m_s)^2+|\alpha_s|^2\bigr).
\]
Consequently,
\begin{align}
\E^P\sup_{r\le t}|X_r^{n+1}|^2
\le C\Bigg[&1+\E^P|\xi|^2
+\int_0^tM_2(m_s)^2ds
+\E^P\int_0^t|\alpha_s|^2ds\nonumber\\
&+\int_0^t\E^P\sup_{u\le s}|X_u^n|^2ds\Bigg].
\label{eq:frozen-iterate-S2}
\end{align}
The right-hand side is finite by the induction hypothesis, so $X^{n+1}\in\mathbb S^2(P;\R^d)$. Since $X_t^0=\xi$, the induction starts.

Consider now two consecutive iterates. Before introducing shorthand notation, subtracting the equations for $X^{n+1}$ and $X^n$ yields, for every $r\in[0,T]$, $P$-a.s.,
\begin{align*}
X_r^{n+1}-X_r^n
={}&\int_0^r\Bigl[b(s,X_{s-}^n,m_s,\alpha_s)
-b(s,X_{s-}^{n-1},m_s,\alpha_s)\Bigr]ds\\
&+\int_0^r\Bigl[\sigma(s,X_{s-}^n,m_s,\alpha_s)
-\sigma(s,X_{s-}^{n-1},m_s,\alpha_s)\Bigr]dB_s\\
&+\int_0^r\int_E\Bigl[\beta(s,X_{s-}^n,m_s,\alpha_s,e)
-\beta(s,X_{s-}^{n-1},m_s,\alpha_s,e)\Bigr]\widetilde N^P(ds,de).
\end{align*}
For the subsequent estimates, set
\[
\Delta_t^{n+1}:=X_t^{n+1}-X_t^n
\]
and
\begin{align*}
\Delta b_s^n&:=b(s,X_{s-}^n,m_s,\alpha_s)
-b(s,X_{s-}^{n-1},m_s,\alpha_s),\\
\Delta\sigma_s^n&:=\sigma(s,X_{s-}^n,m_s,\alpha_s)
-\sigma(s,X_{s-}^{n-1},m_s,\alpha_s),\\
\Delta\beta_s^n(e)&:=\beta(s,X_{s-}^n,m_s,\alpha_s,e)
-\beta(s,X_{s-}^{n-1},m_s,\alpha_s,e).
\end{align*}
Thus the preceding identity reads
\begin{align*}
\Delta_r^{n+1}
={}&\int_0^r\Delta b_s^n ds
+\int_0^r\Delta\sigma_s^n dB_s
+\int_0^r\int_E\Delta\beta_s^n(e)\widetilde N^P(ds,de).
\end{align*}
For $t\le T$, the same estimates give
\begin{align*}
\E^P\sup_{r\le t}|\Delta_r^{n+1}|^2
\le{}&3t\E^P\int_0^t|\Delta b_s^n|^2ds
+12\E^P\int_0^t|\Delta\sigma_s^n|_{a_s^P}^2ds\\
&+12\E^P\int_0^t\|\Delta\beta_s^n\|_\Pi^2ds.
\end{align*}
Because the flow $m$ and control $\alpha$ are the same at both iterates, \eqref{eq:FLip}, followed by $a_s^P\le\overline a$, gives
\[
|\Delta b_s^n|^2+|\Delta\sigma_s^n|_{a_s^P}^2
+\|\Delta\beta_s^n\|_\Pi^2
\le C|\Delta_s^n|^2.
\]
Thus, with
\[
q_{n+1}(t):=\E^P\sup_{r\le t}|\Delta_r^{n+1}|^2,
\]
we obtain
\begin{equation}
q_{n+1}(t)
\le C\int_0^t\E^P|\Delta_s^n|^2ds
\le C\int_0^t q_n(s)ds.
\label{eq:frozen-picard-rec}
\end{equation}
The first difference is square-integrable by \eqref{eq:frozen-iterate-S2}; in particular, $q_1(T)<\infty$. We claim by induction that
\begin{equation}
q_{n+1}(t)
\le q_1(T)\frac{(Ct)^n}{n!},
\qquad 0\le t\le T.
\label{eq:frozen-factorial}
\end{equation}
For $n=0$, this is immediate. If it holds at rank $n-1$, then \eqref{eq:frozen-picard-rec} gives
\[
q_{n+1}(t)
\le Cq_1(T)\int_0^t\frac{(Cs)^{n-1}}{(n-1)!}ds
=q_1(T)\frac{(Ct)^n}{n!}.
\]
Hence
\[
\sum_{n\ge0}\|X^{n+1}-X^n\|_{\mathbb S^2(P)}
\le q_1(T)^{1/2}
\sum_{n\ge0}\frac{(CT)^{n/2}}{\sqrt{n!}}<\infty.
\]
The numerical series converges by the ratio test. Since $\mathbb S^2(P;\R^d)$ is complete, $(X^n)$ converges to some $X^m\in\mathbb S^2(P;\R^d)$. For the limit, set
\[
b_s^m=b(s,X_{s-}^m,m_s,\alpha_s),\qquad
\sigma_s^m=\sigma(s,X_{s-}^m,m_s,\alpha_s),
\]
\[
\beta_s^m(e)=\beta(s,X_{s-}^m,m_s,\alpha_s,e).
\]

It remains to pass to the limit in the equation. Lipschitz continuity yields
\begin{align*}
\E^P\int_0^T|b(s,X_{s-}^n,m_s,\alpha_s)-b(s,X_{s-}^m,m_s,\alpha_s)|^2ds
&\le C T\,\E^P\sup_{t\le T}|X_t^n-X_t^m|^2,\\
\E^P\int_0^T\|\sigma(s,X_{s-}^n,m_s,\alpha_s)-\sigma(s,X_{s-}^m,m_s,\alpha_s)\|^2ds
&\le C T\,\E^P\sup_{t\le T}|X_t^n-X_t^m|^2,\\
\E^P\int_0^T\|\beta(s,X_{s-}^n,m_s,\alpha_s,\cdot)-\beta(s,X_{s-}^m,m_s,\alpha_s,\cdot)\|_\Pi^2ds
&\le C T\,\E^P\sup_{t\le T}|X_t^n-X_t^m|^2,
\end{align*}
which tend to zero. For the drift,
\[
\E^P\sup_{t\le T}\left|\int_0^t(b_s^n-b_s^m)ds\right|^2
\le T\E^P\int_0^T|b_s^n-b_s^m|^2ds\longrightarrow0.
\]
Lemma~\ref{lem:mart} similarly gives
\[
\E^P\sup_{t\le T}\left|\int_0^t(\sigma_s^n-\sigma_s^m)dB_s\right|^2\longrightarrow0,
\]
and
\[
\E^P\sup_{t\le T}\left|\int_0^t\int_E
(\beta_s^n(e)-\beta_s^m(e))\widetilde N^P(ds,de)\right|^2\longrightarrow0.
\]
Passing to the limit in the Picard identity shows that $X^m$ solves the frozen-law equation.

If $X$ and $\bar X$ are two solutions, the preceding difference estimate gives
\[
\E^P\sup_{r\le t}|X_r-\bar X_r|^2
\le C\int_0^t\E^P\sup_{u\le s}|X_u-\bar X_u|^2ds.
\]
Gronwall's inequality implies $X=\bar X$ up to indistinguishability.

Finally, apply directly to $X^m$ the estimate leading to \eqref{eq:frozen-iterate-S2}. We obtain
\begin{align*}
\E^P\sup_{r\le t}|X_r^m|^2
\le C\Bigg[&1+\E^P|\xi|^2
+\E^P\int_0^t|\alpha_s|^2ds
+\int_0^tM_2(m_s)^2ds\\
&+\int_0^t\E^P\sup_{u\le s}|X_u^m|^2ds\Bigg].
\end{align*}
Gronwall's inequality yields
\[
\E^P\sup_{t\le T}|X_t^m|^2
\le C\left[1+\E^P|\xi|^2+
\E^P\int_0^T|\alpha_t|^2dt+
\int_0^TM_2(m_t)^2dt\right].
\]
The constants are uniform in $P$ because only the common structural bounds on $a^P$, the coefficients, and the compensator $\Pi(de)dt$ were used.
\end{proof}

\begin{theorem}[Well-posedness and moments of the forward equation]\label{thm:Fwell}
Let $\alpha\in\mathcal A^2$. Under Assumptions~\ref{HF} and~\ref{HD}, for every $P\in\Pcal$, equation \eqref{eq:FWD} admits a unique strong c\`adl\`ag solution $X^P\in\mathbb S^2(P;\R^d)$. If Assumption~\ref{HqJ} also holds, then
\begin{equation}
\sup_{P\in\Pcal}\E^P\sup_{t\le T}|X_t^P|^q
\le C\left(1+\sup_{P\in\Pcal}\E^P|\xi|^q
+\sup_{P\in\Pcal}\E^P\int_0^T|\alpha_t|^qdt\right).
\label{eq:Fq}
\end{equation}
\end{theorem}

\begin{proof}
\noindent\emph{Step 1: frozen-law equation and fixed point on the law flow.}
Fix $P$. For $m\in C([0,T];\mathcal P_2(\R^d))$, Lemma~\ref{lem:frozenfull} constructs a unique solution $X^m$. Define
\[
\Phi_P(m)_t:=\Lcal^P(X_t^m).
\]
For $0\le s<t\le T$, stochastic isometries and linear growth give
\[
\E^P|X_t^m-X_s^m|^2
\le C\E^P\int_s^t\{1+|X_r^m|^2+M_2(m_r)^2+|\alpha_r|^2\}dr,
\]
which tends to zero as $t-s\to0$. The canonical coupling then yields
\[
W_2^2(\Phi_P(m)_t,\Phi_P(m)_s)
\le\E^P|X_t^m-X_s^m|^2\to0;
\]
hence $\Phi_P(m)\in C([0,T];\mathcal P_2(\R^d))$.

Equip this space with
\[
d_\vartheta(m,\bar m)^2:=\sup_{t\le T}e^{-\vartheta t}W_2^2(m_t,\bar m_t),
\qquad \vartheta>0.
\]
This metric is equivalent to the uniform $W_2$ metric, and the space is complete. For two flows $m,\bar m$, Lemma~\ref{lem:Fstab}, with the same initial condition and the same control, gives
\[
\E^P\sup_{r\le t}|X_r^m-X_r^{\bar m}|^2
\le C\int_0^tW_2^2(m_s,\bar m_s)ds.
\]
By the natural coupling,
\begin{align*}
e^{-\vartheta t}W_2^2(\Phi_P(m)_t,\Phi_P(\bar m)_t)
&\le Ce^{-\vartheta t}\int_0^tW_2^2(m_s,\bar m_s)ds\\
&\le C d_\vartheta(m,\bar m)^2\int_0^te^{-\vartheta(t-s)}ds\\
&\le \frac{C}{\vartheta}d_\vartheta(m,\bar m)^2.
\end{align*}
Taking the supremum over $t$,
\[
d_\vartheta(\Phi_P(m),\Phi_P(\bar m))^2
\le\frac{C}{\vartheta}d_\vartheta(m,\bar m)^2.
\]
Choose $\vartheta>C$. Then $\Phi_P$ is a contraction, and Banach's fixed-point theorem gives a unique $\mu^P\in C([0,T];\mathcal P_2(\R^d))$ such that $\Phi_P(\mu^P)=\mu^P$. The process $X^P:=X^{\mu^P}$ satisfies $\mu_t^P=\Lcal^P(X_t^P)$ and solves \eqref{eq:FWD}. If $X$ and $\bar X$ are two McKean--Vlasov solutions, their law flows are fixed points of $\Phi_P$ and therefore coincide; uniqueness for the frozen-law equation then gives $X=\bar X$. The solution is strong because every Picard iterate of the frozen-law equation, and hence its limit, is adapted to the filtration generated by $(\xi,B,N)$.

\noindent\emph{Step 2: uniform quadratic estimate.}
The a priori estimate from Lemma~\ref{lem:frozenfull}, applied at the fixed point $m=\mu^P$, yields
\begin{align*}
\E^P\sup_{r\le t}|X_r^P|^2
\le{}&C\left(1+\E^P|\xi|^2+\E^P\int_0^t|\alpha_s|^2ds
+\int_0^tM_2(\mu_s^P)^2ds\right)\\
&+C\int_0^t\E^P\sup_{u\le s}|X_u^P|^2ds.
\end{align*}
Since
\[
M_2(\mu_s^P)^2=\E^P|X_s^P|^2
\le\E^P\sup_{u\le s}|X_u^P|^2,
\]
this term can be absorbed into the Gronwall integral, giving
\[
\E^P\sup_{r\le t}|X_r^P|^2
\le C\left(1+\E^P|\xi|^2+\E^P\int_0^T|\alpha_s|^2ds\right)
+C\int_0^t\E^P\sup_{u\le s}|X_u^P|^2ds.
\]
Gronwall's inequality followed by the supremum over $P$ gives
\begin{equation}
\sup_P\E^P\sup_{t\le T}|X_t^P|^2
\le C\left(1+\sup_P\E^P|\xi|^2+
\sup_P\E^P\int_0^T|\alpha_t|^2dt\right).
\label{eq:Fsecond}
\end{equation}
The constant is uniform in $P$ because all constants used in Lemma~\ref{lem:frozenfull} are structural and common to the family.

\noindent\emph{Step 3: moments of order $q$.}
Fix $q>4$ as in Assumption~\ref{HqJ}. Before applying Gronwall to the limiting process, the finiteness of its $q$th moment must first be justified. We therefore return to the fixed-point iteration on law flows.

Set
\[
m_t^0:=\Lcal^P(\xi),\qquad 0\le t\le T.
\]
This flow is constant and belongs to $C([0,T];\mathcal P_q(\R^d))$ because $\E^P|\xi|^q<\infty$. For $n\ge0$, define recursively
\[
X^{n+1}:=X^{m^n},\qquad
m_t^{n+1}:=\Lcal^P(X_t^{n+1}),
\]
where $X^{m^n}$ is the frozen-law solution constructed in Lemma~\ref{lem:frozenfull}.

We prove by induction that each $X^{n+1}$ has a uniform $q$th moment. Assume
\[
\sup_{t\le T}M_q(m_t^n)^q<\infty.
\]
For $R\ge1$, set
\[
\tau_R^{n+1}:=\inf\{t\ge0:|X_t^{n+1}|\ge R\}\wedge T,
\qquad
u_{n+1,R}(t):=\E^P\sup_{s\le t}|X_{s\wedge\tau_R^{n+1}}^{n+1}|^q.
\]
On $[0,\tau_R^{n+1}]$, the coefficients have the required integrability. Applying H\"older's inequality to the drift, BDG to the continuous martingale, Lemma~\ref{lem:martp} to the Poisson martingale, and then the growth assumptions \ref{HF} and \ref{HqJ}, we obtain
\begin{align}
u_{n+1,R}(t)
\le{}&C_q\left(1+\E^P|\xi|^q+
\E^P\int_0^T|\alpha_s|^qds\right)\nonumber\\
&+C_q\int_0^t u_{n+1,R}(s)\,ds
+C_q\int_0^t M_q(m_s^n)^q\,ds .
\label{eq:Fq-frozen-local}
\end{align}
The constant $C_q$ is independent of $R$, $n$, and $P$. The continuous part is controlled by the common upper bound $a^P\le\overline a$, while the jump part uses the common compensator $\Pi(de)dt$.

Applying Gronwall to \eqref{eq:Fq-frozen-local} yields
\begin{equation}
u_{n+1,R}(t)
\le C_q A_P+C_q\int_0^tM_q(m_s^n)^q\,ds,
\qquad
A_P:=1+\E^P|\xi|^q+\E^P\int_0^T|\alpha_s|^qds,
\label{eq:Fq-frozen-G}
\end{equation}
after modifying $C_q$ by a factor depending only on $T$. Letting $R\uparrow\infty$ and using Fatou's lemma gives
\begin{equation}
u_{n+1}(t):=\E^P\sup_{s\le t}|X_s^{n+1}|^q
\le C_q A_P+C_q\int_0^tM_q(m_s^n)^q\,ds.
\label{eq:Fq-frozen-global}
\end{equation}

For $n\ge1$,
\[
M_q(m_s^n)^q=\E^P|X_s^n|^q\le u_n(s),
\]
and for $n=0$,
\[
M_q(m_s^0)^q=\E^P|\xi|^q\le A_P.
\]
Thus \eqref{eq:Fq-frozen-global} implies
\[
u_{n+1}(t)\le C_qA_P+C_q\int_0^tu_n(s)\,ds.
\]
Iterating this inequality gives
\begin{align*}
u_{n+1}(t)
&\le C_qA_P\sum_{j=0}^{n}\frac{(C_qt)^j}{j!}
+A_P\frac{(C_qt)^{n+1}}{(n+1)!}\\
&\le C_qA_Pe^{C_qT},
\qquad 0\le t\le T.
\end{align*}
Consequently,
\begin{equation}
\sup_{n\ge1}\E^P\sup_{t\le T}|X_t^n|^q
\le C_q\left(1+\E^P|\xi|^q+
\E^P\int_0^T|\alpha_t|^qdt\right).
\label{eq:Fq-iter-uniform}
\end{equation}

It remains to transfer this estimate to the fixed point $\mu^P$. Step~1 shows that $m^n=\Phi_P^n(m^0)$ converges to $\mu^P$ in $d_\vartheta$, hence uniformly in time in $W_2$. Lemma~\ref{lem:Fstab}, applied to the equations frozen at $m^n$ and $\mu^P$, gives
\[
\E^P\sup_{t\le T}|X_t^{n+1}-X_t^P|^2
\le C\int_0^TW_2^2(m_t^n,\mu_t^P)\,dt
\longrightarrow0.
\]
We can therefore extract a subsequence $(n_j)$ such that
\[
\sup_{t\le T}|X_t^{n_j+1}-X_t^P|\longrightarrow0
\qquad P\text{-a.s.}
\]
By Fatou's lemma and \eqref{eq:Fq-iter-uniform},
\[
\E^P\sup_{t\le T}|X_t^P|^q
\le\liminf_{j\to\infty}
\E^P\sup_{t\le T}|X_t^{n_j+1}|^q
\le C_qA_P.
\]
All constants are common to $P\in\Pcal$; taking the supremum over $P$ proves \eqref{eq:Fq}.

To complete the uniqueness argument in Step~1, note that every solution $X\in\mathbb S^2(P;\R^d)$ has a $W_2$-continuous law flow: the increment estimate used for $\Phi_P(m)$ applies directly to $X$. Thus the law flow of any solution in the stated class is indeed a fixed point of $\Phi_P$, which completes the uniqueness argument.
\end{proof}

\begin{remark}[Role of higher moments]
Assumption~\ref{HqJ} is not needed for existence. It only provides the higher moments used in the Yosida estimates.
\end{remark}

\subsection{Existence and uniqueness of the backward equation with jumps}

Well-posedness under discontinuous martingales is studied in the preprint by Papapantoleon, Saplaouras, and Theodorakopoulos \cite[Section~4]{PapapantoleonEtAl2024}. For fixed $P$ and $X^P$, set
\begin{align}
F^P(t,\omega,y,z,u,\eta)
&=f(t,X_{t-}^P(\omega),\mu_{t-}^P,y,z,u,\eta,\alpha_t(\omega))
+\chi_t^P(\omega),\nonumber\\
\chi_t^P&=\langle\varrho,A^0(X_{t-}^P,\mu_{t-}^P,\alpha_t)\rangle.
\label{eq:exogenousgenerator}
\end{align}
For $v=(y,z,u,\eta)$ and $v'=(y',z',u',\eta')$, Assumption~\ref{HB} gives
\[
|F^P(t,v)-F^P(t,v')|
\le L_B\{|y-y'|+|z-z'|+\|u-u'\|_\Pi+W_2(\eta,\eta')\}.
\]
The source $\chi^P$ disappears from this difference: once $X^P$ is fixed, it is exogenous to the backward Picard iteration.

\begin{theorem}[Existence and uniqueness of the BSDE with jumps along the forward process]\label{thm:Bwell}
Let $\alpha\in\mathcal A^2$. Assume \ref{HF}, \ref{HB}, \ref{HD}, and the representation property \eqref{eq:PRP}. Fix $P\in\Pcal$, let $X^P$ be the solution from Theorem~\ref{thm:Fwell}, and set $\mu_t^P=\Lcal^P(X_t^P)$,
\[
\chi_t^P=\langle\varrho,A^0(X_{t-}^P,\mu_{t-}^P,\alpha_t)\rangle,
\qquad \mathsf K_t^P=\int_0^t\chi_s^Pds.
\]
Then there exists a unique triplet
\[
(Y^P,Z^P,U^P)\in\mathbb S^2(P)\times\mathbb H_B^2(P)\times\mathbb H_\Pi^2(P)
\]
such that, for every $t\in[0,T]$, $P$-a.s.,
\begin{align}
Y_t^P={}&g(X_T^P,\mu_T^P)+\int_t^T
f(s,X_{s-}^P,\mu_{s-}^P,Y_s^P,Z_s^P,U_s^P,
\Lcal^P(Y_s^P),\alpha_s)ds\nonumber\\
&+\mathsf K_T^P-\mathsf K_t^P-
\int_t^TZ_s^PdB_s-\int_t^T\int_EU_s^P(e)\widetilde N^P(ds,de).
\label{eq:Bexist}
\end{align}
Moreover,
\begin{align}
\sup_{P\in\Pcal}\E^P\Bigl[\sup_{t\le T}|Y_t^P|^2+
\sup_{t\le T}|\mathsf K_t^P|^2+
\int_0^T\{|Z_t^P|_{a_t^P}^2+\|U_t^P\|_\Pi^2\}dt\Bigr]<\infty.
\label{eq:Bexistbound}
\end{align}
Uniqueness of $Y^P$ is understood up to indistinguishability, and uniqueness of $Z^P,U^P$ in their energy norms.
\end{theorem}

\begin{proof}
Fix $P\in\Pcal$ and, throughout this proof, write
\[
X=X^P,\qquad \mu=\mu^P,\qquad
\zeta=g(X_T,\mu_T),\qquad
\chi_t=\langle\varrho,A^0(X_{t-},\mu_{t-},\alpha_t)\rangle.
\]
The process $X$ and the source $\chi$ remain fixed throughout the backward iteration.

\medskip
\noindent\emph{Step 1: integrability of the backward data.}
By \eqref{eq:gLip},
\begin{align*}
|\zeta|
&=|g(X_T,\mu_T)|\\
&\le |g(0,\delta_0)|
+L_B\bigl(|X_T|+W_2(\mu_T,\delta_0)\bigr)\\
&\le C_B+L_B\bigl(|X_T|+M_2(\mu_T)\bigr).
\end{align*}
Since $M_2(\mu_T)^2=\E^P|X_T|^2$ is deterministic under the fixed model,
\begin{align}
\E^P|\zeta|^2
&\le 3C_B^2+3L_B^2\E^P|X_T|^2
+3L_B^2M_2(\mu_T)^2\nonumber\\
&\le C\left(1+\E^P\sup_{t\le T}|X_t|^2\right)<\infty.
\label{eq:Bpicard-zeta}
\end{align}

Next set
\[
f_t^0:=f(t,X_{t-},\mu_{t-},0,0,0,\delta_0,\alpha_t).
\]
Comparing $f_t^0$ with $f(t,0,\delta_0,0,0,0,\delta_0,0)$ and using \eqref{eq:fLip}--\eqref{eq:fzero},
\begin{align*}
|f_t^0|
&\le C_B+L_B\bigl(|X_{t-}|+W_2(\mu_{t-},\delta_0)+|\alpha_t|\bigr)\\
&= C_B+L_B\bigl(|X_{t-}|+M_2(\mu_{t-})+|\alpha_t|\bigr).
\end{align*}
By $(a+b+c+d)^2\le4(a^2+b^2+c^2+d^2)$,
\begin{align}
\E^P\int_0^T|f_t^0|^2dt
&\le C\left[T+\E^P\int_0^T|X_{t-}|^2dt
+\int_0^TM_2(\mu_{t-})^2dt
+\E^P\int_0^T|\alpha_t|^2dt\right]\nonumber\\
&\le C\left[1+\E^P\sup_{t\le T}|X_t|^2
+\E^P\int_0^T|\alpha_t|^2dt\right]<\infty,
\label{eq:Bpicard-f0}
\end{align}
where we used $M_2(\mu_t)^2=\E^P|X_t|^2\le\E^P\sup_{r\le T}|X_r|^2$.

Finally, by the definition of $h$ and the linear growth of $A^0$,
\begin{align*}
|\chi_t|^2
&\le |\varrho|^2|A^0(X_{t-},\mu_{t-},\alpha_t)|^2\\
&\le 3|\varrho|^2C_A^2\bigl(1+|X_{t-}|^2+|\alpha_t|^2\bigr),
\end{align*}
so that
\begin{align}
\E^P\int_0^T|\chi_t|^2dt
\le C\left[1+\E^P\sup_{t\le T}|X_t|^2
+\E^P\int_0^T|\alpha_t|^2dt\right]<\infty.
\label{eq:Bpicard-h}
\end{align}
Combining \eqref{eq:Bpicard-zeta}, \eqref{eq:Bpicard-f0}, \eqref{eq:Bpicard-h}, and \eqref{eq:Fsecond}, we obtain
\begin{equation}
\E^P|\zeta|^2+
\E^P\int_0^T\bigl(|f_t^0|^2+|\chi_t|^2\bigr)dt\le C,
\label{eq:Bpicard-data}
\end{equation}
with $C$ independent of $P$. Set
\[
\ell_t:=|f_t^0|+|\chi_t|.
\]
Then
\begin{equation}
\E^P\int_0^T\ell_t^2dt
\le2\E^P\int_0^T\bigl(|f_t^0|^2+|\chi_t|^2\bigr)dt<\infty.
\label{eq:Bpicard-ell}
\end{equation}

\medskip
\noindent\emph{Step 2: definition of the Picard iteration.}
Initialize
\[
(Y^0,Z^0,U^0)=(0,0,0).
\]
Assume $(Y^n,Z^n,U^n)$ has been constructed and set
\begin{align}
G_t^n:={}&f(t,X_{t-},\mu_{t-},Y_t^n,Z_t^n,U_t^n,
\Lcal^P(Y_t^n),\alpha_t)+\chi_t.
\label{eq:Bpicard-Gn}
\end{align}
Lemma~\ref{lem:lawmeas} allows us to choose a Borel representative of the flow $t\mapsto\Lcal^P(Y_t^n)$; hence the generator above is well defined on equivalence classes in $L^2(dt\otimes P)$. Comparing with the free term $f_t^0$, the Lipschitz condition \ref{HB} gives
\begin{align*}
|G_t^n|
&\le \ell_t+L_B\Bigl(
|Y_t^n|+|Z_t^n|+\|U_t^n\|_\Pi
+W_2(\Lcal^P(Y_t^n),\delta_0)\Bigr).
\end{align*}
The lower ellipticity bound in \eqref{eq:qv} gives
\[
|Z_t^n|^2\le \lambda_{\min}(\underline a)^{-1}|Z_t^n|_{a_t^P}^2,
\]
and
\[
W_2^2(\Lcal^P(Y_t^n),\delta_0)=\E^P|Y_t^n|^2.
\]
Consequently, with $c_a:=\lambda_{\min}(\underline a)^{-1}$,
\begin{align}
\E^P\int_0^T|G_t^n|^2dt
\le C\Bigg(&\E^P\int_0^T\ell_t^2dt
+\E^P\int_0^T|Y_t^n|^2dt
+c_a\E^P\int_0^T|Z_t^n|_{a_t^P}^2dt\nonumber\\
&+\E^P\int_0^T\|U_t^n\|_\Pi^2dt
+\int_0^T\E^P|Y_t^n|^2dt\Bigg)<\infty.
\label{eq:Bpicard-Gn-L2}
\end{align}
Thus $G^n\in L^2(dt\otimes P)$. By Cauchy--Schwarz,
\begin{align}
\E^P\left|\int_0^TG_s^n ds\right|^2
\le T\E^P\int_0^T|G_s^n|^2ds<\infty,
\label{eq:Bpicard-timeint}
\end{align}
so
\[
\Xi^{n+1}:=\zeta+\int_0^TG_s^n ds\in L^2(P).
\]
Define the square-integrable martingale
\[
M_t^{n+1}:=\E^P[\Xi^{n+1}\mid\mathcal F_t^P].
\]
By Doob's inequality,
\begin{equation}
\E^P\sup_{t\le T}|M_t^{n+1}|^2
\le4\E^P|\Xi^{n+1}|^2<\infty.
\label{eq:Bpicard-Mdoob}
\end{equation}
The representation property \eqref{eq:PRP}, applied to the centered martingale $M^{n+1}-M_0^{n+1}$, provides unique processes
\[
Z^{n+1}\in\mathbb H_B^2(P),\qquad
U^{n+1}\in\mathbb H_\Pi^2(P)
\]
such that
\begin{equation}
M_t^{n+1}=M_0^{n+1}
+\int_0^t Z_s^{n+1}dB_s
+\int_0^t\int_EU_s^{n+1}(e)\widetilde N^P(ds,de).
\label{eq:Bpicard-Mrep}
\end{equation}
Orthogonality of the continuous and purely discontinuous parts also gives
\begin{align}
\E^P\int_0^T\Bigl(|Z_s^{n+1}|_{a_s^P}^2
+\|U_s^{n+1}\|_\Pi^2\Bigr)ds
&=\E^P|M_T^{n+1}-M_0^{n+1}|^2\nonumber\\
&\le \E^P|\Xi^{n+1}|^2<\infty.
\label{eq:Bpicard-rep-energy}
\end{align}
Finally, set
\[
Y_t^{n+1}:=M_t^{n+1}-\int_0^tG_s^n ds.
\]
Then $Y^{n+1}$ is adapted and c\`adl\`ag. Moreover,
\begin{align}
\E^P\sup_{t\le T}|Y_t^{n+1}|^2
&\le2\E^P\sup_{t\le T}|M_t^{n+1}|^2
+2\E^P\sup_{t\le T}\left|\int_0^tG_s^n ds\right|^2\nonumber\\
&\le8\E^P|\Xi^{n+1}|^2
+2T\E^P\int_0^T|G_s^n|^2ds<\infty.
\label{eq:Bpicard-Yn-S2}
\end{align}
Thus $Y^{n+1}\in\mathbb S^2(P)$. Using $M_T^{n+1}=\Xi^{n+1}$ in \eqref{eq:Bpicard-Mrep}, we obtain, for every $t\in[0,T]$,
\begin{align}
Y_t^{n+1}={}&\zeta+\int_t^TG_s^n ds
-\int_t^TZ_s^{n+1}dB_s
-\int_t^T\int_EU_s^{n+1}(e)\widetilde N^P(ds,de).
\label{eq:Bpicard-iteration}
\end{align}
The estimates \eqref{eq:Bpicard-Gn-L2}--\eqref{eq:Bpicard-Yn-S2} show by induction that every Picard iterate is well defined in
$\mathbb S^2(P)\times\mathbb H_B^2(P)\times\mathbb H_\Pi^2(P)$.

At this point every iterate exists. It remains to show that the Picard map is contractive. We first write the difference equation explicitly so that the terms that disappear are visible: neither the terminal condition nor the monotone source contributes to the contraction constant.

\medskip
\noindent\emph{Step 3: equation for two successive Picard iterates.}
The terminal condition $\zeta$ is the same at every rank and the source $\chi$ is exogenous to the iteration. Subtracting the equations for ranks $n+1$ and $n$ therefore gives, before any shorthand is introduced,
\begin{align*}
Y_t^{n+1}-Y_t^n
={}&\int_t^T\Bigl[G_s^n-G_s^{n-1}\Bigr]ds
-\int_t^T\bigl(Z_s^{n+1}-Z_s^n\bigr)dB_s\\
&-\int_t^T\int_E\bigl(U_s^{n+1}(e)-U_s^n(e)\bigr)\widetilde N^P(ds,de),
\qquad t\in[0,T],\quad P\text{-a.s.}
\end{align*}
Moreover, since the forward state, its law, the control, and the source are fixed throughout the Picard iteration,
\begin{align*}
G_s^n-G_s^{n-1}
={}&f(s,X_{s-},\mu_{s-},Y_s^n,Z_s^n,U_s^n,\Lcal^P(Y_s^n),\alpha_s)\\
&-f(s,X_{s-},\mu_{s-},Y_s^{n-1},Z_s^{n-1},U_s^{n-1},\Lcal^P(Y_s^{n-1}),\alpha_s).
\end{align*}
Thus the cancellation of both $\zeta$ and $\chi$ is explicit. For the remaining estimates, for $n\ge1$, $t\in[0,T]$, and $e\in E$, set
\[
\Delta Y_t^{n+1}:=Y_t^{n+1}-Y_t^n,\qquad
\Delta Z_t^{n+1}:=Z_t^{n+1}-Z_t^n,\qquad
\Delta U_t^{n+1}(e):=U_t^{n+1}(e)-U_t^n(e),
\]
and
\[
\Delta G_t^n:=G_t^n-G_t^{n-1}.
\]
The identity for $\Delta Y^{n+1}$ is understood up to indistinguishability under $P$; those for $\Delta Z^{n+1}$ and $\Delta G^n$ are understood $dt\otimes P$-a.e., and that for $\Delta U^{n+1}$ is understood $dt\otimes P\otimes\Pi$-a.e. With this notation, for every $t\in[0,T]$, $P$-a.s.,
\begin{align}
\Delta Y_t^{n+1}
={}&\int_t^T\Delta G_s^n ds
-\int_t^T\Delta Z_s^{n+1}dB_s
-\int_t^T\int_E\Delta U_s^{n+1}(e)\widetilde N^P(ds,de).
\label{eq:Bpicard-difference}
\end{align}
Moreover, $\Delta Y_T^{n+1}=0$, $P$-a.s. In differential form,
\begin{equation}
d\Delta Y_t^{n+1}
=-\Delta G_t^n dt
+\Delta Z_t^{n+1}dB_t
+\int_E\Delta U_t^{n+1}(e)\widetilde N^P(dt,de).
\label{eq:Bpicard-differential}
\end{equation}

\medskip
\noindent\emph{Step 4: weighted It\^o formula with jumps.}
Fix $\gamma>0$. Apply It\^o's formula to $e^{\gamma t}|\Delta Y_t^{n+1}|^2$ on $[t,T]$. After localization,
\begin{align*}
&e^{\gamma t}|\Delta Y_t^{n+1}|^2
+\int_t^Te^{\gamma s}\Bigl(
\gamma|\Delta Y_s^{n+1}|^2
+|\Delta Z_s^{n+1}|_{a_s^P}^2\Bigr)ds\\
&\quad+\int_t^T\int_Ee^{\gamma s}
|\Delta U_s^{n+1}(e)|^2N^P(ds,de)\\
&=2\int_t^Te^{\gamma s}\Delta Y_{s-}^{n+1}\Delta G_s^n ds
-2\int_t^Te^{\gamma s}\Delta Y_{s-}^{n+1}\Delta Z_s^{n+1}dB_s\\
&\quad-2\int_t^T\int_Ee^{\gamma s}
\Delta Y_{s-}^{n+1}\Delta U_s^{n+1}(e)\widetilde N^P(ds,de).
\end{align*}
The last two integrals are local martingales. Stop the processes when $\sup|\Delta Y^{n+1}|$ and the energies of $\Delta Z^{n+1},\Delta U^{n+1}$ exceed a level $k$. After stopping, the martingales are square-integrable and have zero expectation. Letting $k\to\infty$, BDG estimates and square integrability justify the limit. By compensation of the Poisson random measure,
\[
\E^P\int_t^T\int_Ee^{\gamma s}
|\Delta U_s^{n+1}(e)|^2N^P(ds,de)
=
\E^P\int_t^Te^{\gamma s}\|\Delta U_s^{n+1}\|_\Pi^2ds.
\]
Therefore,
\begin{align}
&\E^P e^{\gamma t}|\Delta Y_t^{n+1}|^2
+\E^P\int_t^Te^{\gamma s}\Bigl(
\gamma|\Delta Y_s^{n+1}|^2
+|\Delta Z_s^{n+1}|_{a_s^P}^2
+\|\Delta U_s^{n+1}\|_\Pi^2\Bigr)ds
\nonumber\\
&\qquad=2\E^P\int_t^Te^{\gamma s}
\Delta Y_s^{n+1}\Delta G_s^n ds.
\label{eq:Bpicard-Ito}
\end{align}

\medskip
\noindent\emph{Step 5: generator estimate and law term.}
By Assumption~\ref{HB},
\begin{align}
|\Delta G_s^n|
\le L_B\Bigl(&|\Delta Y_s^n|+|\Delta Z_s^n|
+\|\Delta U_s^n\|_\Pi
+W_2(\Lcal^P(Y_s^n),\Lcal^P(Y_s^{n-1}))\Bigr).
\label{eq:Bpicard-Gdiff}
\end{align}
The natural coupling $(Y_s^n,Y_s^{n-1})$ gives
\begin{equation}
W_2^2(\Lcal^P(Y_s^n),\Lcal^P(Y_s^{n-1}))
\le \E^P|\Delta Y_s^n|^2.
\label{eq:Bpicard-law}
\end{equation}
Moreover, the uniform ellipticity bound \eqref{eq:qv} implies
\[
|\Delta Z_s^n|^2\le \lambda_{\min}(\underline a)^{-1}|\Delta Z_s^n|_{a_s^P}^2.
\]
Using $(a+b+c+d)^2\le4(a^2+b^2+c^2+d^2)$ and then \eqref{eq:Bpicard-law}, we obtain, for a constant $C_0$ depending only on $\lambda_{\min}(\underline a)^{-1}$,
\begin{align}
\E^P|\Delta G_s^n|^2
\le C_0L_B^2\E^P\Bigl(
|\Delta Y_s^n|^2
+|\Delta Z_s^n|_{a_s^P}^2
+\|\Delta U_s^n\|_\Pi^2\Bigr).
\label{eq:Bpicard-Gsquare}
\end{align}

\medskip
\noindent\emph{Step 6: Young's inequality and contraction in the energy norm.}
In \eqref{eq:Bpicard-Ito}, apply
\[
2ab\le \frac{\gamma}{2}a^2+\frac{2}{\gamma}b^2
\]
with $a=|\Delta Y_s^{n+1}|$ and $b=|\Delta G_s^n|$. Absorbing $\frac\gamma2|\Delta Y^{n+1}|^2$ into the left-hand side and taking $t=0$ gives
\begin{align}
&\E^P\int_0^Te^{\gamma s}\Bigl(
\frac\gamma2|\Delta Y_s^{n+1}|^2
+|\Delta Z_s^{n+1}|_{a_s^P}^2
+\|\Delta U_s^{n+1}\|_\Pi^2\Bigr)ds
\nonumber\\
&\qquad\le \frac{2}{\gamma}\E^P\int_0^Te^{\gamma s}|\Delta G_s^n|^2ds.
\label{eq:Bpicard-young}
\end{align}
Introduce the norm
\begin{equation}
\|(y,z,u)\|_{\gamma,P}^2
:=\E^P\int_0^Te^{\gamma s}
\Bigl(\frac\gamma2|y_s|^2+|z_s|_{a_s^P}^2+\|u_s\|_\Pi^2\Bigr)ds.
\label{eq:Bpicard-norm}
\end{equation}
To relate \eqref{eq:Bpicard-Gsquare} precisely to \eqref{eq:Bpicard-norm}, note that for $\gamma\ge1$,
\begin{align*}
\E^P\int_0^Te^{\gamma s}|\Delta Y_s^n|^2ds
&\le \frac{2}{\gamma}
\|(\Delta Y^n,\Delta Z^n,\Delta U^n)\|_{\gamma,P}^2,\\
\E^P\int_0^Te^{\gamma s}|\Delta Z_s^n|_{a_s^P}^2ds
&\le \|(\Delta Y^n,\Delta Z^n,\Delta U^n)\|_{\gamma,P}^2,\\
\E^P\int_0^Te^{\gamma s}\|\Delta U_s^n\|_\Pi^2ds
&\le \|(\Delta Y^n,\Delta Z^n,\Delta U^n)\|_{\gamma,P}^2.
\end{align*}
Multiplying \eqref{eq:Bpicard-Gsquare} by $e^{\gamma s}$ and integrating, there exists a structural constant $C_1$, independent of $P,n$, and $\gamma\ge1$, such that
\begin{equation}
\E^P\int_0^Te^{\gamma s}|\Delta G_s^n|^2ds
\le C_1L_B^2
\|(\Delta Y^n,\Delta Z^n,\Delta U^n)\|_{\gamma,P}^2.
\label{eq:Bpicard-Gweighted}
\end{equation}
Substituting \eqref{eq:Bpicard-Gweighted} into \eqref{eq:Bpicard-young} gives
\begin{equation}
\|(\Delta Y^{n+1},\Delta Z^{n+1},\Delta U^{n+1})\|_{\gamma,P}^2
\le \frac{2C_1L_B^2}{\gamma}
\|(\Delta Y^n,\Delta Z^n,\Delta U^n)\|_{\gamma,P}^2.
\label{eq:Bpicard-contraction}
\end{equation}
Choose
\[
\gamma>2C_1L_B^2
\qquad\text{and set}\qquad
q:=\frac{2C_1L_B^2}{\gamma}\in(0,1).
\]
Then, recursively,
\begin{equation}
\|(\Delta Y^{n+1},\Delta Z^{n+1},\Delta U^{n+1})\|_{\gamma,P}^2
\le q^n\|(\Delta Y^1,\Delta Z^1,\Delta U^1)\|_{\gamma,P}^2.
\label{eq:Bpicard-geometric}
\end{equation}
The right-hand side at rank $1$ is finite by \eqref{eq:Bpicard-data} and the construction in Step~2.

\medskip
\noindent\emph{Step 7: supremum estimate and the Cauchy property in $\mathbb S^2$.}
Equation \eqref{eq:Bpicard-difference} gives
\begin{align*}
\sup_{t\le T}|\Delta Y_t^{n+1}|^2
\le{}&3T\int_0^T|\Delta G_s^n|^2ds
+3\sup_{t\le T}\left|\int_t^T\Delta Z_s^{n+1}dB_s\right|^2\\
&+3\sup_{t\le T}\left|\int_t^T\int_E\Delta U_s^{n+1}(e)\widetilde N^P(ds,de)\right|^2.
\end{align*}
For a martingale $M$ with $M_0=0$, $\sup_{t\le T}|M_T-M_t|\le2\sup_{t\le T}|M_t|$. Doob--BDG inequalities and the Poisson isometry therefore give
\begin{align}
\E^P\sup_{t\le T}|\Delta Y_t^{n+1}|^2
\le C\E^P\int_0^T\Bigl(
|\Delta G_s^n|^2
+|\Delta Z_s^{n+1}|_{a_s^P}^2
+\|\Delta U_s^{n+1}\|_\Pi^2\Bigr)ds.
\label{eq:Bpicard-sup}
\end{align}
Combining \eqref{eq:Bpicard-Gsquare}, \eqref{eq:Bpicard-contraction}, and \eqref{eq:Bpicard-geometric}, we obtain
\begin{equation}
\E^P\sup_{t\le T}|\Delta Y_t^{n+1}|^2
+\E^P\int_0^T\Bigl(
|\Delta Z_s^{n+1}|_{a_s^P}^2
+\|\Delta U_s^{n+1}\|_\Pi^2\Bigr)ds
\le Cq^{n-1}.
\label{eq:Bpicard-cauchy-rate}
\end{equation}
Hence the series
\[
\sum_{n\ge1}\|\Delta Y^{n+1}\|_{\mathbb S^2(P)},\qquad
\sum_{n\ge1}\|\Delta Z^{n+1}\|_{\mathbb H_B^2(P)},\qquad
\sum_{n\ge1}\|\Delta U^{n+1}\|_{\mathbb H_\Pi^2(P)}
\]
converge because $\sum_n q^{n/2}<\infty$. Thus $(Y^n,Z^n,U^n)$ is Cauchy in
\[
\mathbb S^2(P)\times\mathbb H_B^2(P)\times\mathbb H_\Pi^2(P).
\]
By completeness, there exists $(Y,Z,U)$ in this space such that
\begin{equation}
Y^n\to Y\ \text{in }\mathbb S^2(P),\qquad
Z^n\to Z\ \text{in }\mathbb H_B^2(P),\qquad
U^n\to U\ \text{in }\mathbb H_\Pi^2(P).
\label{eq:Bpicard-strongconv}
\end{equation}

\medskip
\noindent\emph{Step 8: passage to the limit in the generator.}
The natural coupling gives
\[
\int_0^T W_2^2(\Lcal^P(Y_s^n),\Lcal^P(Y_s))ds
\le \E^P\int_0^T|Y_s^n-Y_s|^2ds\longrightarrow0.
\]
By Assumption~\ref{HB}, \eqref{eq:qv}, and \eqref{eq:Bpicard-strongconv}, if
\[
F_s^n:=f(s,X_{s-},\mu_{s-},Y_s^n,Z_s^n,U_s^n,\Lcal^P(Y_s^n),\alpha_s)
\]
and
\[
F_s:=f(s,X_{s-},\mu_{s-},Y_s,Z_s,U_s,\Lcal^P(Y_s),\alpha_s),
\]
then
\begin{align*}
\E^P\int_0^T|F_s^n-F_s|^2ds
\le C L_B^2\Bigg(&\E^P\int_0^T|Y_s^n-Y_s|^2ds\\
&+\E^P\int_0^T|Z_s^n-Z_s|_{a_s^P}^2ds\\
&+\E^P\int_0^T\|U_s^n-U_s\|_\Pi^2ds\\
&+\int_0^T W_2^2(\Lcal^P(Y_s^n),\Lcal^P(Y_s))ds\Bigg)
\longrightarrow0.
\end{align*}
By Cauchy--Schwarz,
\begin{equation}
\E^P\sup_{t\le T}\left|\int_t^T(F_s^n-F_s)ds\right|^2
\le T\E^P\int_0^T|F_s^n-F_s|^2ds\longrightarrow0.
\label{eq:Bpicard-drift-limit}
\end{equation}
Moreover, BDG and Doob--Poisson inequalities give
\begin{align*}
\E^P\sup_{t\le T}\left|\int_0^t(Z_s^n-Z_s)dB_s\right|^2
&\le C\E^P\int_0^T|Z_s^n-Z_s|_{a_s^P}^2ds\longrightarrow0,\\
\E^P\sup_{t\le T}\left|\int_0^t\!\int_E(U_s^n-U_s)(e)\widetilde N^P(ds,de)\right|^2
&\le C\E^P\int_0^T\|U_s^n-U_s\|_\Pi^2ds\longrightarrow0.
\end{align*}
Finally, $Y^{n+1}\to Y$ in $\mathbb S^2(P)$ and the source $\chi$ is independent of $n$. Passing to the limit in each term of \eqref{eq:Bpicard-iteration} yields, simultaneously for all $t$ up to indistinguishability,
\[
Y_t=\zeta+\int_t^T(F_s+\chi_s)ds
-\int_t^TZ_s dB_s
-\int_t^T\!\int_EU_s(e)\widetilde N^P(ds,de),
\]
which is \eqref{eq:Bexist}.

\medskip
\noindent\emph{Step 9: uniqueness.}
Let $(Y,Z,U)$ and $(\bar Y,\bar Z,\bar U)$ be two solutions of \eqref{eq:Bexist}. Set
\[
\widehat Y=Y-\bar Y,\qquad
\widehat Z=Z-\bar Z,\qquad
\widehat U=U-\bar U,
\]
and
\begin{align*}
\widehat F_t:={}&f(t,X_{t-},\mu_{t-},Y_t,Z_t,U_t,\Lcal^P(Y_t),\alpha_t)\\
&-f(t,X_{t-},\mu_{t-},\bar Y_t,\bar Z_t,\bar U_t,\Lcal^P(\bar Y_t),\alpha_t).
\end{align*}
The terminal condition $\zeta$ and the source $\chi$ are the same for both solutions, so they cancel in the difference and
\[
d\widehat Y_t=-\widehat F_tdt+\widehat Z_tdB_t
+\int_E\widehat U_t(e)\widetilde N^P(dt,de),
\qquad \widehat Y_T=0.
\]
It\^o's formula applied to $e^{\gamma t}|\widehat Y_t|^2$, followed by expectation, gives
\begin{align*}
&\E^P\int_0^Te^{\gamma s}\Bigl(
\gamma|\widehat Y_s|^2+|\widehat Z_s|_{a_s^P}^2
+\|\widehat U_s\|_\Pi^2\Bigr)ds\\
&\qquad=2\E^P\int_0^Te^{\gamma s}\widehat Y_s\widehat F_sds.
\end{align*}
As in \eqref{eq:Bpicard-Gdiff}--\eqref{eq:Bpicard-Gsquare},
\[
\E^P|\widehat F_s|^2
\le C_0L_B^2\E^P\Bigl(
|\widehat Y_s|^2+|\widehat Z_s|_{a_s^P}^2
+\|\widehat U_s\|_\Pi^2\Bigr).
\]
Young's inequality, with the same $\gamma$ as in \eqref{eq:Bpicard-contraction}, then gives
\[
\|(\widehat Y,\widehat Z,\widehat U)\|_{\gamma,P}^2
\le q\|(\widehat Y,\widehat Z,\widehat U)\|_{\gamma,P}^2,
\qquad 0<q<1.
\]
Thus $(1-q)\|(\widehat Y,\widehat Z,\widehat U)\|_{\gamma,P}^2\le0$, so this norm vanishes. Hence
\[
\widehat Z=0\quad dt\otimes P\text{-a.e. in the }a^P\text{-norm},
\qquad
\widehat U=0\quad dt\otimes P\otimes\Pi\text{-a.e.},
\]
and $\widehat Y=0$ in $L^2(dt\otimes P)$. Returning to the difference equation, $\widehat Y$ is a c\`adl\`ag martingale that vanishes $dt\otimes P$-a.e. and has zero terminal value; therefore it vanishes up to indistinguishability. Uniqueness follows.

\medskip
\noindent\emph{Step 10: a priori estimate and uniformity in $P$.}
For the solution obtained above, set $\eta_t=\Lcal^P(Y_t)$ and
\[
F_t:=f(t,X_{t-},\mu_{t-},Y_t,Z_t,U_t,\eta_t,\alpha_t)+\chi_t.
\]
By Assumption~\ref{HB} and the definition of $\ell$,
\begin{equation}
|F_t|\le \ell_t+L_B\Bigl(|Y_t|+|Z_t|+\|U_t\|_\Pi+W_2(\eta_t,\delta_0)\Bigr),
\label{eq:Bpicard-Fgrowth}
\end{equation}
and
\begin{equation}
W_2^2(\eta_t,\delta_0)=\E^P|Y_t|^2.
\label{eq:Bpicard-lawzero}
\end{equation}
Apply It\^o's formula to $e^{\gamma t}|Y_t|^2$ on $[t,T]$ and take expectations as in Step~4. We obtain
\begin{align}
&\E^Pe^{\gamma t}|Y_t|^2
+\E^P\int_t^Te^{\gamma s}\Bigl(
\gamma|Y_s|^2+|Z_s|_{a_s^P}^2+\|U_s\|_\Pi^2\Bigr)ds
\nonumber\\
&\qquad=\E^Pe^{\gamma T}|\zeta|^2
+2\E^P\int_t^Te^{\gamma s}Y_sF_sds.
\label{eq:Bpicard-apriori-Ito}
\end{align}
We estimate the terms in $2Y_sF_s$ separately. First,
\[
2|Y_s|\ell_s\le \frac{\gamma}{4}|Y_s|^2+\frac{4}{\gamma}\ell_s^2.
\]
Next, using \eqref{eq:qv},
\begin{align*}
2L_B|Y_s||Z_s|
&\le \frac14|Z_s|_{a_s^P}^2
+4L_B^2\lambda_{\min}(\underline a)^{-1}|Y_s|^2,\\
2L_B|Y_s|\|U_s\|_\Pi
&\le \frac14\|U_s\|_\Pi^2+4L_B^2|Y_s|^2,\\
2L_B|Y_s|^2&=2L_B|Y_s|^2.
\end{align*}
Finally, $W_2(\eta_s,\delta_0)$ is deterministic under $P$, and Cauchy--Schwarz together with \eqref{eq:Bpicard-lawzero} gives
\begin{align*}
2L_B\E^P\bigl[|Y_s|W_2(\eta_s,\delta_0)\bigr]
&\le2L_B(\E^P|Y_s|^2)^{1/2}W_2(\eta_s,\delta_0)\\
&=2L_B\E^P|Y_s|^2.
\end{align*}
Choose $\gamma$ larger than the sum of the structural coefficients multiplying $|Y_s|^2$. The terms $\frac14|Z|_{a^P}^2$, $\frac14\|U\|_\Pi^2$, and the $|Y|^2$ term are absorbed into the left-hand side of \eqref{eq:Bpicard-apriori-Ito}. With $t=0$,
\begin{align}
\E^P\int_0^T\Bigl(|Y_t|^2+|Z_t|_{a_t^P}^2+\|U_t\|_\Pi^2\Bigr)dt
\le C\E^P\left[|\zeta|^2+\int_0^T\ell_t^2dt\right].
\label{eq:Bpicard-apriori-energy}
\end{align}
To control the supremum of $Y$, rewrite the BSDE as
\[
Y_t+\int_0^tF_sds
=\E^P\left[\zeta+\int_0^TF_sds\,\middle|\,\mathcal F_t^P\right]
=:M_t.
\]
Hence
\[
\sup_{t\le T}|Y_t|
\le \sup_{t\le T}|M_t|+\int_0^T|F_s|ds.
\]
Doob's $L^2$ inequality gives
\begin{align*}
\E^P\sup_{t\le T}|M_t|^2
&\le4\E^P\left|\zeta+\int_0^TF_sds\right|^2\\
&\le8\E^P|\zeta|^2+8T\E^P\int_0^T|F_s|^2ds.
\end{align*}
Therefore,
\begin{equation}
\E^P\sup_{t\le T}|Y_t|^2
\le C\E^P\left[|\zeta|^2+\int_0^T|F_s|^2ds\right].
\label{eq:Bpicard-sup-intermediate}
\end{equation}
Squaring \eqref{eq:Bpicard-Fgrowth} and using $|Z_s|^2\le\lambda_{\min}(\underline a)^{-1}|Z_s|_{a_s^P}^2$, we obtain
\begin{align*}
|F_s|^2
\le C\Bigl(&\ell_s^2+|Y_s|^2
+\lambda_{\min}(\underline a)^{-1}|Z_s|_{a_s^P}^2
+\|U_s\|_\Pi^2+W_2^2(\eta_s,\delta_0)\Bigr).
\end{align*}
After integration and expectation, \eqref{eq:Bpicard-lawzero} gives
\begin{align*}
\E^P\int_0^T|F_s|^2ds
\le C\Bigg(&\E^P\int_0^T\ell_s^2ds
+2\E^P\int_0^T|Y_s|^2ds\\
&+\E^P\int_0^T|Z_s|_{a_s^P}^2ds
+\E^P\int_0^T\|U_s\|_\Pi^2ds\Bigg).
\end{align*}
The energy estimate \eqref{eq:Bpicard-apriori-energy} therefore implies
\[
\E^P\int_0^T|F_s|^2ds
\le C\E^P\left[\int_0^T\ell_s^2ds+|\zeta|^2\right].
\]
Inserted into \eqref{eq:Bpicard-sup-intermediate}, this yields
\begin{align}
\E^P\sup_{t\le T}|Y_t|^2
+\E^P\int_0^T\Bigl(|Z_t|_{a_t^P}^2+\|U_t\|_\Pi^2\Bigr)dt
\le C\E^P\left[|\zeta|^2+\int_0^T\ell_t^2dt\right].
\label{eq:Bpicard-apriori-sup}
\end{align}
Finally,
\[
\sup_{t\le T}|\mathsf K_t|^2
=\sup_{t\le T}\left|\int_0^t\chi_sds\right|^2
\le T\int_0^T|\chi_s|^2ds.
\]
Together with \eqref{eq:Bpicard-data}, this proves \eqref{eq:Bexistbound}. All constants are uniform in $P$ because $\underline a,\overline a$, the Lipschitz constants, the data bounds, and the compensator $\Pi$ are common to the model family. No occupation assumption is used in this construction.
\end{proof}

\begin{remark}[One-way coupling]
The forward equation does not depend on $(Y,Z,U)$. The forward process is therefore fixed before the backward iteration; consequently, the source $\chi$ is the same at every Picard step and cancels in the differences.
\end{remark}

\begin{theorem}[Well-posedness of the selected system]\label{thm:well}
Let $\alpha\in\mathcal A^2$. Assume \ref{HF}, \ref{HB}, \ref{HD}, and the representation property \eqref{eq:PRP}. Then, for every $P\in\Pcal$, system \eqref{eq:FWD}--\eqref{eq:Kdef} admits a unique solution
\[
(X^P,Y^P,Z^P,U^P,\mathsf K^P)
\in\mathbb S^2(P;\R^d)\times\mathbb S^2(P)\times\mathbb H_B^2(P)
\times\mathbb H_\Pi^2(P)\times\mathbb S^2(P),
\]
and
\begin{align}
\sup_{P\in\Pcal}\E^P\Bigl[
&\sup_{t\le T}|X_t^P|^2+\sup_{t\le T}|Y_t^P|^2
+\sup_{t\le T}|\mathsf K_t^P|^2\nonumber\\
&+\int_0^T\{|Z_t^P|_{a_t^P}^2+\|U_t^P\|_\Pi^2\}dt
\Bigr]<\infty.
\label{eq:uniformwell}
\end{align}
\end{theorem}

\begin{proof}
Fix $P\in\Pcal$. Theorem~\ref{thm:Fwell} provides a unique forward process
\[
X^P\in\mathbb S^2(P;\R^d),\qquad \mu_t^P=\Lcal^P(X_t^P),
\]
satisfying \eqref{eq:FWD}. By the growth of $A^0$ and \eqref{eq:Fsecond},
\begin{align*}
\E^P\int_0^T
|A^0(X_{t-}^P,\mu_{t-}^P,\alpha_t)|^2dt
&\le C\E^P\int_0^T(1+|X_{t-}^P|^2+|\alpha_t|^2)dt\\
&<\infty.
\end{align*}
Hence
\[
\chi_t^P=\langle\varrho,A^0(X_{t-}^P,\mu_{t-}^P,\alpha_t)\rangle
\]
is predictable and belongs to $L^2(dt\otimes P)$. The assumptions of Theorem~\ref{thm:Bwell} are therefore satisfied along this forward process, and that theorem constructs a unique triplet
\[
(Y^P,Z^P,U^P)\in
\mathbb S^2(P)\times\mathbb H_B^2(P)\times\mathbb H_\Pi^2(P).
\]
Finally, set
\[
\Gamma_t^P=A^0(X_{t-}^P,\mu_{t-}^P,\alpha_t),
\qquad
\mathsf K_t^P=\int_0^t\langle\varrho,\Gamma_s^P\rangle ds.
\]
By Cauchy--Schwarz,
\[
\sup_{t\le T}|\mathsf K_t^P|^2
\le T|\varrho|^2\int_0^T|\Gamma_s^P|^2ds,
\]
so $\mathsf K^P\in\mathbb S^2(P)$. The five processes satisfy \eqref{eq:FWD}--\eqref{eq:Kdef}, establishing existence.

For uniqueness, let $(X,Y,Z,U,\mathsf K)$ and $(\bar X,\bar Y,\bar Z,\bar U,\bar{\mathsf K})$ be two solutions. Uniqueness in Theorem~\ref{thm:Fwell} gives $X=\bar X$ up to indistinguishability, hence $\mu=\bar\mu$. Therefore $A^0(X_-,\mu_-,\alpha)$ is the same in both systems and $\mathsf K=\bar{\mathsf K}$. The two BSDEs have the same terminal condition, the same exogenous source, and the same integrators; Step~9 of Theorem~\ref{thm:Bwell} then gives
\[
Y=\bar Y\quad\text{up to indistinguishability},
\qquad Z=\bar Z\text{ in }\mathbb H_B^2(P),
\qquad U=\bar U\text{ in }\mathbb H_\Pi^2(P).
\]

Finally, \eqref{eq:Fsecond}, \eqref{eq:Bexistbound}, and the preceding estimate for $\mathsf K$ give
\[
\E^P\Bigl[\sup_{t\le T}|X_t^P|^2+\sup_{t\le T}|Y_t^P|^2
+\sup_{t\le T}|\mathsf K_t^P|^2
+\int_0^T(|Z_t^P|_{a_t^P}^2+\|U_t^P\|_\Pi^2)dt\Bigr]\le C.
\]
All constants are common to $P\in\Pcal$; taking the supremum proves \eqref{eq:uniformwell}.
\end{proof}

\subsection{Stability of the backward equation with jumps}

We now compare two solutions driven by the same integrators.

\begin{lemma}[Stability of a McKean--Vlasov BSDE with jumps]\label{lem:Bstab}
Fix $P\in\Pcal$ and assume \ref{HB}. Let $(Y,Z,U)$ and $(\bar Y,\bar Z,\bar U)$ be two solutions, driven by the same generator $f$ and the same integrators, with terminal conditions $\zeta,\bar\zeta$, exogenous terms $\chi,\bar\chi$, and arguments $(X,\mu,\alpha)$ and $(\bar X,\bar\mu,\bar\alpha)$, respectively. Then
\begin{align}
&\E^P\sup_{t\le T}|Y_t-\bar Y_t|^2
+\E^P\int_0^T|Z_t-\bar Z_t|_{a_t^P}^2dt
+\E^P\int_0^T\|U_t-\bar U_t\|_\Pi^2dt\nonumber\\
&\le C\E^P\left[
|\zeta-\bar\zeta|^2+\int_0^T\{
|X_t-\bar X_t|^2+\Wtwo^2(\mu_t,\bar\mu_t)
+|\alpha_t-\bar\alpha_t|^2+|\chi_t-\bar \chi_t|^2\}dt\right],
\label{eq:Bstab}
\end{align}
where $C$ is independent of $P$.
\end{lemma}

\begin{proof}
Fix $P$. We first display the subtraction in full. The two backward equations imply, for every $t\in[0,T]$, $P$-a.s.,
\begin{align*}
Y_t-\bar Y_t
={}&\zeta-\bar\zeta\\
&+\int_t^T\Bigl[
 f(s,X_s,\mu_s,Y_s,Z_s,U_s,\Lcal^P(Y_s),\alpha_s)\\
&\hspace{23mm}-f(s,\bar X_s,\bar\mu_s,\bar Y_s,\bar Z_s,\bar U_s,\Lcal^P(\bar Y_s),\bar\alpha_s)
 +\chi_s-\bar\chi_s\Bigr]ds\\
&-\int_t^T(Z_s-\bar Z_s)dB_s
-\int_t^T\int_E\bigl(U_s(e)-\bar U_s(e)\bigr)\widetilde N^P(ds,de).
\end{align*}
By Assumption~\ref{HB}, the generator difference appearing here satisfies
\begin{align*}
&\bigl|f(s,X_s,\mu_s,Y_s,Z_s,U_s,\Lcal^P(Y_s),\alpha_s)
-f(s,\bar X_s,\bar\mu_s,\bar Y_s,\bar Z_s,\bar U_s,\Lcal^P(\bar Y_s),\bar\alpha_s)\bigr|\\
&\quad\le L_B\Bigl(
|X_s-\bar X_s|+W_2(\mu_s,\bar\mu_s)+|Y_s-\bar Y_s|+|Z_s-\bar Z_s|\\
&\hspace{31mm}+\|U_s-\bar U_s\|_\Pi
+W_2(\Lcal^P(Y_s),\Lcal^P(\bar Y_s))+|\alpha_s-\bar\alpha_s|\Bigr).
\end{align*}
Applying It\^o's formula with jumps directly to $e^{\gamma s}|Y_s-\bar Y_s|^2$ on $[t,T]$ gives, after localization,
\begin{align*}
&e^{\gamma t}|Y_t-\bar Y_t|^2
+\int_t^Te^{\gamma s}\Bigl(\gamma|Y_s-\bar Y_s|^2+|Z_s-\bar Z_s|_{a_s^P}^2\Bigr)ds\\
&\quad+\int_t^T\int_Ee^{\gamma s}|U_s(e)-\bar U_s(e)|^2N^P(ds,de)\\
&=e^{\gamma T}|\zeta-\bar\zeta|^2\\
&\quad+2\int_t^Te^{\gamma s}(Y_s-\bar Y_s)\Bigl[
 f(s,X_s,\mu_s,Y_s,Z_s,U_s,\Lcal^P(Y_s),\alpha_s)\\
&\hspace{33mm}-f(s,\bar X_s,\bar\mu_s,\bar Y_s,\bar Z_s,\bar U_s,\Lcal^P(\bar Y_s),\bar\alpha_s)
 +\chi_s-\bar\chi_s\Bigr]ds\\
&\quad-2\int_t^Te^{\gamma s}(Y_{s-}-\bar Y_{s-})(Z_s-\bar Z_s)dB_s\\
&\quad-\int_t^T\int_Ee^{\gamma s}\Bigl(2(Y_{s-}-\bar Y_{s-})(U_s(e)-\bar U_s(e))
+|U_s(e)-\bar U_s(e)|^2\Bigr)\widetilde N^P(ds,de).
\end{align*}
This fully expanded form identifies every term before any abbreviation is used. We now introduce compact notation for the remaining estimates:
\begin{align*}
\Delta X_t&:=X_t-\bar X_t,\qquad
\Delta Y_t:=Y_t-\bar Y_t,
&& t\in[0,T],\quad P\text{-a.s.},\\
\Delta Z_t&:=Z_t-\bar Z_t,\qquad
\Delta\alpha_t:=\alpha_t-\bar\alpha_t,\qquad
\Delta\chi_t:=\chi_t-\bar\chi_t,
&& dt\otimes P\text{-a.e.},\\
\Delta U_t(e)&:=U_t(e)-\bar U_t(e),
&& dt\otimes P\otimes\Pi(de)\text{-a.e.},\\
\Delta\zeta&:=\zeta-\bar\zeta,
&& P\text{-a.s.}
\end{align*}
and
\begin{align*}
\Delta f_s:={}&f(s,X_s,\mu_s,Y_s,Z_s,U_s,\Lcal^P(Y_s),\alpha_s)\\
&-f(s,\bar X_s,\bar\mu_s,\bar Y_s,\bar Z_s,\bar U_s,\Lcal^P(\bar Y_s),\bar\alpha_s).
\end{align*}
Hence
\begin{align}
|\Delta f_s|
\le L_B\bigl(&|\Delta X_s|+W_2(\mu_s,\bar\mu_s)+|\Delta Y_s|+|\Delta Z_s|\nonumber\\
&+\|\Delta U_s\|_\Pi+W_2(\Lcal^P(Y_s),\Lcal^P(\bar Y_s))
+|\Delta\alpha_s|\bigr).
\label{eq:Bstab-df-detail}
\end{align}
The natural coupling gives
\begin{equation}
W_2^2(\Lcal^P(Y_s),\Lcal^P(\bar Y_s))
\le\E^P|\Delta Y_s|^2.
\label{eq:Bstab-law-detail}
\end{equation}
Since $a_s^P\ge\underline a>0$,
\[
|\Delta Z_s|^2
\le\lambda_{\min}(\underline a)^{-1}|\Delta Z_s|_{a_s^P}^2.
\]
With this notation, the fully expanded subtraction above becomes
\begin{align*}
\Delta Y_t={}&\Delta\zeta+\int_t^T(\Delta f_s+\Delta \chi_s)ds
-\int_t^T\Delta Z_s dB_s\\
&-\int_t^T\int_E\Delta U_s(e)\widetilde N^P(ds,de),
\qquad t\in[0,T],\quad P\text{-a.s.}
\end{align*}
For $\gamma>0$, It\^o's formula with jumps applied to $e^{\gamma s}|\Delta Y_s|^2$ on $[t,T]$ gives, after localization of the martingales,
\begin{align*}
&e^{\gamma t}|\Delta Y_t|^2
+\int_t^Te^{\gamma s}\bigl(\gamma|\Delta Y_s|^2+|\Delta Z_s|_{a_s^P}^2\bigr)ds\\
&\quad+\int_t^T\int_Ee^{\gamma s}|\Delta U_s(e)|^2N^P(ds,de)\\
&=e^{\gamma T}|\Delta\zeta|^2
+2\int_t^Te^{\gamma s}\Delta Y_s(\Delta f_s+\Delta \chi_s)ds\\
&\quad-2\int_t^Te^{\gamma s}\Delta Y_{s-}\Delta Z_s dB_s
-\int_t^T\int_Ee^{\gamma s}
\bigl(2\Delta Y_{s-}\Delta U_s(e)+|\Delta U_s(e)|^2\bigr)
\widetilde N^P(ds,de).
\end{align*}
Taking expectations and using the compensator for the quadratic jump term,
\begin{align}
&\E^Pe^{\gamma t}|\Delta Y_t|^2
+\E^P\int_t^Te^{\gamma s}
\bigl(\gamma|\Delta Y_s|^2+|\Delta Z_s|_{a_s^P}^2+\|\Delta U_s\|_\Pi^2\bigr)ds\nonumber\\
&\qquad=\E^Pe^{\gamma T}|\Delta\zeta|^2
+2\E^P\int_t^Te^{\gamma s}\Delta Y_s(\Delta f_s+\Delta \chi_s)ds.
\label{eq:Bstab-Ito-detail}
\end{align}
Set
\[
R_s:=|\Delta X_s|+W_2(\mu_s,\bar\mu_s)+|\Delta\alpha_s|+|\Delta \chi_s|.
\]
Using \eqref{eq:Bstab-df-detail}, ellipticity, and Young's inequality in the form $2ab\le\varepsilon a^2+\varepsilon^{-1}b^2$, we obtain constants $C$ such that, for every $\varepsilon>0$,
\begin{align*}
2|\Delta Y_s|(|\Delta f_s|+|\Delta \chi_s|)
\le{}&C|\Delta Y_s|^2+C R_s^2+C\E^P|\Delta Y_s|^2\\
&+\varepsilon|\Delta Z_s|_{a_s^P}^2
+\varepsilon\|\Delta U_s\|_\Pi^2.
\end{align*}
Choose first $\varepsilon=1/4$, then $\gamma$ large enough to absorb all $|\Delta Y|^2$ terms on the right-hand side of \eqref{eq:Bstab-Ito-detail}. With $t=0$, we obtain the energy estimate
\begin{align}
&\sup_{t\le T}\E^P|\Delta Y_t|^2
+\E^P\int_0^T\bigl(|\Delta Z_s|_{a_s^P}^2+\|\Delta U_s\|_\Pi^2\bigr)ds\nonumber\\
&\quad\le C\E^P\left[|\Delta\zeta|^2+
\int_0^T\bigl(|\Delta X_s|^2+W_2^2(\mu_s,\bar\mu_s)
+|\Delta\alpha_s|^2+|\Delta \chi_s|^2\bigr)ds\right].
\label{eq:Bstab-energy-detail}
\end{align}

For the time supremum, the difference equation gives
\begin{align*}
\sup_{t\le T}|\Delta Y_t|
\le{}&|\Delta\zeta|+\int_0^T(|\Delta f_s|+|\Delta \chi_s|)ds\\
&+2\sup_{t\le T}\left|\int_0^t\Delta Z_s dB_s\right|
+2\sup_{t\le T}\left|\int_0^t\int_E\Delta U_s(e)\widetilde N^P(ds,de)\right|.
\end{align*}
Squaring, then applying Cauchy--Schwarz in time and Lemma~\ref{lem:mart},
\begin{align*}
\E^P\sup_{t\le T}|\Delta Y_t|^2
\le C\E^P\Bigl[&|\Delta\zeta|^2+
\int_0^T(|\Delta f_s|^2+|\Delta \chi_s|^2)ds\\
&+\int_0^T(|\Delta Z_s|_{a_s^P}^2+\|\Delta U_s\|_\Pi^2)ds\Bigr].
\end{align*}
The estimate \eqref{eq:Bstab-df-detail}, the law bound \eqref{eq:Bstab-law-detail}, and \eqref{eq:Bstab-energy-detail} absorb the terms involving $(\Delta Y,\Delta Z,\Delta U)$, giving \eqref{eq:Bstab}. The constants are independent of $P$ because the ellipticity bounds, Lipschitz constants, and Poisson compensator are common to the family.
\end{proof}

\begin{lemma}[Continuity of the laws despite jumps]\label{lem:lawcontinuous}
Fix $P\in\Pcal$. Let
\[
(X^P,Y^P,Z^P,U^P,\mathsf K^P)
\]
be a square-integrable solution of the selected system. Then
\[
t\longmapsto\mu_t^P:=\Lcal^P(X_t^P),\qquad
t\longmapsto\eta_t^P:=\Lcal^P(Y_t^P)
\]
are continuous from $[0,T]$ into $\mathcal P_2$ for $W_2$. Moreover, for every deterministic $t>0$,
\[
X_t^P=X_{t-}^P,\qquad Y_t^P=Y_{t-}^P,
\quad P\text{-a.s.}
\]
\end{lemma}
\begin{proof}
For the forward process, stochastic isometries and Cauchy--Schwarz give, for $s<t$,
\[
\E^P|X_t-X_s|^2\le3(t-s)\E^P\int_s^t|b_u|^2du
+3\E^P\int_s^t\operatorname{Tr}(\sigma_u a_u^P\sigma_u^\top)du
+3\E^P\int_s^t\|\beta_u\|_\Pi^2du.
\]
All integrands on the right-hand side are integrable on $[0,T]$. For the backward component, the differential form of the equation gives, for $0\le s<t\le T$,
\[
Y_t-Y_s=-\int_s^t(f_u+h_u)du
+\int_s^t Z_u\,dB_u
+\int_s^t\int_EU_u(e)\widetilde N^P(du,de),
\]
where $h_u=\langle\varrho,\Gamma_u\rangle$. By $(a+b+c)^2\le3(a^2+b^2+c^2)$, Cauchy--Schwarz for the time integral, the continuous martingale isometry, and the Poisson isometry,
\begin{align}
\E^P|Y_t-Y_s|^2
\le{}&3(t-s)\E^P\int_s^t|f_u+h_u|^2du
+3\E^P\int_s^t|Z_u|_{a_u^P}^2du\nonumber\\
&+3\E^P\int_s^t\|U_u\|_\Pi^2du.
\label{eq:lawcontinuity-Y}
\end{align}
Moreover,
\[
|f_u+h_u|^2\le2|f_u|^2+2|\varrho|^2|\Gamma_u|^2.
\]
The Lipschitz growth of $f$, the linear growth of $A^0$, and the estimates already established imply
\[
\E^P\int_0^T|f_u|^2du
+\E^P\int_0^T|\Gamma_u|^2du
+\E^P\int_0^T|Z_u|_{a_u^P}^2du
+\E^P\int_0^T\|U_u\|_\Pi^2du<\infty.
\]
Hence all four time-integrated terms appearing in the estimates for $X$ and $Y$ belong to $L^1(0,T)$. Absolute continuity of the Lebesgue integral then implies that the right-hand sides tend to zero as $t-s\to0$. Thus
\[
\E^P|X_t-X_s|^2\longrightarrow0,
\qquad
\E^P|Y_t-Y_s|^2\longrightarrow0,
\]
which proves $L^2$ continuity at deterministic times.

For the laws, $(X_t,X_s)$ is a coupling of $(\mu_t^P,\mu_s^P)$ and $(Y_t,Y_s)$ is a coupling of $(\eta_t^P,\eta_s^P)$. Hence
\[
W_2^2(\mu_t^P,\mu_s^P)\le\E^P|X_t-X_s|^2,
\qquad
W_2^2(\eta_t^P,\eta_s^P)\le\E^P|Y_t-Y_s|^2,
\]
so both law flows are $W_2$-continuous.

Finally, fix deterministic $t>0$ and a sequence $s_n\uparrow t$. Since $X$ and $Y$ are c\`adl\`ag,
\[
X_{s_n}\to X_{t-},\qquad Y_{s_n}\to Y_{t-}
\quad P\text{-a.s.}
\]
At the same time, $L^2$ continuity gives
\[
X_{s_n}\to X_t,\qquad Y_{s_n}\to Y_t
\quad\text{in }L^2(P),
\]
hence in probability. Uniqueness of the limit in probability implies
\[
X_t=X_{t-},\qquad Y_t=Y_{t-},\quad P\text{-a.s.}
\]
This statement at deterministic times does not mean that the sample paths are continuous at their random jump times.
\end{proof}

\section{Stability under Perturbations of the Data}

This section works in the quadratic framework $\alpha\in\mathcal A^2$ under Assumptions~\ref{HF}, \ref{HB}, and \ref{HD}. The higher-moment assumption \ref{HqJ} is not needed here.

We keep $P$, the filtration, and the integrators $(B,N)$ fixed. Stability when the integrators and filtrations vary is studied in the preprint by Papapantoleon, Saplaouras, and Theodorakopoulos \cite{PapapantoleonEtAl2025}.

We keep the dependence on $\lambda$ explicit in order to identify the origin of the factor $\lambda^{-2}$ in the Yosida source.

Consider two data sets
\[
\mathfrak D=(\xi,\alpha,b,\sigma,\beta,f,g,A),\qquad
\bar{\mathfrak D}=(\bar\xi,\bar\alpha,\bar b,\bar\sigma,
\bar\beta,\bar f,\bar g,\bar A),
\]
on the same stochastic basis under $P$. Their coefficients are assumed to satisfy Assumptions~\ref{HF}--\ref{HD} with the same structural constants. Objects associated with the second data set carry a bar. For $\lambda>0$, the two operators are replaced by their Yosida approximations $A^\lambda$ and $\bar A^\lambda$.

The forward residual, evaluated along the second solution, is
\begin{align}
\mathfrak R_F(P):={}&\E^P|\xi-\bar\xi|^2
+\E^P\int_0^T|\alpha_t-\bar\alpha_t|^2dt\nonumber\\
&+\E^P\int_0^T\Bigl[
|b(t,\bar X_{t-},\bar\mu_{t-},\bar\alpha_t)
-\bar b(t,\bar X_{t-},\bar\mu_{t-},\bar\alpha_t)|^2\nonumber\\
&\qquad+\|\sigma(t,\bar X_{t-},\bar\mu_{t-},\bar\alpha_t)
-\bar\sigma(t,\bar X_{t-},\bar\mu_{t-},\bar\alpha_t)\|^2\nonumber\\
&\qquad+\|\beta(t,\bar X_{t-},\bar\mu_{t-},\bar\alpha_t,\cdot)
-\bar\beta(t,\bar X_{t-},\bar\mu_{t-},\bar\alpha_t,\cdot)
\|_{L_\Pi^2}^2\Bigr]dt.
\label{eq:dataRF}
\end{align}
For the regularized backward part, set
\begin{align}
\mathfrak R_B^\lambda(P):={}&
\E^P|g(\bar X_T,\bar\mu_T)-\bar g(\bar X_T,\bar\mu_T)|^2
\nonumber\\
&+\E^P\int_0^T\bigl|f(t,\bar X_{t-},\bar\mu_{t-},
\bar Y_t,\bar Z_t,\bar U_t,\bar\eta_t,\bar\alpha_t)\nonumber\\
&\hspace{34mm}-\bar f(t,\bar X_{t-},\bar\mu_{t-},
\bar Y_t,\bar Z_t,\bar U_t,\bar\eta_t,\bar\alpha_t)
\bigr|^2dt\nonumber\\
&+\E^P\int_0^T|A^\lambda(\bar X_{t-},\bar\mu_{t-},\bar\alpha_t)
-\bar A^\lambda(\bar X_{t-},\bar\mu_{t-},\bar\alpha_t)|^2dt.
\label{eq:dataRB}
\end{align}
In \eqref{eq:dataRB}, $(\bar Y,\bar Z,\bar U)$ denotes the regularized solution of the second system with the same parameter $\lambda$.

\begin{theorem}[Strong stability with fixed integrators]\label{thm:datastability}
Fix $P\in\Pcal$ and $\lambda>0$. Assume \ref{HF}, \ref{HB}, and \ref{HD}. Let two data sets be defined on the same stochastic basis, with the same filtration and integrators, and satisfy the same structural constants; let $\alpha,\bar\alpha\in\mathcal A^2$. Then the corresponding regularized solutions satisfy
\begin{align}
\E^P\sup_{t\le T}|X_t-\bar X_t|^2
&\le C\mathfrak R_F(P),
\label{eq:dataXstab}\\
\E^P\Biggl[\sup_{t\le T}|Y_t-\bar Y_t|^2
&+\sup_{t\le T}|\mathsf K_t^\lambda-\bar{\mathsf K}_t^\lambda|^2
+\int_0^T|Z_t-\bar Z_t|_{a_t^P}^2dt\nonumber\\
&+\int_0^T\|U_t-\bar U_t\|_\Pi^2dt\Biggr]
\le C\Bigl\{(1+\lambda^{-2})\mathfrak R_F(P)+\mathfrak R_B^\lambda(P)\Bigr\}.
\label{eq:dataYstab}
\end{align}
The constant $C$ is independent of $P$, $\lambda$, and the two data sets.
\end{theorem}

\begin{proof}
\noindent\emph{Step 1: stability of the forward equation.}
We first subtract the two forward equations without abbreviations. For every $t\in[0,T]$, $P$-a.s.,
\begin{align*}
X_t-\bar X_t
={}&\xi-\bar\xi\\
&+\int_0^t\Bigl[b(s,X_{s-},\mu_{s-},\alpha_s)-\bar b(s,\bar X_{s-},\bar\mu_{s-},\bar\alpha_s)\Bigr]ds\\
&+\int_0^t\Bigl[\sigma(s,X_{s-},\mu_{s-},\alpha_s)-\bar\sigma(s,\bar X_{s-},\bar\mu_{s-},\bar\alpha_s)\Bigr]dB_s\\
&+\int_0^t\int_E\Bigl[\beta(s,X_{s-},\mu_{s-},\alpha_s,e)-\bar\beta(s,\bar X_{s-},\bar\mu_{s-},\bar\alpha_s,e)\Bigr]\widetilde N^P(ds,de).
\end{align*}
For the estimates, define
\begin{align*}
\Delta X_t&:=X_t-\bar X_t,\qquad
\Delta Y_t:=Y_t-\bar Y_t,
&& t\in[0,T],\quad P\text{-a.s.},\\
\Delta Z_t&:=Z_t-\bar Z_t,\qquad
\Delta\alpha_t:=\alpha_t-\bar\alpha_t,
&& dt\otimes P\text{-a.e.},\\
\Delta U_t(e)&:=U_t(e)-\bar U_t(e),
&& dt\otimes P\otimes\Pi(de)\text{-a.e.},\\
\Delta\xi&:=\xi-\bar\xi,
&& P\text{-a.s.}
\end{align*}
For the drift,
\[
\Delta b_t=
\underbrace{b(t,X_{t-},\mu_{t-},\alpha_t)-b(t,\bar X_{t-},\bar\mu_{t-},\bar\alpha_t)}_{=:\delta b_t}
+r_t^b,
\]
where
\[
r_t^b=b(t,\bar X_{t-},\bar\mu_{t-},\bar\alpha_t)
-\bar b(t,\bar X_{t-},\bar\mu_{t-},\bar\alpha_t).
\]
Define $\delta\sigma_t,r_t^\sigma$ and $\delta\beta_t,r_t^\beta$ similarly. Then
\[
|\delta b_t|^2+\|\delta\sigma_t\|^2+\|\delta\beta_t\|_{L_\Pi^2}^2
\le C\bigl(|\Delta X_{t-}|^2+W_2^2(\mu_{t-},\bar\mu_{t-})+|\Delta\alpha_t|^2\bigr).
\]
Moreover, the coupling $(X_t,\bar X_t)$ under $P$ gives
\[
W_2^2(\mu_t,\bar\mu_t)\le\E^P|\Delta X_t|^2.
\]
The difference equation is, for every $t\in[0,T]$, $P$-a.s.,
\[
\Delta X_t=\Delta\xi+
\int_0^t(\delta b_s+r_s^b)ds+
\int_0^t(\delta\sigma_s+r_s^\sigma)dB_s+
\int_0^t\int_E(\delta\beta_s(e)+r_s^\beta(e))\widetilde N^P(ds,de).
\]
As in Lemma~\ref{lem:Fstab}, Cauchy--Schwarz for the drift and Lemma~\ref{lem:mart} for the two martingales give
\begin{align*}
\E^P\sup_{r\le t}|\Delta X_r|^2
\le{}&C\E^P|\Delta\xi|^2+C\E^P\int_0^t|\Delta\alpha_s|^2ds\\
&+C\E^P\int_0^t(|r_s^b|^2+\|r_s^\sigma\|^2+\|r_s^\beta\|_{L_\Pi^2}^2)ds\\
&+C\int_0^t\E^P\sup_{u\le s}|\Delta X_u|^2ds.
\end{align*}
The first group of terms is bounded by $C\mathfrak R_F(P)$. Gronwall's inequality yields
\[
\E^P\sup_{t\le T}|\Delta X_t|^2\le C\mathfrak R_F(P),
\]
which is \eqref{eq:dataXstab}.

\noindent\emph{Step 2: Yosida source and cumulative process.}
The source difference is first written explicitly as
\begin{align*}
&A^\lambda(X_{t-},\mu_{t-},\alpha_t)
-\bar A^\lambda(\bar X_{t-},\bar\mu_{t-},\bar\alpha_t)\\
&\quad= A^\lambda(X_{t-},\mu_{t-},\alpha_t)
-A^\lambda(\bar X_{t-},\bar\mu_{t-},\bar\alpha_t)\\
&\qquad+ A^\lambda(\bar X_{t-},\bar\mu_{t-},\bar\alpha_t)
-\bar A^\lambda(\bar X_{t-},\bar\mu_{t-},\bar\alpha_t).
\end{align*}
We then set
\[
r_t^{A,\lambda}:=A^\lambda(\bar X_{t-},\bar\mu_{t-},\bar\alpha_t)
-\bar A^\lambda(\bar X_{t-},\bar\mu_{t-},\bar\alpha_t),
\]
and
\[
\Delta A_t^\lambda:=A^\lambda(X_{t-},\mu_{t-},\alpha_t)
-\bar A^\lambda(\bar X_{t-},\bar\mu_{t-},\bar\alpha_t).
\]
Thus
\begin{align*}
\Delta A_t^\lambda
={}&A^\lambda(X_{t-},\mu_{t-},\alpha_t)
-A^\lambda(\bar X_{t-},\bar\mu_{t-},\bar\alpha_t)+r_t^{A,\lambda}.
\end{align*}
Estimate \eqref{eq:AYlip} gives
\begin{align*}
|\Delta A_t^\lambda|^2
\le{}&C\lambda^{-2}(|\Delta X_{t-}|^2+|\Delta\alpha_t|^2)
+C W_2^2(\mu_{t-},\bar\mu_{t-})+C|r_t^{A,\lambda}|^2.
\end{align*}
After integration, Step~1 and the definition of $\mathfrak R_B^\lambda(P)$ yield
\begin{equation}
\E^P\int_0^T|\Delta A_t^\lambda|^2dt
\le C\{(1+\lambda^{-2})\mathfrak R_F(P)+\mathfrak R_B^\lambda(P)\}.
\label{eq:datastab-source}
\end{equation}
Since
\[
\mathsf K_t^\lambda-\bar{\mathsf K}_t^\lambda
=\int_0^t\langle\varrho,\Delta A_s^\lambda\rangle ds,
\]
Cauchy--Schwarz gives
\begin{equation}
\E^P\sup_{t\le T}|\mathsf K_t^\lambda-\bar{\mathsf K}_t^\lambda|^2
\le T|\varrho|^2\E^P\int_0^T|\Delta A_s^\lambda|^2ds.
\label{eq:datastab-K}
\end{equation}

\noindent\emph{Step 3: terminal condition and backward generator.}
The terminal difference is
\[
g(X_T,\mu_T)-\bar g(\bar X_T,\bar\mu_T),
\]
which we decompose explicitly as
\begin{align*}
&g(X_T,\mu_T)-\bar g(\bar X_T,\bar\mu_T)\\
&\quad=g(X_T,\mu_T)-g(\bar X_T,\bar\mu_T)
+\bigl[g(\bar X_T,\bar\mu_T)-\bar g(\bar X_T,\bar\mu_T)\bigr].
\end{align*}
We now set
\[
\Delta\zeta:=g(X_T,\mu_T)-\bar g(\bar X_T,\bar\mu_T).
\]
By \eqref{eq:gLip}, the law coupling, and Step~1,
\begin{equation}
\E^P|\Delta\zeta|^2
\le C\mathfrak R_F(P)+C\mathfrak R_B^\lambda(P).
\label{eq:datastab-terminal-detail}
\end{equation}
For the generator, first write the complete difference
\begin{align*}
&f(t,X_{t-},\mu_{t-},Y_t,Z_t,U_t,\eta_t,\alpha_t)
-\bar f(t,\bar X_{t-},\bar\mu_{t-},\bar Y_t,\bar Z_t,\bar U_t,\bar\eta_t,\bar\alpha_t)\\
&\quad= f(t,X_{t-},\mu_{t-},Y_t,Z_t,U_t,\eta_t,\alpha_t)
-f(t,\bar X_{t-},\bar\mu_{t-},\bar Y_t,\bar Z_t,\bar U_t,\bar\eta_t,\bar\alpha_t)\\
&\qquad+ f(t,\bar X_{t-},\bar\mu_{t-},\bar Y_t,\bar Z_t,\bar U_t,\bar\eta_t,\bar\alpha_t)
-\bar f(t,\bar X_{t-},\bar\mu_{t-},\bar Y_t,\bar Z_t,\bar U_t,\bar\eta_t,\bar\alpha_t).
\end{align*}
Set
\[
\Delta f_t:=f(t,X_{t-},\mu_{t-},Y_t,Z_t,U_t,\eta_t,\alpha_t)
-\bar f(t,\bar X_{t-},\bar\mu_{t-},\bar Y_t,\bar Z_t,\bar U_t,\bar\eta_t,\bar\alpha_t),
\]
and
\[
r_t^f:=f(t,\bar X_{t-},\bar\mu_{t-},\bar Y_t,\bar Z_t,\bar U_t,\bar\eta_t,\bar\alpha_t)
-\bar f(t,\bar X_{t-},\bar\mu_{t-},\bar Y_t,\bar Z_t,\bar U_t,\bar\eta_t,\bar\alpha_t).
\]
Assumption~\ref{HB} gives
\begin{align}
|\Delta f_t|^2\le C\bigl(&|\Delta X_{t-}|^2+W_2^2(\mu_{t-},\bar\mu_{t-})+|\Delta\alpha_t|^2
+|\Delta Y_t|^2\nonumber\\
&+|\Delta Z_t|^2+\|\Delta U_t\|_\Pi^2
+W_2^2(\eta_t,\bar\eta_t)+|r_t^f|^2\bigr).
\label{eq:datastab-f-detail}
\end{align}
Since $(Y_t,\bar Y_t)$ is a coupling of $(\eta_t,\bar\eta_t)$,
\[
W_2^2(\eta_t,\bar\eta_t)\le\E^P|\Delta Y_t|^2.
\]

\noindent\emph{Step 4: BSDE energy estimate.}
Before using the abbreviations, the backward difference reads, for every $t\in[0,T]$, $P$-a.s.,
\begin{align*}
Y_t-\bar Y_t
={}&g(X_T,\mu_T)-\bar g(\bar X_T,\bar\mu_T)\\
&+\int_t^T\Bigl[
 f(s,X_{s-},\mu_{s-},Y_s,Z_s,U_s,\eta_s,\alpha_s)
-\bar f(s,\bar X_{s-},\bar\mu_{s-},\bar Y_s,\bar Z_s,\bar U_s,\bar\eta_s,\bar\alpha_s)\\
&\hspace{22mm}+\langle\varrho,
 A^\lambda(X_{s-},\mu_{s-},\alpha_s)
-\bar A^\lambda(\bar X_{s-},\bar\mu_{s-},\bar\alpha_s)\rangle\Bigr]ds\\
&-\int_t^T(Z_s-\bar Z_s)dB_s
-\int_t^T\int_E\bigl(U_s(e)-\bar U_s(e)\bigr)\widetilde N^P(ds,de).
\end{align*}
With the notation introduced above, this is exactly
\[
\Delta Y_t=\Delta\zeta+
\int_t^T\bigl(\Delta f_s+\langle\varrho,\Delta A_s^\lambda\rangle\bigr)ds
-\int_t^T\Delta Z_s dB_s
-\int_t^T\int_E\Delta U_s(e)\widetilde N^P(ds,de).
\]
Apply It\^o's formula with jumps to $e^{\gamma s}|\Delta Y_s|^2$ between $t$ and $T$. After localization and expectation,
\begin{align*}
&\E^Pe^{\gamma t}|\Delta Y_t|^2
+\E^P\int_t^Te^{\gamma s}
\bigl(\gamma|\Delta Y_s|^2+|\Delta Z_s|_{a_s^P}^2+\|\Delta U_s\|_\Pi^2\bigr)ds\\
&\quad=\E^Pe^{\gamma T}|\Delta\zeta|^2
+2\E^P\int_t^Te^{\gamma s}\Delta Y_s
\bigl(\Delta f_s+\langle\varrho,\Delta A_s^\lambda\rangle\bigr)ds.
\end{align*}
Ellipticity gives $|\Delta Z|^2\le\lambda_{\min}(\underline a)^{-1}|\Delta Z|_{a^P}^2$. Applying Young's inequality separately to the terms involving $\Delta Z$, $\Delta U$, $\Delta A^\lambda$, and the remaining residuals yields, for a constant $C$ independent of $\gamma$,
\begin{align*}
2|\Delta Y_s|\bigl(|\Delta f_s|+|\varrho|\,|\Delta A_s^\lambda|\bigr)
\le{}&C|\Delta Y_s|^2
+\frac14|\Delta Z_s|_{a_s^P}^2+\frac14\|\Delta U_s\|_\Pi^2\\
&+C\bigl(|\Delta X_{s-}|^2+W_2^2(\mu_{s-},\bar\mu_{s-})+|\Delta\alpha_s|^2\\
&\hspace{18mm}+|r_s^f|^2+|\Delta A_s^\lambda|^2\bigr)
+C\E^P|\Delta Y_s|^2.
\end{align*}
Choose $\gamma$ large enough to absorb the two contributions in $|\Delta Y|^2$, then take $t=0$. This choice depends only on the Lipschitz constants of $f$, the common ellipticity bounds, and $|\varrho|$; it is independent of $\lambda$. Thus the singularity $\lambda^{-2}$ remains confined to the source estimate \eqref{eq:datastab-source} and does not generate an exponential factor of the form $e^{c\lambda^{-2}T}$. Using \eqref{eq:datastab-terminal-detail}, Step~1, and \eqref{eq:datastab-source}, we obtain
\begin{align}
\sup_{t\le T}\E^P|\Delta Y_t|^2
+\E^P\int_0^T\bigl(|\Delta Z_t|_{a_t^P}^2+\|\Delta U_t\|_\Pi^2\bigr)dt
\le C\{(1+\lambda^{-2})\mathfrak R_F(P)+\mathfrak R_B^\lambda(P)\}.
\label{eq:datastab-energy}
\end{align}

\noindent\emph{Step 5: time supremum.}
From the difference equation,
\begin{align*}
\sup_{t\le T}|\Delta Y_t|
\le{}&|\Delta\zeta|+\int_0^T(|\Delta f_s|+|\varrho|\,|\Delta A_s^\lambda|)ds\\
&+2\sup_{t\le T}\left|\int_0^t\Delta Z_s dB_s\right|
+2\sup_{t\le T}\left|\int_0^t\int_E\Delta U_s(e)\widetilde N^P(ds,de)\right|.
\end{align*}
After squaring and taking expectations, Cauchy--Schwarz and Lemma~\ref{lem:mart} yield
\begin{align*}
\E^P\sup_{t\le T}|\Delta Y_t|^2
\le C\E^P\Bigl[&|\Delta\zeta|^2+
\int_0^T(|\Delta f_s|^2+|\Delta A_s^\lambda|^2)ds\\
&+\int_0^T(|\Delta Z_s|_{a_s^P}^2+\|\Delta U_s\|_\Pi^2)ds\Bigr].
\end{align*}
Steps~1--4 bound every term on the right by $C\{(1+\lambda^{-2})\mathfrak R_F(P)+\mathfrak R_B^\lambda(P)\}$. Together with \eqref{eq:datastab-K}, this proves \eqref{eq:dataYstab}. The final constant is uniform in $P$ and independent of $\lambda$; all $\lambda$-dependence is displayed explicitly through $1+\lambda^{-2}$.
\end{proof}

\begin{remark}[Origin of the $\lambda$-amplification]
The factor $\lambda^{-2}$ comes from $|A^\lambda(x)-A^\lambda(x')|\le\lambda^{-1}|x-x'|$ after squaring; the exponential weight in the BSDE estimate remains independent of $\lambda$.
\end{remark}

\begin{corollary}[Uniform modelwise passage to the limit]\label{cor:datauniform}
Fix $\lambda>0$ and let $(\mathfrak D^n)_n$ be a sequence of data sets satisfying the assumptions of Theorem~\ref{thm:datastability} uniformly. If
\[
\sup_{P\in\Pcal}\mathfrak R_F^n(P)\longrightarrow0,
\qquad
\sup_{P\in\Pcal}\mathfrak R_B^{n,\lambda}(P)\longrightarrow0,
\]
then
\begin{align*}
&\sup_{P\in\Pcal}\E^P\Biggl[
\sup_{t\le T}|X_t^n-X_t|^2+\sup_{t\le T}|Y_t^{n,\lambda}-Y_t^\lambda|^2
+\sup_{t\le T}|\mathsf K_t^{n,\lambda}-\mathsf K_t^\lambda|^2\\
&\hspace{32mm}+\int_0^T|Z_t^{n,\lambda}-Z_t^\lambda|_{a_t^P}^2dt
+\int_0^T\|U_t^{n,\lambda}-U_t^\lambda\|_\Pi^2dt
\Biggr]\longrightarrow0.
\end{align*}
\end{corollary}

\begin{proof}
Apply \eqref{eq:dataXstab}--\eqref{eq:dataYstab} to $(\mathfrak D^n,\mathfrak D)$. For fixed $\lambda$,
\begin{align*}
\sup_P\E^P\sup_{t\le T}|X_t^n-X_t|^2
&\le C\sup_P\mathfrak R_F^n(P),\\
\sup_P\E^P\Bigl[&\sup_{t\le T}|Y_t^{n,\lambda}-Y_t^\lambda|^2
+\sup_{t\le T}|\mathsf K_t^{n,\lambda}-\mathsf K_t^\lambda|^2\\
&+\int_0^T|Z_t^{n,\lambda}-Z_t^\lambda|_{a_t^P}^2dt
+\int_0^T\|U_t^{n,\lambda}-U_t^\lambda\|_\Pi^2dt\Bigr]\\
&\le C\left\{(1+\lambda^{-2})\sup_P\mathfrak R_F^n(P)
+\sup_P\mathfrak R_B^{n,\lambda}(P)\right\}.
\end{align*}
The two quantities on the right tend to zero by assumption. The constants are common to $n$ and $P$ because the structural constants of the data are uniform. Hence every term on the left tends to zero, proving the stated strong convergences.
\end{proof}

\begin{corollary}[Stability of the selected systems]\label{cor:selectedData}
For each $P\in\Pcal$, denote the selected solutions corresponding to the $n$th data set and the limiting data set by
\begin{align*}
&(X^{n,P},Y^{n,P},Z^{n,P},U^{n,P},\mathsf K^{n,P}),\\
&(X^P,Y^P,Z^P,U^P,\mathsf K^P).
\end{align*}
Assume
\[
\sup_{P\in\Pcal}\mathfrak R_F^n(P)\longrightarrow0,
\]
and define
\begin{align*}
\mathfrak R_{g}^n(P)&=
\E^P|g_n(X_T^P,\mu_T^P)-g(X_T^P,\mu_T^P)|^2,\\
\mathfrak R_{f}^n(P)&=
\E^P\int_0^T|f_n(t,X_{t-}^P,\mu_{t-}^P,Y_t^P,Z_t^P,U_t^P,
\eta_t^P,\alpha_t)\\
&\hspace{39mm}-f(t,X_{t-}^P,\mu_{t-}^P,Y_t^P,Z_t^P,U_t^P,
\eta_t^P,\alpha_t)|^2dt.
\end{align*}
If
\[
\sup_{P\in\Pcal}\mathfrak R_g^n(P)\to0,
\qquad
\sup_{P\in\Pcal}\mathfrak R_f^n(P)\to0,
\]
and
\begin{equation}
\sup_{P\in\Pcal}\E^P\int_0^T
|A_n^0(X_{t-}^{n,P},\mu_{t-}^{n,P},\alpha_t^n)
-A^0(X_{t-}^P,\mu_{t-}^P,\alpha_t)|^2dt\longrightarrow0,
\label{eq:selectedDataSource}
\end{equation}
then
\begin{align*}
\sup_{P\in\Pcal}\E^P\Biggl[&
\sup_{t\le T}|Y_t^{n,P}-Y_t^P|^2
+\sup_{t\le T}|\mathsf K_t^{n,P}-\mathsf K_t^P|^2\\
&+\int_0^T|Z_t^{n,P}-Z_t^P|_{a_t^P}^2dt
+\int_0^T\|U_t^{n,P}-U_t^P\|_\Pi^2dt\Biggr]\longrightarrow0.
\end{align*}
\end{corollary}

\begin{proof}
Fix $P\in\Pcal$. The two selected systems are compared first without abbreviations. Their forward components differ by $X^{n,P}-X^P$, their backward components by $Y^{n,P}-Y^P$, their martingale integrands by $Z^{n,P}-Z^P$ and $U^{n,P}-U^P$, and their controls by $\alpha^n-\alpha$. For the estimates, define
\begin{align*}
\Delta X_t&:=X_t^{n,P}-X_t^P,\qquad
\Delta Y_t:=Y_t^{n,P}-Y_t^P,
&& t\in[0,T],\quad P\text{-a.s.},\\
\Delta Z_t&:=Z_t^{n,P}-Z_t^P,\qquad
\Delta\alpha_t:=\alpha_t^n-\alpha_t,
&& dt\otimes P\text{-a.e.},\\
\Delta U_t(e)&:=U_t^{n,P}(e)-U_t^P(e),
&& dt\otimes P\otimes\Pi(de)\text{-a.e.}
\end{align*}
The forward estimate \eqref{eq:dataXstab} gives
\begin{equation}
\E^P\sup_{t\le T}|\Delta X_t|^2
\le C\mathfrak R_F^n(P).
\label{eq:selected-forward-detail}
\end{equation}

For the terminal condition, define
\[
\Delta\zeta_n:=g_n(X_T^{n,P},\mu_T^{n,P})-g(X_T^P,\mu_T^P),
\qquad P\text{-a.s.}
\]
and add and subtract $g_n(X_T^P,\mu_T^P)$:
\begin{align*}
\Delta\zeta_n
&=\bigl[g_n(X_T^{n,P},\mu_T^{n,P})
-g_n(X_T^P,\mu_T^P)\bigr]\\
&\quad+\bigl[g_n(X_T^P,\mu_T^P)-g(X_T^P,\mu_T^P)\bigr].
\end{align*}
Because the Lipschitz constants of $g_n$ are uniform in $n$,
\[
\E^P|\Delta\zeta_n|^2
\le C\mathfrak R_F^n(P)+2\mathfrak R_g^n(P).
\label{eq:selected-terminal-detail}
\]

For the generator, write
\begin{align*}
\Delta f_t^n
={}&f_n(t,X_{t-}^{n,P},\mu_{t-}^{n,P},Y_t^{n,P},
Z_t^{n,P},U_t^{n,P},\eta_t^{n,P},\alpha_t^n)\\
&-f(t,X_{t-}^{P},\mu_{t-}^{P},Y_t^{P},
Z_t^{P},U_t^{P},\eta_t^{P},\alpha_t).
\end{align*}
Add and subtract
\[
f_n(t,X_{t-}^{P},\mu_{t-}^{P},Y_t^{P},
Z_t^{P},U_t^{P},\eta_t^{P},\alpha_t).
\]
Then
\[
\Delta f_t^n=\delta f_t^n+r_t^{f,n},
\]
where $r_t^{f,n}$ is precisely the residual entering $\mathfrak R_f^n(P)$. By the uniform Lipschitz continuity of $f_n$,
\begin{align}
|\delta f_t^n|^2
\le C\bigl(&|\Delta X_{t-}|^2
+W_2^2(\mu_{t-}^{n,P},\mu_{t-}^{P})
+|\alpha_t^n-\alpha_t|^2
+|\Delta Y_t|^2+|\Delta Z_t|^2\nonumber\\
&+\|\Delta U_t\|_\Pi^2
+W_2^2(\eta_t^{n,P},\eta_t^P)\bigr).
\label{eq:selected-generator-detail}
\end{align}
The natural couplings give
\[
W_2^2(\mu_t^{n,P},\mu_t^P)\le\E^P|\Delta X_t|^2,
\qquad
W_2^2(\eta_t^{n,P},\eta_t^P)\le\E^P|\Delta Y_t|^2.
\]

Finally, set
\[
\Delta A_t^{0,n}:=
A_n^0(X_{t-}^{n,P},\mu_{t-}^{n,P},\alpha_t^n)
-A^0(X_{t-}^{P},\mu_{t-}^{P},\alpha_t).
\]
By assumption,
\[
\sup_{P\in\Pcal}\E^P\int_0^T|\Delta A_t^{0,n}|^2dt\longrightarrow0.
\]
The difference of the cumulative sources satisfies
\[
\mathsf K_t^{n,P}-\mathsf K_t^P
=\int_0^t\langle\varrho,\Delta A_s^{0,n}\rangle ds,
\]
so
\begin{equation}
\E^P\sup_{t\le T}|\mathsf K_t^{n,P}-\mathsf K_t^P|^2
\le T|\varrho|^2\E^P\int_0^T|\Delta A_s^{0,n}|^2ds.
\label{eq:selected-K-detail}
\end{equation}

Apply It\^o's formula with jumps to $e^{\gamma t}|\Delta Y_t|^2$. Using \eqref{eq:selected-generator-detail}, ellipticity $|\Delta Z_t|^2\le\lambda_{\min}(\underline a)^{-1}|\Delta Z_t|_{a_t^P}^2$, and Young's inequality, choose $\gamma$ large enough, independently of $n$ and $P$, to absorb the terms in $|\Delta Y|^2$, $|\Delta Z|_{a^P}^2$, and $\|\Delta U\|_\Pi^2$. This gives
\begin{align}
&\sup_{t\le T}\E^P|\Delta Y_t|^2
+\E^P\int_0^T\bigl(
|\Delta Z_t|_{a_t^P}^2+\|\Delta U_t\|_\Pi^2\bigr)dt\nonumber\\
&\quad\le C\biggl\{
\mathfrak R_F^n(P)+\mathfrak R_g^n(P)+\mathfrak R_f^n(P)
+\E^P\int_0^T|\Delta A_t^{0,n}|^2dt
\biggr\}.
\label{eq:selected-energy-detail}
\end{align}
To control the time supremum, we first display the backward subtraction in its original variables:
\begin{align*}
Y_t^{n,P}-Y_t^P
={}&g_n(X_T^{n,P},\mu_T^{n,P})-g(X_T^P,\mu_T^P)\\
&+\int_t^T\Bigl[
 f_n(s,X_{s-}^{n,P},\mu_{s-}^{n,P},Y_s^{n,P},Z_s^{n,P},U_s^{n,P},\eta_s^{n,P},\alpha_s^n)\\
&\hspace{20mm}-f(s,X_{s-}^{P},\mu_{s-}^{P},Y_s^{P},Z_s^{P},U_s^{P},\eta_s^{P},\alpha_s)\\
&\hspace{20mm}+\langle\varrho,
 A_n^0(X_{s-}^{n,P},\mu_{s-}^{n,P},\alpha_s^n)
-A^0(X_{s-}^{P},\mu_{s-}^{P},\alpha_s)\rangle\Bigr]ds\\
&-\int_t^T\bigl(Z_s^{n,P}-Z_s^P\bigr)dB_s\\
&-\int_t^T\int_E\bigl(U_s^{n,P}(e)-U_s^P(e)\bigr)\widetilde N^P(ds,de),
\qquad t\in[0,T],\quad P\text{-a.s.}
\end{align*}
Using the abbreviations already introduced, the same identity is
\begin{align*}
\Delta Y_t
={}&\Delta\zeta_n+\int_t^T
\bigl(\delta f_s^n+r_s^{f,n}
+\langle\varrho,\Delta A_s^{0,n}\rangle\bigr)ds\\
&-\int_t^T\Delta Z_s\,dB_s
-\int_t^T\int_E\Delta U_s(e)\widetilde N^P(ds,de),
\qquad t\in[0,T],\quad P\text{-a.s.}
\end{align*}
By the triangle inequality,
\begin{align*}
\sup_{t\le T}|\Delta Y_t|
\le{}&|\Delta\zeta_n|
+\int_0^T\bigl(|\delta f_s^n|+|r_s^{f,n}|
+|\varrho|\,|\Delta A_s^{0,n}|\bigr)ds\\
&+2\sup_{t\le T}\left|\int_0^t\Delta Z_s\,dB_s\right|
+2\sup_{t\le T}\left|\int_0^t\int_E\Delta U_s(e)
\widetilde N^P(ds,de)\right|.
\end{align*}
After squaring, applying Cauchy--Schwarz in time, and using Lemma~\ref{lem:mart},
\begin{align*}
\E^P\sup_{t\le T}|\Delta Y_t|^2
\le C\E^P\Biggl[&
|\Delta\zeta_n|^2
+\int_0^T\bigl(|\delta f_s^n|^2+|r_s^{f,n}|^2
+|\Delta A_s^{0,n}|^2\bigr)ds\\
&+\int_0^T\bigl(
|\Delta Z_s|_{a_s^P}^2+\|\Delta U_s\|_\Pi^2\bigr)ds
\Biggr].
\end{align*}
Estimate \eqref{eq:selected-generator-detail}, the two coupling inequalities, \eqref{eq:selected-forward-detail}, \eqref{eq:selected-terminal-detail}, and the energy estimate \eqref{eq:selected-energy-detail} control all terms on the right. Therefore,
\begin{align*}
&\E^P\sup_{t\le T}|\Delta Y_t|^2
+\E^P\int_0^T\bigl(
|\Delta Z_t|_{a_t^P}^2+\|\Delta U_t\|_\Pi^2\bigr)dt\\
&\quad\le C\biggl\{
\mathfrak R_F^n(P)+\mathfrak R_g^n(P)+\mathfrak R_f^n(P)
+\E^P\int_0^T|\Delta A_t^{0,n}|^2dt
\biggr\}.
\end{align*}
Adding \eqref{eq:selected-K-detail} and taking the supremum over $P\in\Pcal$, every term on the right tends to zero by assumption, proving the stated convergence.
\end{proof}

Condition \eqref{eq:selectedDataSource} does not follow from the Lipschitz assumptions on the forward equation alone because $A^0$ may be discontinuous on the moving set $\Sigma(\alpha_t)$. It must be verified through a non-contact condition, an occupation estimate, or uniform convergence of the selections.

\section{Yosida Approximation and Graph Closedness}

Yosida regularization first provides a Lipschitz approximation in the general monotone framework. For the ALA prototype, occupation of the interface then yields a quantitative rate. Minty-type closedness identifies weak limits when the state, its law, and the control---and therefore the moving interface---vary.

Lemma~\ref{lem:lawcontinuous} ensures, in particular, continuity of the law flows at deterministic times. We retain left limits below because they provide predictable integrands in the jump setting.

For $\lambda>0$, replace \eqref{eq:Kdef} by
\begin{equation}
\mathsf K_t^{\lambda,P}
=\int_0^t\langle\varrho,
A^\lambda(X_{s-}^P,\mu_{s-}^P,\alpha_s)\rangle ds,
\qquad t\in[0,T],\quad P\text{-a.s.}
\label{eq:Klambda}
\end{equation}
and let $(Y^{\lambda,P},Z^{\lambda,P},U^{\lambda,P})$ denote the corresponding backward solution, with
\[
\eta_t^{P,\lambda}=\Lcal^P(Y_t^{\lambda,P}).
\]
The regularized solution exists and is unique by the construction in Theorem~\ref{thm:Bwell}. Its data satisfy the required conditions because
\[
\chi_t^{\lambda,P}=\langle\varrho,A^\lambda(X_{t-}^P,\mu_{t-}^P,\alpha_t)\rangle,
\qquad |\chi_t^{\lambda,P}|^2\le|\varrho|^2|A^0(X_{t-}^P,\mu_{t-}^P,\alpha_t)|^2.
\]
The resolvent is Borel measurable, hence $\chi^{\lambda,P}$ is predictable, and
\[
\sup_{P,\lambda>0}\E^P\int_0^T|\chi_t^{\lambda,P}|^2dt<\infty.
\]
For $F^{\lambda,P}=f(\cdot,X_-,\mu_-,\cdot,\alpha)+\chi^{\lambda,P}$,
\begin{align*}
\sup_{P,\lambda>0}\E^P\int_0^T
|F^{\lambda,P}(t,0,0,0,\delta_0)|^2dt&<\infty,\\
F^{\lambda,P}(t,v)-F^{\lambda,P}(t,v')
&=F^{P}(t,v)-F^{P}(t,v').
\end{align*}
Thus Theorem~\ref{thm:Bwell} applies with the same contraction and energy constants for every $\lambda>0$. The forward process is unchanged.

\begin{lemma}[Quadratic convergence of the resolvents]\label{lem:Jstrong}
Let $\alpha\in\mathcal A^2$. Under Assumptions~\ref{HF} and~\ref{HD}, for every $\lambda>0$,
\[
\sup_{P\in\Pcal}\E^P\int_0^T
|J_\lambda(X_{t-}^P,\mu_{t-}^P,\alpha_t)-X_{t-}^P|^2dt\le C\lambda^2,
\]
where $C$ is independent of $P$ and $\lambda$.
\end{lemma}
\begin{proof}
For $dt\otimes P$-a.e. $(t,\omega)$, the resolvent identity gives
\[
J_\lambda(X_{t-}^P,\mu_{t-}^P,\alpha_t)-X_{t-}^P
=-\lambda A^\lambda(X_{t-}^P,\mu_{t-}^P,\alpha_t).
\]
By the domination in \eqref{eq:Aconv},
\begin{align*}
&\E^P\int_0^T|J_\lambda(X_{t-}^P,\mu_{t-}^P,\alpha_t)-X_{t-}^P|^2dt\\
&\qquad=\lambda^2\E^P\int_0^T|A^\lambda(X_{t-}^P,\mu_{t-}^P,\alpha_t)|^2dt\\
&\qquad\le\lambda^2\E^P\int_0^T|A^0(X_{t-}^P,\mu_{t-}^P,\alpha_t)|^2dt.
\end{align*}
The growth of the minimal-norm selection yields
\[
|A^0(X_{t-}^P,\mu_{t-}^P,\alpha_t)|^2
\le C_A^2(1+|X_{t-}^P|+|\alpha_t|)^2
\le3C_A^2(1+|X_{t-}^P|^2+|\alpha_t|^2).
\]
Therefore,
\begin{align*}
&\sup_P\E^P\int_0^T|A^0(X_{t-}^P,\mu_{t-}^P,\alpha_t)|^2dt\\
&\qquad\le C\left(T+T\sup_P\E^P\sup_{t\le T}|X_t^P|^2
+\sup_P\E^P\int_0^T|\alpha_t|^2dt\right)<\infty
\end{align*}
by \eqref{eq:Fsecond} and $\alpha\in\mathcal A^2$. Multiplication by $\lambda^2$ gives the stated bound. No occupation assumption is used.
\end{proof}

\paragraph{Role of the quadratic resolvent estimate.}
Lemma~\ref{lem:Jstrong} isolates the geometric displacement created by regularization: $J_\lambda(X^P,\mu^P,\alpha)-X^P=O(\lambda)$ in $L^2(dt\otimes P)$, uniformly in $P$. The next lemma obtains convergence of the source directly from $|A^\lambda|\le|A^0|$. We keep the resolvent estimate separate because it quantifies the displacement of the graph argument and is used conceptually in the interface arguments of Section~4.1.

\begin{lemma}[Convergence under a fixed model]\label{lem:fixedY}
Let $\alpha\in\mathcal A^2$. Under Assumptions~\ref{HF}, \ref{HB}, and~\ref{HD}, for every fixed $P\in\Pcal$,
\[
\E^P\int_0^T
|A^\lambda(X_{t-}^P,\mu_{t-}^P,\alpha_t)
-A^0(X_{t-}^P,\mu_{t-}^P,\alpha_t)|^2dt\longrightarrow0
\qquad(\lambda\downarrow0).
\]
Consequently, \eqref{eq:Yconv} holds under this model $P$, without the supremum over $\Pcal$.
\end{lemma}

\begin{proof}
Fix $P$. For $dt\otimes P$-a.e. $(t,\omega)$, the pointwise convergence in \eqref{eq:Aconv} gives
\[
A^\lambda(X_{t-}^P,\mu_{t-}^P,\alpha_t)
\longrightarrow A^0(X_{t-}^P,\mu_{t-}^P,\alpha_t).
\]
Moreover,
\[
|A^\lambda-A^0|^2
\le2|A^\lambda|^2+2|A^0|^2
\le4|A^0|^2.
\]
The dominating function is integrable because Assumption~\ref{HD} gives
\[
\E^P\int_0^T|A^0(X_{t-}^P,\mu_{t-}^P,\alpha_t)|^2dt
\le C\left(T+T\E^P\sup_{t\le T}|X_t^P|^2+
\E^P\int_0^T|\alpha_t|^2dt\right)<\infty.
\]
The dominated convergence theorem therefore yields
\[
\E^P\int_0^T|A^\lambda(X_{t-}^P,\mu_{t-}^P,\alpha_t)-A^0(X_{t-}^P,\mu_{t-}^P,\alpha_t)|^2dt\to0.
\]
The regularized and selected systems have the same forward process, terminal condition, and coefficient $f$; only the exogenous source differs. Lemma~\ref{lem:Bstab} then gives
\begin{align*}
&\E^P\sup_{t\le T}|Y_t^{\lambda,P}-Y_t^P|^2
+\E^P\int_0^T\bigl(|Z_t^{\lambda,P}-Z_t^P|_{a_t^P}^2
+\|U_t^{\lambda,P}-U_t^P\|_\Pi^2\bigr)dt\\
&\qquad\le C|\varrho|^2\E^P\int_0^T|A^\lambda(X_{t-}^P,\mu_{t-}^P,\alpha_t)-A^0(X_{t-}^P,\mu_{t-}^P,\alpha_t)|^2dt\to0.
\end{align*}
Finally,
\[
\sup_{t\le T}|\mathsf K_t^{\lambda,P}-\mathsf K_t^P|^2
\le T|\varrho|^2\int_0^T|A^\lambda-A^0|^2dt,
\]
which also gives convergence of $\mathsf K$ and proves the version of \eqref{eq:Yconv} without the supremum over $P$.
\end{proof}

\begin{hypothesis}[Uniform convergence of the source]\label{HY}
There exists $\delta>0$ such that
\begin{equation}
\sup_{P\in\Pcal}\E^P\int_0^T
|A^0(X_{t-}^P,\mu_{t-}^P,\alpha_t)|^{2+\delta}dt<\infty,
\label{eq:AUI}
\end{equation}
and, for every $\varepsilon>0$,
\begin{equation}
\lim_{\lambda\downarrow0}\sup_{P\in\Pcal}
(dt\otimes P)\left(
|A^\lambda(X_{t-}^P,\mu_{t-}^P,\alpha_t)
-A^0(X_{t-}^P,\mu_{t-}^P,\alpha_t)|>\varepsilon\right)=0.
\label{eq:Ameasure}
\end{equation}
\end{hypothesis}

\begin{lemma}[Uniform Vitali argument]\label{lem:Vitali}
Let $\alpha\in\mathcal A^2$. Under Assumptions~\ref{HF}, \ref{HD}, and~\ref{HY},
\begin{equation}
\lim_{\lambda\downarrow0}\sup_{P\in\Pcal}
\E^P\int_0^T|A^\lambda(X_{t-}^P,\mu_{t-}^P,\alpha_t)
-A^0(X_{t-}^P,\mu_{t-}^P,\alpha_t)|^2dt=0.
\label{eq:Astrong}
\end{equation}
\end{lemma}

\begin{proof}
Set
\[
D_{\lambda,t}^P=
A^\lambda(X_{t-}^P,\mu_{t-}^P,\alpha_t)
-A^0(X_{t-}^P,\mu_{t-}^P,\alpha_t).
\]
The domination in \eqref{eq:Aconv} yields
\[
|D_{\lambda,t}^P|
\le |A^\lambda(X_{t-}^P,\mu_{t-}^P,\alpha_t)|
+|A^0(X_{t-}^P,\mu_{t-}^P,\alpha_t)|
\le2|A^0(X_{t-}^P,\mu_{t-}^P,\alpha_t)|.
\]
Hence, if $\delta>0$ is as in \eqref{eq:AUI},
\begin{align}
\sup_{P\in\Pcal}\sup_{0<\lambda\le1}
\E^P\int_0^T|D_{\lambda,t}^P|^{2+\delta}dt
&\le 2^{2+\delta}
\sup_{P\in\Pcal}
\E^P\int_0^T|A^0(X_{t-}^P,\mu_{t-}^P,\alpha_t)|^{2+\delta}dt\\
&<\infty.
\label{eq:Vitali-uniform-moment}
\end{align}

Fix $\varepsilon>0$ and split the integral over
$\{|D_\lambda^P|\le\varepsilon\}$ and $\{|D_\lambda^P|>\varepsilon\}$. On the first set,
\[
\E^P\int_0^T|D_{\lambda,t}^P|^2
\one_{\{|D_\lambda^P|\le\varepsilon\}}dt
\le T\varepsilon^2.
\]
On the second set, H\"older's inequality with exponents $(2+\delta)/2$ and $(2+\delta)/\delta$ gives
\begin{align*}
&\E^P\int_0^T|D_{\lambda,t}^P|^2
\one_{\{|D_\lambda^P|>\varepsilon\}}dt\\
&\qquad\le
\left(\E^P\int_0^T|D_{\lambda,t}^P|^{2+\delta}dt\right)^{2/(2+\delta)}
\left((dt\otimes P)(|D_\lambda^P|>\varepsilon)\right)^{\delta/(2+\delta)}.
\end{align*}
Using \eqref{eq:Vitali-uniform-moment}, there is a constant $C$, independent of $P$ and $\lambda$, such that
\begin{equation}
\E^P\int_0^T|D_{\lambda,t}^P|^2dt
\le T\varepsilon^2+C
\left((dt\otimes P)(|D_\lambda^P|>\varepsilon)\right)^{\delta/(2+\delta)}.
\label{eq:Vitali-split}
\end{equation}
Set
\[
q_\lambda(\varepsilon):=
\sup_{P\in\Pcal}(dt\otimes P)(|D_\lambda^P|>\varepsilon).
\]
Assumption~\ref{HY}, through \eqref{eq:Ameasure}, gives $q_\lambda(\varepsilon)\to0$ as $\lambda\downarrow0$. Taking the supremum over $P$ in \eqref{eq:Vitali-split} and then the $\limsup$ as $\lambda\downarrow0$ yields
\[
0\le\limsup_{\lambda\downarrow0}\sup_{P\in\Pcal}
\E^P\int_0^T|D_{\lambda,t}^P|^2dt
\le T\varepsilon^2.
\]
Since this holds for every $\varepsilon>0$, letting $\varepsilon\downarrow0$ gives \eqref{eq:Astrong}.
\end{proof}

\begin{theorem}[Strong Yosida convergence]\label{thm:Yconv}
Let $\alpha\in\mathcal A^2$. Under Assumptions~\ref{HF}, \ref{HB}, \ref{HD}, and~\ref{HY},
\begin{align}
\lim_{\lambda\downarrow0}\sup_{P\in\Pcal}\E^P\Bigl[
&\sup_{t\le T}|Y_t^{\lambda,P}-Y_t^P|^2
+\sup_{t\le T}|\mathsf K_t^{\lambda,P}-\mathsf K_t^P|^2\nonumber\\
&+\int_0^T|Z_t^{\lambda,P}-Z_t^P|_{a_t^P}^2dt
+\int_0^T\|U_t^{\lambda,P}-U_t^P\|_\Pi^2dt
\Bigr]=0.
\label{eq:Yconv}
\end{align}
\end{theorem}

\begin{proof}
Set
\[
D_{\lambda,t}^P=
A^\lambda(X_{t-}^P,\mu_{t-}^P,\alpha_t)
-A^0(X_{t-}^P,\mu_{t-}^P,\alpha_t).
\]
The regularized and selected systems have exactly the same forward process $X^P$, law flow $\mu^P$, control, and terminal condition. Thus, in Lemma~\ref{lem:Bstab}, every perturbation term vanishes except the exogenous source
\[
\chi_t^{\lambda,P}-\chi_t^P=\langle\varrho,D_{\lambda,t}^P\rangle.
\]
Hence
\begin{align*}
&\E^P\sup_{t\le T}|Y_t^{\lambda,P}-Y_t^P|^2
+\E^P\int_0^T\bigl(|Z_t^{\lambda,P}-Z_t^P|_{a_t^P}^2
+\|U_t^{\lambda,P}-U_t^P\|_\Pi^2\bigr)dt\\
&\qquad\le C|\varrho|^2\E^P\int_0^T|D_{\lambda,t}^P|^2dt.
\end{align*}
For the cumulative source,
\begin{align*}
\sup_{t\le T}|\mathsf K_t^{\lambda,P}-\mathsf K_t^P|^2
&=\sup_{t\le T}\left|\int_0^t\langle\varrho,D_{\lambda,s}^P\rangle ds\right|^2\\
&\le T|\varrho|^2\int_0^T|D_{\lambda,s}^P|^2ds.
\end{align*}
Therefore,
\begin{align*}
&\E^P\Bigl[\sup_{t\le T}|Y_t^{\lambda,P}-Y_t^P|^2
+\sup_{t\le T}|\mathsf K_t^{\lambda,P}-\mathsf K_t^P|^2\\
&\hspace{22mm}+\int_0^T|Z_t^{\lambda,P}-Z_t^P|_{a_t^P}^2dt
+\int_0^T\|U_t^{\lambda,P}-U_t^P\|_\Pi^2dt\Bigr]\\
&\qquad\le C(1+T)|\varrho|^2\E^P\int_0^T|D_{\lambda,t}^P|^2dt.
\end{align*}
By Lemma~\ref{lem:Vitali}, the supremum over $P$ of the right-hand side tends to zero as $\lambda\downarrow0$, proving \eqref{eq:Yconv}.
\end{proof}

The next proposition treats the case in which the states vary and the selections converge only weakly.

\begin{proposition}[Strong--weak closedness with variable law and interface]\label{prop:closure}
Fix $P\in\Pcal$. Let $\alpha^n,\alpha\in\mathcal A^2$ satisfy
\[
\alpha^n\longrightarrow\alpha
\quad\text{strongly in }L^2(dt\otimes P;\R^k),
\]
and let $(X^n)_n\subset L^2(dt\otimes P;\R^d)$ and $(\Gamma^n)_n\subset L^2(dt\otimes P;\R^d)$ satisfy
\begin{align*}
X^n&\longrightarrow X
\quad\text{strongly in }L^2(dt\otimes P;\R^d),\\
\Gamma^n&\rightharpoonup\Gamma
\quad\text{weakly in }L^2(dt\otimes P;\R^d),\\
\Gamma_t^n&\in A(X_t^n,\Lcal^P(X_t^n),\alpha_t^n)
\quad dt\otimes P\text{-a.e.}
\end{align*}
Then
\[
\Gamma_t\in A(X_t,\Lcal^P(X_t),\alpha_t)
\quad dt\otimes P\text{-a.e.}
\]
\end{proposition}

\begin{proof}
Set $\mu_t^n=\Lcal^P(X_t^n)$ and $\mu_t=\Lcal^P(X_t)$. The canonical coupling gives, for a.e. $t$,
\[
W_2^2(\mu_t^n,\mu_t)\le\E^P|X_t^n-X_t|^2.
\]
After integration,
\begin{equation}
\int_0^TW_2^2(\mu_t^n,\mu_t)dt
\le\E^P\int_0^T|X_t^n-X_t|^2dt\longrightarrow0.
\label{eq:Minty-law-detail}
\end{equation}
Since $c$ is $W_2$-Lipschitz,
\[
\int_0^T|c(\mu_t^n)-c(\mu_t)|^2dt
\le L_c^2\int_0^TW_2^2(\mu_t^n,\mu_t)dt\to0.
\]
Thus $c(\mu^n)\to c(\mu)$ strongly in $L^2(dt\otimes P;\R^d)$; these functions are deterministic in $\omega$, which does not change the norm.

Define
\[
S_t^n=X_t^n+h(\alpha_t^n),\qquad S_t=X_t+h(\alpha_t),
\]
and
\[
H_t^n=\Gamma_t^n-c(\mu_t^n),\qquad H_t=\Gamma_t-c(\mu_t).
\]
The Lipschitz continuity of $h$ gives
\begin{align*}
\|S^n-S\|_{L^2(dt\otimes P)}
&\le \|X^n-X\|_{L^2(dt\otimes P)}
+\|h(\alpha^n)-h(\alpha)\|_{L^2(dt\otimes P)}\\
&\le \|X^n-X\|_{L^2(dt\otimes P)}
+L_h\|\alpha^n-\alpha\|_{L^2(dt\otimes P)}\longrightarrow0.
\end{align*}
Moreover, as the sum of a weakly convergent sequence and a strongly convergent one,
\[
H^n\rightharpoonup H\quad\text{in }L^2(dt\otimes P;\R^d).
\]
The initial inclusion becomes
\[
H_t^n\in\partial\phi(S_t^n)\qquad dt\otimes P\text{-a.e.}
\]

Now take $(V,W)\in L^2\times L^2$ such that $W_t\in\partial\phi(V_t)$ a.e. Monotonicity of $\partial\phi$ gives pointwise
\[
\langle H_t^n-W_t,S_t^n-V_t\rangle\ge0.
\]
After integration,
\begin{equation}
\E^P\int_0^T\langle H_t^n-W_t,S_t^n-V_t\rangle dt\ge0.
\label{eq:Minty-n}
\end{equation}
Expand the inner product:
\begin{align*}
\int\langle H^n-W,S^n-V\rangle
={}&\int\langle H^n,S-V\rangle
+\int\langle H^n,S^n-S\rangle\\
&-\int\langle W,S^n-V\rangle.
\end{align*}
The first term converges to $\int\langle H,S-V\rangle$ by weak convergence of $H^n$. The third converges by strong convergence of $S^n$. For the second,
\[
\left|\E^P\int_0^T\langle H_t^n,S_t^n-S_t\rangle dt\right|
\le\|H^n\|_{L^2}\|S^n-S\|_{L^2}\to0,
\]
because every weakly convergent sequence is bounded. Passing to the limit in \eqref{eq:Minty-n} gives
\begin{equation}
\E^P\int_0^T\langle H_t-W_t,S_t-V_t\rangle dt\ge0.
\label{eq:Minty-limit}
\end{equation}

To identify the graph, use the resolvent $J=(I+\partial\phi)^{-1}$, which is $1$-Lipschitz. Since $\phi$ is finite on $\R^d$, $J(0)\in\R^d$ is well defined and finite. For every $z\in\R^d$,
\[
|J(z)|\le |J(z)-J(0)|+|J(0)|\le |z|+|J(0)|.
\]
Since $S+H\in L^2(dt\otimes P;\R^d)$ and $T<\infty$,
\[
\E^P\int_0^T|J(S_t+H_t)|^2dt
\le 2\E^P\int_0^T|S_t+H_t|^2dt+2T|J(0)|^2<\infty.
\]
Thus
\[
V:=J(S+H)
\]
belongs to $L^2$, and
\[
W:=S+H-V\in L^2.
\]
By definition of the resolvent, $W_t\in\partial\phi(V_t)$ a.e., so this choice is admissible in \eqref{eq:Minty-limit}. Since $W=S+H-V$,
\begin{align*}
\langle H-W,S-V\rangle
&=\langle H-(S+H-V),S-V\rangle\\
&=\langle V-S,S-V\rangle=-|S-V|^2.
\end{align*}
Hence
\[
0\le-\E^P\int_0^T|S_t-V_t|^2dt\le0.
\]
Therefore $S=V$ a.e. By the definition of $W$, $H=W$, and since $W\in\partial\phi(V)$,
\[
H_t\in\partial\phi(S_t)\qquad dt\otimes P\text{-a.e.}
\]
Finally, $S=X+h(\alpha)$ and $\Gamma=H+c(\mu)$, so
\[
\Gamma_t\in\partial\phi(X_t+h(\alpha_t))+c(\mu_t)
=A(X_t,\mu_t,\alpha_t)
\]
a.e.
\end{proof}

\paragraph{Role of strong--weak closedness.}
Proposition~\ref{prop:closure} is not needed for Theorem~\ref{thm:Yconv}, where the forward process and control are fixed and strong convergence of the source is obtained directly. It instead identifies graph limits when the state, its law, and the control vary while the source converges only weakly. The conclusion is deliberately limited to $\Gamma\in A(X,\Lcal^P(X),\alpha)$: graph membership alone does not identify the minimal-norm selection $A^0$ without additional information.

\subsection{Regularization Rates and Shifted Occupation}

The Yosida rate depends on the amount of time spent near the interface. We therefore use occupation estimates; see \cite{Avikainen2009,FournierPrintems2010,TissotDaguette2026,Friesen2026}.

The singular geometry is described by the moving distance $\rho(x,a)=\operatorname{dist}(x,\Sigma(a))$ from \eqref{eq:Sigma}. For the ALA prototype, the connected components of $\R^d\setminus\Sigma(a)$ are translated open orthants. On each component the operator is single-valued, and, for fixed $(\mu,a)\in\mathcal P_2(\R^d)\times D$ and all $x,x'$ belonging to the same connected component of $\R^d\setminus\Sigma(a)$,
\begin{equation}
|A(x,\mu,a)-A(x',\mu,a)|\le \alpha_A|x-x'|.
\label{eq:mobile-face-lip}
\end{equation}

\begin{hypothesis}[Quantitative shifted occupation]\label{HoccJ}
There exist $\theta,\delta>0$ and $C_{\rm occ}<\infty$ such that
\begin{align}
\sup_{P\in\Pcal}\E^P\int_0^T
|A^0(X_{t-}^P,\mu_{t-}^P,\alpha_t)|^{2+\delta}dt&<\infty,
\label{eq:occmoment}\\
\sup_{P\in\Pcal}\E^P\int_0^T
\one_{\{\rho(X_{t-}^P,\alpha_t)\le r\}}dt
&\le C_{\rm occ}r^\theta,
\qquad 0<r\le1.
\label{eq:occ}
\end{align}
\end{hypothesis}

For the moving ALA source, the moment condition is automatic under Assumption~\ref{HqJ} for every $\delta\in(0,q-2]$. Indeed, \eqref{eq:A0growth}, H\"older's inequality, and \eqref{eq:Fq} yield
\begin{align*}
&\sup_P\E^P\int_0^T
|A^0(X_{t-}^P,\mu_{t-}^P,\alpha_t)|^{2+\delta}dt\\
&\qquad\le C\left(1+\sup_P\E^P\sup_{t\le T}|X_t^P|^{2+\delta}
+\sup_P\E^P\int_0^T|\alpha_t|^{2+\delta}dt\right)<\infty.
\end{align*}

\begin{proposition}[Uniform qualitative non-contact for the moving interface]\label{prop:noncontactJ}
Let $\delta>0$. Assume the quadratic well-posedness conditions and
\[
\sup_{P\in\Pcal}\E^P\int_0^T
|A^0(X_{t-}^P,\mu_{t-}^P,\alpha_t)|^{2+\delta}dt<\infty.
\]
If
\[
\omega(r):=\sup_{P\in\Pcal}\E^P\int_0^T
\one_{\{\rho(X_{t-}^P,\alpha_t)\le r\}}dt
\longrightarrow0\qquad(r\downarrow0),
\]
then Assumption~\ref{HY} holds and the convergence in Theorem~\ref{thm:Yconv} is uniform over $\Pcal$.
\end{proposition}

\begin{proof}
For notational simplicity, fix $(P,t)$ and set
\[
x=X_{t-}^P,\qquad \mu=\mu_{t-}^P,\qquad a=\alpha_t,
\qquad J_\lambda=J_\lambda(x,\mu,a).
\]
For $0<r\le1$, introduce
\[
G_{\lambda,r}:=\{\rho(x,a)>r,\ \lambda|A^0(x,\mu,a)|\le r/2\}.
\]
On this set, \eqref{eq:Yosida-displacement} and $|A^\lambda|\le|A^0|$ give
\[
|J_\lambda-x|=\lambda|A^\lambda(x,\mu,a)|
\le\lambda|A^0(x,\mu,a)|\le r/2.
\]
Since $\rho(x,a)>r$, every coordinate of $x+h(a)$ lies at distance strictly greater than $r$ from zero. A displacement of norm at most $r/2$ cannot change the sign of any coordinate. Hence $x+h(a)$ and $J_\lambda+h(a)$ belong to the same translated open orthant. On that orthant the ALA selection is affine in the state with slope $\alpha_A$, and
\[
A^\lambda(x,\mu,a)=A(J_\lambda,\mu,a).
\]
Consequently,
\[
|A^\lambda(x,\mu,a)-A^0(x,\mu,a)|
\le\alpha_A|J_\lambda-x|
\le\alpha_A\lambda|A^0(x,\mu,a)|
\quad\text{on }G_{\lambda,r}.
\]
On $G_{\lambda,r}^c$, Yosida domination gives
\[
|A^\lambda-A^0|\le|A^\lambda|+|A^0|\le2|A^0|.
\]
Thus, with $D_\lambda=A^\lambda-A^0$,
\begin{align*}
\E^P\int_0^T|D_{\lambda,t}|^2dt
\le{}&\alpha_A^2\lambda^2\E^P\int_0^T|A_t^0|^2dt
+4\E^P\int_0^T|A_t^0|^2\one_{G_{\lambda,r}^c}dt.
\end{align*}
The complement satisfies
\[
G_{\lambda,r}^c
\subset\{\rho(X_{t-}^P,\alpha_t)\le r\}
\cup\{|A_t^0|>r/(2\lambda)\}.
\]
For the first set, H\"older's inequality with exponents $(2+\delta)/2$ and $(2+\delta)/\delta$ yields
\begin{align*}
\E^P\int_0^T|A_t^0|^2\one_{\{\rho\le r\}}dt
&\le\left(\E^P\int_0^T|A_t^0|^{2+\delta}dt\right)^{2/(2+\delta)}
\left((dt\otimes P)(\rho\le r)\right)^{\delta/(2+\delta)}\\
&\le C\,\omega(r)^{\delta/(2+\delta)}.
\end{align*}
For the second set, the elementary inequality
\[
|v|^2\one_{\{|v|>M\}}\le M^{-\delta}|v|^{2+\delta}
\]
with $M=r/(2\lambda)$ gives
\[
\E^P\int_0^T|A_t^0|^2\one_{\{|A_t^0|>r/(2\lambda)\}}dt
\le C(\lambda/r)^\delta.
\]
Taking the supremum over $P$,
\begin{equation}
\sup_P\E^P\int_0^T|D_{\lambda,t}^P|^2dt
\le C\{\lambda^2+\omega(r)^{\delta/(2+\delta)}+(\lambda/r)^\delta\}.
\label{eq:noncontact-detail}
\end{equation}
Fix $r>0$ and let $\lambda\downarrow0$; the first and third terms vanish. Then let $r\downarrow0$ and use $\omega(r)\to0$ to remove the remaining term. This proves uniform convergence of the source in $L^2(dt\otimes P)$. Finally, for every $\varepsilon>0$, Markov's inequality gives
\[
\sup_P(dt\otimes P)(|D_\lambda^P|>\varepsilon)
\le\varepsilon^{-2}\sup_P\E^P\int_0^T|D_{\lambda,t}^P|^2dt\to0.
\]
Both conditions in Assumption~\ref{HY} are therefore satisfied.
\end{proof}

\begin{hypothesis}[Absolutely continuous motion of the interface]\label{HmoveAC}
The control has an absolutely continuous representative
\[
\alpha_t=\alpha_0+\int_0^t\dot\alpha_sds,\qquad t\in[0,T],
\]
with
\[
\sup_{P\in\Pcal}\E^P\int_0^T|\dot\alpha_s|^2ds<\infty.
\]
The shift $h$ is Lipschitz as in \eqref{eq:hlip}. Consequently, $t\mapsto h(\alpha_t)$ is absolutely continuous and has a derivative $\dot H_t$ for a.e. $t\in[0,T]$, satisfying $|\dot H_t|\le L_h|\dot\alpha_t|$.
\end{hypothesis}

\begin{hypothesis}[Continuous nondegeneracy at the moving ALA interfaces]
\label{HellJ}
Set $\sigma_t^P=\sigma(t,X_{t-}^P,\mu_{t-}^P,\alpha_t)$ and $q_t^P=\sigma_t^Pa_t^P(\sigma_t^P)^\top$.
There exists $c_*>0$, independent of $P$, such that $(q_t^P)_{kk}\ge c_*$ for $k=1,\ldots,d$, $dt\otimes P$-a.e.
\end{hypothesis}

Nondegeneracy of the continuous martingale part yields the following occupation estimate.

\begin{theorem}[Occupation of a control-dependent moving interface --- ALA prototype]\label{thm:occupation-jumps}
Let $\alpha\in\mathcal A^2$ satisfy Assumption~\ref{HmoveAC}. Under Assumptions~\ref{HF}, \ref{HD}, and~\ref{HellJ}, there exists a constant $C<\infty$, independent of $P$ and $r$, such that
\begin{equation}
\sup_{P\in\Pcal}\E^P\int_0^T
\one_{\{\rho(X_{t-}^P,\alpha_t)\le r\}}dt
\le Cr,
\qquad 0<r\le1.
\label{eq:occ-jumps-proved}
\end{equation}
If Assumption~\ref{HqJ} also holds, then Assumption~\ref{HoccJ} is satisfied with $\theta=1$ and every $\delta\in(0,q-2]$.
\end{theorem}

\begin{remark}
Assumption~\ref{HellJ} requires a nondegenerate continuous martingale component in every direction normal to the interface. The preceding theorem therefore does not cover the purely discontinuous regime. In that case, an occupation estimate must be obtained by another mechanism, for example from transition densities or assumptions tailored to the jump law.
\end{remark}

\begin{proof}
Fix $P\in\Pcal$ and a coordinate $k\in\{1,\ldots,d\}$. Set
\[
S_t^k=X_t^{P,k}+h_k(\alpha_t).
\]
By Assumption~\ref{HmoveAC}, the process $H_t^k:=h_k(\alpha_t)$ is absolutely continuous and has a derivative $\dot H_t^k$ such that
\[
|\dot H_t^k|\le L_h|\dot\alpha_t|
\quad dt\otimes P\text{-a.e.}
\]
The identity $\alpha_t=\alpha_0+\int_0^t\dot\alpha_sds$ implies
\[
\sup_{t\le T}|\alpha_t|^2
\le2|\alpha_0|^2+2T\int_0^T|\dot\alpha_s|^2ds.
\]
Moreover,
\[
T|\alpha_0|^2
\le2\int_0^T|\alpha_t|^2dt+2T^2\int_0^T|\dot\alpha_s|^2ds,
\]
obtained by writing $\alpha_0=\alpha_t-\int_0^t\dot\alpha_sds$ and integrating in $t$. The assumptions on $\alpha$ and $\dot\alpha$ therefore imply
\begin{equation}
\sup_P\E^P\sup_{t\le T}|\alpha_t|^2<\infty.
\label{eq:alpha-sup-occ}
\end{equation}

The semimartingale $S^k$ satisfies
\begin{align*}
S_t^k={}&S_0^k+\int_0^t\widetilde b_s^kds
+\int_0^t\sigma_{k\cdot}(s,X_{s-}^P,\mu_{s-}^P,\alpha_s)dB_s\\
&+\int_0^t\int_E\beta_k(s,X_{s-}^P,\mu_{s-}^P,\alpha_s,e)\widetilde N^P(ds,de),
\qquad t\in[0,T],\quad P\text{-a.s.},
\end{align*}
where
\[
\widetilde b_s^k=b_k(s,X_{s-}^P,\mu_{s-}^P,\alpha_s)+\dot H_s^k.
\]
The linear growth of $b$, \eqref{eq:Fsecond}, \eqref{eq:alpha-sup-occ}, and Cauchy--Schwarz give
\begin{equation}
\sup_{P,k}\E^P\int_0^T|\widetilde b_s^k|ds<\infty.
\label{eq:occ-drift-L1}
\end{equation}

Fix $0<r\le1$ and consider the convex $C^2$ function
\[
\psi_r(x)=\sqrt{x^2+r^2}.
\]
We have
\[
\psi_r'(x)=\frac{x}{\sqrt{x^2+r^2}},
\qquad |\psi_r'(x)|\le1,
\]
and
\[
\psi_r''(x)=\frac{r^2}{(x^2+r^2)^{3/2}}
\ge\frac{1}{2\sqrt2\,r}\one_{\{|x|\le r\}}.
\]
Applying It\^o's formula with jumps to $\psi_r(S_t^k)$ and writing $q_t^P=\sigma_t^Pa_t^P(\sigma_t^P)^\top$ gives
\begin{align*}
\psi_r(S_T^k)-\psi_r(S_0^k)
={}&\int_0^T\psi_r'(S_{t-}^k)\widetilde b_t^kdt
+\int_0^T\psi_r'(S_{t-}^k)\sigma_{k\cdot,t}dB_t\\
&+\frac12\int_0^T\psi_r''(S_{t-}^k)(q_t^P)_{kk}dt\\
&+\int_0^T\int_E[\psi_r(S_{t-}^k+\beta_{k,t}(e))-\psi_r(S_{t-}^k)]\widetilde N^P(dt,de)\\
&+\int_0^T\int_E R_r(S_{t-}^k,\beta_{k,t}(e))\Pi(de)dt,
\end{align*}
where
\[
R_r(x,z)=\psi_r(x+z)-\psi_r(x)-\psi_r'(x)z\ge0
\]
by convexity of $\psi_r$. For fixed $r>0$, we also have $0\le\psi_r''\le r^{-1}$; Taylor's formula with integral remainder therefore gives
\[
0\le R_r(x,z)
=z^2\int_0^1(1-\theta)\psi_r''(x+\theta z)\,d\theta
\le \frac{|z|^2}{2r}.
\]
Consequently,
\[
\E^P\int_0^T\int_E
R_r(S_{t-}^k,\beta_{k,t}(e))\Pi(de)dt
\le \frac{1}{2r}\E^P\int_0^T\|\beta_{k,t}\|_\Pi^2dt<\infty.
\]
Thus the compensator term is integrable. The two stochastic integrals are in fact square-integrable martingales, so no final localization is required. Indeed, $|\psi_r'|\le1$ and
\[
\E^P\int_0^T
|\psi_r'(S_{t-}^k)|^2
\bigl|\sigma_{k\cdot,t}\bigr|_{a_t^P}^2dt
\le
\E^P\int_0^T
\bigl|\sigma_{k\cdot,t}\bigr|_{a_t^P}^2dt<\infty.
\]
For the jump part, the $1$-Lipschitz property of $\psi_r$ gives
\[
|\psi_r(S_{t-}^k+\beta_{k,t}(e))-\psi_r(S_{t-}^k)|
\le|\beta_{k,t}(e)|,
\]
and hence
\[
\E^P\int_0^T\int_E
|\psi_r(S_{t-}^k+\beta_{k,t}(e))-\psi_r(S_{t-}^k)|^2
\Pi(de)dt
\le
\E^P\int_0^T\|\beta_{k,t}\|_\Pi^2dt<\infty.
\]
Their terminal expectations are therefore zero. Taking expectations and dropping the last, nonnegative term yields
\begin{align*}
\frac12\E^P\int_0^T\psi_r''(S_{t-}^k)(q_t^P)_{kk}dt
\le \E^P|\psi_r(S_T^k)-\psi_r(S_0^k)|
+\E^P\int_0^T|\widetilde b_t^k|dt.
\end{align*}
Since $\psi_r$ is $1$-Lipschitz,
\[
|\psi_r(S_T^k)-\psi_r(S_0^k)|\le|S_T^k-S_0^k|,
\]
and the right-hand side is uniformly bounded by \eqref{eq:Fsecond}, \eqref{eq:alpha-sup-occ}, and \eqref{eq:occ-drift-L1}. On $\{|S_{t-}^k|\le r\}$, Assumption~\ref{HellJ} and the lower bound on $\psi_r''$ give
\[
\frac12\psi_r''(S_{t-}^k)(q_t^P)_{kk}
\ge\frac{c_*}{4\sqrt2\,r}\one_{\{|S_{t-}^k|\le r\}}.
\]
Consequently,
\[
\sup_{P,k}\E^P\int_0^T\one_{\{|X_{t-}^{P,k}+h_k(\alpha_t)|\le r\}}dt
\le Cr.
\]
Finally,
\[
\{\rho(X_{t-}^P,\alpha_t)\le r\}
=\bigcup_{k=1}^d\{|X_{t-}^{P,k}+h_k(\alpha_t)|\le r\},
\]
so the union bound gives
\[
\E^P\int_0^T\one_{\{\rho(X_{t-}^P,\alpha_t)\le r\}}dt
\le\sum_{k=1}^d\E^P\int_0^T\one_{\{|X_{t-}^{P,k}+h_k(\alpha_t)|\le r\}}dt
\le Cdr.
\]
Absorbing $d$ into the constant proves \eqref{eq:occ-jumps-proved}. If Assumption~\ref{HqJ} holds and $0<\delta\le q-2$, the growth estimate \eqref{eq:A0growth}, H\"older's inequality, and \eqref{eq:Fq} provide the $2+\delta$ moment of the source; hence Assumption~\ref{HoccJ} holds with $\theta=1$.
\end{proof}

\begin{proposition}[Density criterion for deterministic moving thresholds --- ALA prototype]\label{prop:density}
Let $\alpha$ be deterministic. Assume that for every $P\in\Pcal$ and every $1\le k\le d$, the random variable $X_{t-}^{P,k}$ has a density $p_{t,k}^P$ such that
\[
\sup_{P\in\Pcal}\max_{1\le k\le d}
\int_0^T\|p_{t,k}^P\|_\infty dt\le C_{\rm dens}.
\]
Then \eqref{eq:occ} holds with $\theta=1$ and $C_{\rm occ}=2dC_{\rm dens}$.
\end{proposition}

\begin{proof}
Fix $P,t,k$. Since $\alpha_t$ is deterministic, the center $-h_k(\alpha_t)$ of the layer is deterministic. Using the density of $X_{t-}^{P,k}$,
\begin{align*}
P(|X_{t-}^{P,k}+h_k(\alpha_t)|\le r)
&=\int_{-h_k(\alpha_t)-r}^{-h_k(\alpha_t)+r}p_{t,k}^P(x)dx\\
&\le2r\|p_{t,k}^P\|_\infty.
\end{align*}
The event $\{\rho(X_{t-}^P,\alpha_t)\le r\}$ is the union of the $d$ coordinate events. By the union bound, Tonelli's theorem, and the density assumption,
\begin{align*}
\E^P\int_0^T\one_{\{\rho(X_{t-}^P,\alpha_t)\le r\}}dt
&\le\sum_{k=1}^d\int_0^TP(|X_{t-}^{P,k}+h_k(\alpha_t)|\le r)dt\\
&\le2r\sum_{k=1}^d\int_0^T\|p_{t,k}^P\|_\infty dt\\
&\le2drC_{\rm dens}.
\end{align*}
Since the bound is uniform in $P$, \eqref{eq:occ} holds with $\theta=1$ and $C_{\rm occ}=2dC_{\rm dens}$.
\end{proof}

Optimizing over the layer thickness $r$ gives the following rate.

\begin{proposition}[Yosida rate under shifted occupation --- ALA prototype]\label{prop:Yrate}
Let $\alpha\in\mathcal A^2$. Under Assumptions~\ref{HF}, \ref{HB}, \ref{HD}, and~\ref{HoccJ}, set
\begin{equation}
\gamma:=\min\left\{2,\frac{\theta\delta}{2+\delta+\theta}\right\}.
\label{eq:gamma}
\end{equation}
Then, for every $0<\lambda\le1$,
\begin{equation}
\sup_{P\in\Pcal}\E^P\int_0^T
|A^\lambda(X_{t-}^P,\mu_{t-}^P,\alpha_t)
-A^0(X_{t-}^P,\mu_{t-}^P,\alpha_t)|^2dt
\le C\lambda^\gamma.
\label{eq:Arate}
\end{equation}
In particular, the left-hand side of \eqref{eq:Yconv} is bounded by $C\lambda^\gamma$.
\end{proposition}

\begin{proof}
Fix $P,t$ and set
\[
x=X_{t-}^P,\quad \mu=\mu_{t-}^P,\quad a=\alpha_t,
\quad A_0=A^0(x,\mu,a),\quad D_\lambda=A^\lambda(x,\mu,a)-A_0.
\]
For $0<r\le1$, set
\[
G_{\lambda,r}=\{\rho(x,a)>r,\ \lambda|A_0|\le r/2\}.
\]
On $G_{\lambda,r}$, \eqref{eq:Yosida-displacement} and $|A^\lambda|\le|A_0|$ give
\[
|J_\lambda-x|\le\lambda|A_0|\le r/2.
\]
The two points $x+h(a)$ and $J_\lambda+h(a)$ remain in the same translated orthant. On that orthant the ALA operator is single-valued and its state-dependent part is $\alpha_A$-Lipschitz. Since $A^\lambda(x,\mu,a)=A(J_\lambda,\mu,a)$,
\begin{equation}
|D_\lambda|\le\alpha_A|J_\lambda-x|
\le\alpha_A\lambda|A_0|
\quad\text{on }G_{\lambda,r}.
\label{eq:Yrate-good}
\end{equation}
On the complement,
\begin{equation}
|D_\lambda|\le|A^\lambda|+|A_0|\le2|A_0|.
\label{eq:Yrate-bad}
\end{equation}
Integrating \eqref{eq:Yrate-good}--\eqref{eq:Yrate-bad} yields
\begin{align*}
\E^P\int_0^T|D_{\lambda,t}|^2dt
\le{}&C\lambda^2\E^P\int_0^T|A_t^0|^2dt
+C\E^P\int_0^T|A_t^0|^2\one_{G_{\lambda,r}^c}dt.
\end{align*}
The complement satisfies
\[
G_{\lambda,r}^c
\subset E_r\cup F_{\lambda,r},
\quad
E_r=\{\rho(X_{t-}^P,\alpha_t)\le r\},
\quad
F_{\lambda,r}=\{|A_t^0|>r/(2\lambda)\}.
\]
For $E_r$, H\"older's inequality gives
\begin{align*}
\E^P\int_0^T|A_t^0|^2\one_{E_r}dt
&\le\left(\E^P\int_0^T|A_t^0|^{2+\delta}dt\right)^{2/(2+\delta)}
\left((dt\otimes P)(E_r)\right)^{\delta/(2+\delta)}\\
&\le C r^{\theta\delta/(2+\delta)}.
\end{align*}
For $F_{\lambda,r}$,
\[
|A_t^0|^2\one_{F_{\lambda,r}}
\le\left(\frac{2\lambda}{r}\right)^\delta|A_t^0|^{2+\delta},
\]
so
\[
\E^P\int_0^T|A_t^0|^2\one_{F_{\lambda,r}}dt
\le C(\lambda/r)^\delta.
\]
Taking the supremum over $P$ gives
\begin{equation}
\sup_P\E^P\int_0^T|D_{\lambda,t}^P|^2dt
\le C\left\{\lambda^2+r^{\theta\delta/(2+\delta)}+(\lambda/r)^\delta\right\}.
\label{eq:Yrate-balance}
\end{equation}
Choose
\[
r=\lambda^{(2+\delta)/(2+\delta+\theta)}.
\]
Since $0<\lambda\le1$, we have $0<r\le1$. A direct calculation gives
\begin{align*}
r^{\theta\delta/(2+\delta)}
&=\lambda^{\theta\delta/(2+\delta+\theta)},\\
(\lambda/r)^\delta
&=\lambda^{\theta\delta/(2+\delta+\theta)}.
\end{align*}
Hence
\[
\sup_P\E^P\int_0^T|D_{\lambda,t}^P|^2dt
\le C\left\{\lambda^2+
\lambda^{\theta\delta/(2+\delta+\theta)}\right\}.
\]
For $0<\lambda\le1$, this is bounded by $C\lambda^\gamma$, with
\[
\gamma=\min\left\{2,\frac{\theta\delta}{2+\delta+\theta}\right\}.
\]
This proves \eqref{eq:Arate}. Finally, the proof of Theorem~\ref{thm:Yconv} gives quantitatively
\begin{align*}
&\sup_P\E^P\Bigl[\sup_{t\le T}|Y_t^{\lambda,P}-Y_t^P|^2
+\sup_{t\le T}|\mathsf K_t^{\lambda,P}-\mathsf K_t^P|^2\\
&\hspace{14mm}+\int_0^T|Z_t^{\lambda,P}-Z_t^P|_{a_t^P}^2dt
+\int_0^T\|U_t^{\lambda,P}-U_t^P\|_\Pi^2dt\Bigr]\\
&\qquad\le C\sup_P\E^P\int_0^T|D_{\lambda,t}^P|^2dt
\le C\lambda^\gamma.
\end{align*}
The same exponent controls the error of the regularized backward system.
\end{proof}

\section{Conclusion}

We studied a one-way coupled forward--backward McKean--Vlasov system with jumps in which the source generated by the maximal monotone operator
\[
A(x,\mu,a)=\partial\phi(x+h(a))+c(\mu)
\]
is evaluated along the forward state and accumulated through
\[
\mathsf K_t=\int_0^t\langle\varrho,A^0(X_{s-},\mu_{s-},\alpha_s)\rangle ds.
\]
Well-posedness and strong stability provide the basic analytical framework. The singular part of the analysis concerns the Yosida limit when the minimal-norm selection changes branch across an interface shifted by the control.

We establish strong Yosida convergence, Minty-type strong--weak closedness with variable law and interface, and a quantitative rate under an occupation condition. If the occupation of a layer of width $r$ is of order $r^\theta$ and $A^0$ has a $(2+\delta)$th moment, then
\begin{align*}
&\sup_P\E^P\int_0^T
|A^\lambda(X_{t-}^P,\mu_{t-}^P,\alpha_t)
-A^0(X_{t-}^P,\mu_{t-}^P,\alpha_t)|^2dt\\
&\hspace{35mm}\le C\lambda^\gamma,
\qquad
\gamma=\min\left\{2,\frac{\theta\delta}{2+\delta+\theta}\right\}.
\end{align*}
The same exponent controls $(Y^\lambda,Z^\lambda,U^\lambda,\mathsf K^\lambda)-(Y,Z,U,\mathsf K)$. For the ALA prototype, occupation is linear in the layer thickness when the continuous martingale part is nondegenerate in the normal directions and the interface moves absolutely continuously.

The framework remains modelwise and uses a common jump compensator. The explicit occupation estimate relies on a nondegenerate continuous martingale component, so the purely discontinuous regime requires a different mechanism. The system is also one-way coupled rather than fully coupled in both directions.

Natural extensions include occupation mechanisms for purely discontinuous dynamics, moving interfaces of lower regularity, genuinely fully coupled forward--backward systems, and quasi-sure formulations under nondominated families with varying jump characteristics. Particle approximations, propagation of chaos, and maximum principles are further directions and are not treated here.

\section*{Declarations}

\paragraph{Conflict of interest}
The authors declare that they have no conflict of interest.

\paragraph{Funding}
This research received no specific grant from funding agencies in the public, commercial, or not-for-profit sectors.

\paragraph{Declaration of generative AI and AI-assisted technologies in the manuscript preparation process}
During the preparation of this work, the authors used ChatGPT (OpenAI) to assist with English translation, language editing, and the organization and presentation of the manuscript. The authors subsequently reviewed and edited the content, verified the mathematical arguments and references, and take full responsibility for the content of the article.

\paragraph{Data availability}
This is a theoretical study and does not rely on any data set.

\end{document}